\documentclass[a4paper, 11pt]{article}

\usepackage[utf8]{inputenc}
\usepackage[T1]{fontenc}

\usepackage{amsthm}
\usepackage{amssymb}
\usepackage{mathtools}
\usepackage{stmaryrd}	\SetSymbolFont{stmry}{bold}{U}{stmry}{m}{n}	% To avoid warnings
\usepackage{bm}
\usepackage{mathrsfs}
\usepackage{euscript}
\usepackage{tikz-cd}
\usetikzlibrary{positioning, shapes.geometric, patterns, graphs, decorations.pathmorphing}

\usepackage{paralist}

\usepackage[colorlinks]{hyperref}
\hypersetup{
	linkcolor=blue,
	citecolor=red,
	urlcolor=teal,
	pdftitle = {Contragredient representations over local fields of equal characteristic},
	pdfauthor = {Wen-Wei Li}
}

\newcommand{\Z}{\ensuremath{\mathbb{Z}}}
\newcommand{\Q}{\ensuremath{\mathbb{Q}}}
\newcommand{\R}{\ensuremath{\mathbb{R}}}
\newcommand{\CC}{\ensuremath{\mathbb{C}}}
\newcommand{\A}{\ensuremath{\mathbb{A}}}
\newcommand{\F}{\ensuremath{\mathbb{F}}}

\newcommand{\Aut}{\operatorname{Aut}}

\newcommand{\Ind}{\operatorname{Ind}}

\newcommand{\Gal}{\operatorname{Gal}}
\newcommand{\Frob}{\operatorname{Frob}}
\newcommand{\Weil}[1]{\ensuremath{\mathrm{W}_{#1}}}	% The Weil group

\newcommand{\dd}{\mathop{}\!\mathrm{d}}
\newcommand{\Ccusp}{\ensuremath{C_c^{\mathrm{cusp}}}}	% Cuspidal functions

\newcommand{\lrangle}[1]{\ensuremath{\langle #1 \rangle}}

\newcommand{\identity}{\ensuremath{\mathrm{id}}}

\newcommand{\Hom}{\operatorname{Hom}}
\newcommand{\iHom}{\ensuremath{\mathrm{R}\mathcal{H}{om}}}
\newcommand{\End}{\operatorname{End}}
\newcommand{\rightiso}{\ensuremath{\stackrel{\sim}{\rightarrow}}}
\newcommand{\leftiso}{\ensuremath{\stackrel{\sim}{\leftarrow}}}

\newcommand{\Ker}{\operatorname{ker}}

\newcommand{\Image}{\operatorname{im}}
\newcommand{\dotimes}[1]{\ensuremath{\underset{#1}{\otimes}}}

\newcommand{\Lotimes}{\ensuremath{\overset{\mathsf{L}}{\otimes}}}

\newcommand{\Bun}{\operatorname{Bun}}
\newcommand{\Hecke}{\operatorname{Hecke}}
\newcommand{\Cht}{\operatorname{Cht}}
\newcommand{\Gra}{\operatorname{Gr}}

\newcommand{\Spec}{\operatorname{Spec}}
\newcommand{\Gm}{\ensuremath{\mathbb{G}_\mathrm{m}}}
\newcommand{\Ga}{\ensuremath{\mathbb{G}_\mathrm{a}}}

\newcommand{\dtimes}[1]{\ensuremath{\underset{#1}{\times}}}
\newcommand{\Supp}{\operatorname{Supp}}
\newcommand{\VerD}{\ensuremath{\mathbb{D}}}	% Verdier dual

\newcommand{\GL}{\operatorname{GL}}
\newcommand{\SO}{\operatorname{SO}}

\newcommand{\Sp}{\operatorname{Sp}}
\newcommand{\Lgrp}[1]{\ensuremath{{}^{\mathrm{L}} #1}}	% The L-group

\theoremstyle{plain}
\newtheorem{proposition}{Proposition}
\newtheorem{lemma}[proposition]{Lemma}
\newtheorem{theorem}[proposition]{Theorem}
\newtheorem{corollary}[proposition]{Corollary}
\theoremstyle{definition}
\newtheorem{definition}[proposition]{Definition}
\newtheorem{definition-theorem}[proposition]{Definition-Theorem}
\newtheorem{definition-proposition}[proposition]{Definition-Proposition}
\newtheorem{remark}[proposition]{Remark}

\theoremstyle{definition}
\newtheorem{conj}{Conjecture}
\newtheorem{rem}[conj]{Remark}
\theoremstyle{plain}
\newtheorem{thm}[conj]{Theorem}

\numberwithin{equation}{section}
\numberwithin{proposition}{subsection}
\numberwithin{conj}{section}	% In the Introduction only

\newcommand{\cate}[1]{\ensuremath{\mathsf{#1}}}	% Font series for categories
\newcommand{\dcate}[1]{\ensuremath{\text{-}\mathsf{#1}}}	% Categories with a dash on the left

\renewcommand{\emptyset}{\ensuremath{\varnothing}}	% Symbol for the emptyset

\usepackage{amsmath}
\usepackage[nottoc]{tocbibind}

\usepackage{geometry}
\usepackage{lmodern}

\title{Contragredient representations over local fields of positive characteristic}
\author{Wen-Wei Li}
\date{}
\begin{document}

\maketitle

\begin{abstract}
	It is conjectured by Adams--Vogan and Prasad that under the local Langlands correspondence, the $L$-parameter of the contragredient representation equals that of the original representation composed with the Chevalley involution of the $L$-group. We verify a variant of their prediction for all connected reductive groups over local fields of positive characteristic, in terms of the local Langlands parameterization of Genestier--Lafforgue. We deduce this from a global result for cuspidal automorphic representations over function fields, which is in turn based on a description of the transposes of V.\ Lafforgue's excursion operators.
\end{abstract}

{\scriptsize
\begin{tabular}{ll}
	\textbf{MSC (2010)} & 11F70; 11R58 22E55 \\
	\textbf{Keywords} & contragredient representation, function field, local Langlands conjecture
\end{tabular}}

\tableofcontents

\section{Introduction}
Let $F$ be a local field. Choose a separable closure $\overline{F}|F$ and let $\Weil{F}$ be the Weil group of $F$. For a connected reductive $F$-group $G$, the \emph{local Langlands conjecture} asserts the existence of a map
\[ \mathrm{LLC}: \Pi(G) \to \Phi(G) \]
where $\Pi(G)$ is the set of isomorphism classes of irreducible smooth representations $\pi$ of $G(F)$ (or Harish--Chandra modules when $F$ is Archimedean), and $\Phi(G)$ is the set of $\hat{G}$-conjugacy classes of $L$-parameters $\Weil{F} \xrightarrow{\phi} \Lgrp{G}$. Here the representations and the $L$-groups are taken over $\CC$, but we will soon switch to the setting of non-Archimedean $F$ and $\ell$-adic coefficients.

It is expected that the \emph{$L$-packets} $\Pi_\phi := \mathrm{LLC}^{-1}(\phi)$ are finite sets; if $\pi \in \Pi_\phi$, we say $\phi$ is the \emph{$L$-parameter} of $\pi$. The local Langlands correspondence also predicates on the internal structure of $\Pi_\phi$ when $\phi$ is a \emph{tempered} parameter; this requires additional structures as follows.
\begin{itemize}
	\item When $G$ is quasisplit, choose a Whittaker datum $\mathfrak{w} = (U,\chi)$ of $G$, taken up to $G(F)$-conjugacy, where $U \subset G$ is a maximal unipotent subgroup and $\chi$ is a generic smooth character of $U(F)$. The individual members of $\Pi_\phi$ are described in terms of
	\[ S_\phi := Z_{\hat{G}}(\Image(\phi)), \quad \mathcal{S}_\phi := \pi_0(S_\phi). \]
	Specifically, to each $\pi \in \Pi_\phi$ one should be able to attach an irreducible representation $\rho$ of the finite group $\mathcal{S}_\phi$ (up to isomorphism), such that a $\mathfrak{w}$-generic $\pi \in \Pi_\phi$ maps to $\rho = \mathbf{1}$.
	\item For non-split $G$, one needs to connect $G$ to a quasisplit group by means of a pure inner twist, or more generally a rigid inner twist \cite{Kal16a}; in parallel, the $L$-packets will extend across various inner forms of $G$. We refer to \cite[\S 5.4]{Kal16a} for a discussion in this generality.
\end{itemize}

One natural question is to describe various operations on $\Pi(G)$ in terms of $L$-parameters. Among them, we consider the \emph{contragredient} $\check{\pi}$ of $\pi$. The question is thus:
\begin{center}
	How is $\pi \mapsto \check{\pi}$ in $\Pi(G)$ reflected on $\Phi(G)$?
\end{center}
Despite its immediate appearance, this question has not been considered in this generality until the works of Adams--Vogan \cite[Conjecture 1.1]{AV16} and D.\ Prasad \cite[\S 4]{Pra18} independently. The answer hinges on the \emph{Chevalley involution} $\Lgrp{\theta}$ on $\Lgrp{G}$ to be reviewed in \S\ref{sec:Chevalley}.

\begin{conj}[Adams--Vogan--Prasad]\label{conj:AVP}
	Let $\pi$ be an irreducible smooth representation of $G(F)$.
	\begin{enumerate}
		\item If $\pi$ has $L$-parameter $\phi$, then $\check{\pi}$ has $L$-parameter $\Lgrp{\theta} \circ \phi$.
		\item Assume for simplicity that $G$ is quasisplit and fix a Whittaker datum $\mathfrak{w}$. If a tempered representation $\pi \in \Pi_\phi$ corresponds to an irreducible representation $\rho$ of $\mathcal{S}_\phi$, then $\check{\pi}$ corresponds to $(\rho \circ \Lgrp{\theta})^\vee$ tensored with a character $\xi$ of $\mathcal{S}_\phi$.
	\end{enumerate}
	To define $\xi$, we use the general recipe \cite[Lemma 4.1]{Kal13}:
	\[\begin{tikzcd}[row sep=small, column sep=small]
		\pi_0\left( S_\phi / Z_{\hat{G}}^{\Gal(\overline{F}|F)} \right) \arrow[r] & \mathrm{ker}\left[ \mathrm{H}^1(\mathrm{W}_F, Z_{\hat{G}^\mathrm{sc}}) \to \mathrm{H}^1(\mathrm{W}_F, Z_{\hat{G}}) \right] \arrow[twoheadrightarrow, d] \\
		\mathcal{S}_\phi \arrow[u] & \left( \dfrac{G^\mathrm{ad}(F)}{\Image[G(F) \to G^\mathrm{ad}(F)]} \right)^\text{Pontryagin dual}
	\end{tikzcd}\]
	Let $B$ be the Borel subgroup of $G$ included in the Whittaker datum, and choose a maximal torus $T \subset B$. Take the $\kappa \in T^\mathrm{ad}(F)$ acting as $-1$ on each $\mathfrak{g}_\alpha$ where $\alpha$ is any $B$-simple root. This furnishes the character $\xi$ of $\mathcal{S}_\phi$. When $G$ is not quasisplit, we have to endow it with a pure or rigid inner twist alluded to above.
\end{conj}

Conjecture \ref{conj:AVP} comprises two layers: the second one is due to \cite{Pra18}. In this article, we will focus exclusively on the first layer thereof.

An obvious precondition of the Adams--Vogan--Prasad conjecture is the existence of a map $\Pi(G) \to \Phi(G)$, baptized the \emph{Langlands parameterization}, which has been constructed for many groups in various ways.
\begin{itemize}
	\item When $F$ is Archimedean, the local Langlands correspondence is Langlands' paraphrase of Harish-Chandra's works. The ``first layer'' of the Adams--Vogan--Prasad conjecture is established by Adams--Vogan \cite{AV16}, and Kaletha \cite[Theorem 5.9]{Kal13} obtained the necessary refinement for the ``second layer''.
	\item When $F$ is non-Archimedean of characteristic zero and $G$ is a symplectic or quasisplit orthogonal group, Kaletha \cite[Theorem 5.9, Corollary 5.10]{Kal13} verified the Adams--Vogan--Prasad conjecture in terms of Arthur's \emph{endoscopic classification} of representations, which offers the local Langlands correspondence for these groups.
	\item For non-Archimedean $F$ and general $G$, Kaletha \cite[\S 6]{Kal13} also verified the conjecture for the depth-zero and epipelagic supercuspidal $L$-packets, constructed by DeBacker--Reeder and Kaletha using induction from open compact subgroups. 
\end{itemize}

The aim of this article is to address the first layer of Conjecture \ref{conj:AVP} when $F$ is a non-Archimedean local field of characteristic $p > 0$ and $G$ is arbitrary, in terms of the Langlands parameterization $\Pi(G) \to \Phi(G)$ of A.\ Genestier and V.\ Lafforgue \cite{GL17}. Their method is based on the geometry of the moduli stack of \emph{restricted chtoucas}, intimately related to the global Langlands parameterization of cuspidal automorphic representations by V.\ Lafforgue \cite{Laf18}. Accordingly, our representations $\pi$ will be realized on $\overline{\Q_\ell}$-vector spaces, where $\ell$ is a prime number not equal to $p$, and the $L$-group $\Lgrp{G}$ is viewed as a $\overline{\Q_\ell}$-group. As $\CC \simeq \overline{\Q_\ell}$ as abstract fields, passing to $\overline{\Q_\ell}$ does not alter the smooth representation theory of $G(F)$. On the other hand, there are subtle issues such as the independence of $\ell$ in the Langlands parameterization, which we refer to \cite[\S 12.2.4]{Laf18} for further discussions.

Our main local result is
\begin{thm}[Theorem \ref{prop:local-contragredient}]\label{prop:A}
	Let $F$ be a non-Archimedean local field of characteristic $p > 0$ and $G$ be a connected reductive $F$-group. Fix $\ell \neq p$ as above. If an irreducible smooth representation $\pi$ of $G(F)$ has parameter $\phi \in \Phi(G)$ under the Langlands parameterization of Genestier--Lafforgue, then $\check{\pi}$ has parameter $\Lgrp{\theta} \circ \phi$.
\end{thm}

\begin{rem}
	The prefix $L$ for local parameters and local packets is dropped for the following reason. The parameters of Genestier--Lafforgue are always \emph{semisimple} or completely reducible in the sense of Serre \cite{Se05}; in other words, the monodromy part of the Weil--Deligne parameter is trivial; see Lemma \ref{prop:semisimple-equiv}. As mentioned in \cite{GL17}, one expects that their parameter is the semi-simplification of the ``true'' $L$-parameter of $\pi$. Hence the packets $\Pi_\phi$ in question are larger than expected, and the Langlands parameterization we adopt is coarser, unless when $\phi$ does not factorize through any Levi $\Lgrp{M} \hookrightarrow \Lgrp{G}$, i.e.\ $\phi$ is \emph{semisimple and elliptic}.
\end{rem}

Our strategy is to reduce it into a global statement. Let $\mathring{F}$ be a global field of characteristic $p > 0$, say $\mathring{F} = \F_q(X)$ for a geometrically irreducible smooth proper $\F_q$-curve $X$, and set $\A = \A_{\mathring{F}}$. Let $G$ be a connected reductive $\mathring{F}$-group. Fix a level $N \subset X$, whence the corresponding congruence subgroup $K_N \subset G(\A)$ and the Hecke algebra $C_c(K_N \backslash G(\A)/K_N; \overline{\Q_\ell})$. Also fix a co-compact lattice $\Xi$ in $A_G(\mathring{F}) \backslash A_G(\A)$ where $A_G \subset G$ is the maximal central split torus. \textit{Grosso modo}, the global Langlands parameterization in \cite{Laf18} is deduced from a commutative $\overline{\Q_\ell}$-algebra $\mathcal{B}$ acting on the Hecke module
\[ H_{\{0\}, \mathbf{1}} := \Ccusp(\Bun_{G,N}(\F_q)/\Xi; \overline{\Q_\ell}) = \bigoplus_{\alpha \in \Ker^1(\mathring{F}, G)} \Ccusp\left( G_\alpha(\mathring{F}) \backslash G_\alpha(\A) / K_N \Xi; \overline{\Q_\ell}\right) \]
of $\overline{\Q_\ell}$-valued cusp forms, extended across pure inner forms indexed by $\Ker^1(\mathring{F}, G)$ (finite in number). The algebra $\mathcal{B}$ is generated by the \emph{excursion operators} $S_{I,f,\vec{\gamma}}$. For any character $\nu: \mathcal{B} \to \overline{\Q_\ell}$ of algebras, denote by $\mathfrak{H}_\nu$ the generalized $\nu$-eigenspace of $H_{\{0\}, \mathbf{1}}$. Then $H_{\{0\}, \mathbf{1}} = \bigoplus_\nu \mathfrak{H}_\nu$ as Hecke modules. Moreover, Lafforgue's machinery of $\Lgrp{G}$-pseudo-characters associates a semisimple $L$-parameter $\sigma: \Gal(\overline{\mathring{F}}|\mathring{F}) \to \Lgrp{G}(\overline{\Q_\ell})$ to $\nu$. In fact $\nu$ is determined by $\sigma$, so that we may write $\mathfrak{H}_\sigma = \mathfrak{H}_\nu$.

There is an evident Hecke-invariant bilinear form on $H_{\{0\}, \mathbf{1}}$, namely the \emph{integration pairing}
\[ \lrangle{h, h'} := \sum_{\alpha \in \Ker^1(\mathring{F}, G)} \int_{G_\alpha(\mathring{F}) \backslash G_\alpha(\A) / \Xi} hh', \quad h,h' \in H_{\{0\}, \mathbf{1}}, \]
with respect to some Haar measure on $G(\A) = G_\alpha(\A)$ which is $\Q$-valued on compact open subgroups. It is non-degenerate as easily seen by passing to $\overline{\Q_\ell} \simeq \CC$. Now comes our global theorem.
\begin{thm}[Theorem \ref{prop:global-contragredient}]\label{prop:B}
	If $\sigma, \sigma'$ are two semisimple $L$-parameters for $G$ such that $\lrangle{\cdot, \cdot}$ is nontrivial on $\mathfrak{H}_\sigma \otimes \mathfrak{H}_{\sigma'}$, then $\sigma' = \Lgrp{\theta} \circ \sigma$ up to $\hat{G}(\overline{\Q_\ell})$-conjugacy.
\end{thm}

Our local--global argument runs by first reducing Theorem \ref{prop:A} to the case that $\pi$ is integral supercuspidal such that $\omega_\pi$ has finite order when restricted to $A_G$; this step makes use of the compatibility of Langlands parameterization with parabolic induction, as established in \cite{GL17}. The second step is to globalize $\pi$ into a cuspidal automorphic representation $\mathring{\pi}$ with a suitable global model of $G$ and $\Xi$, satisfying $\mathring{\pi}^{K_N} \neq \{0\}$. The subspaces $\mathfrak{H}_\sigma$ of $H_{\{0\}, \mathbf{1}}$ might have isomorphic irreducible constituents in common, but upon modifying the automorphic realization, one can always assume that $\mathring{\pi}^{K_N}$ lands in some $\mathfrak{H}_\sigma$. An application of Theorem \ref{prop:B} and the local--global compatibility of Langlands parameterization \cite{GL17} will conclude the proof.

The proof of Theorem \ref{prop:B} relies upon the determination of the transpose $S \mapsto S^*$ of excursion operators with respect to $\lrangle{\cdot, \cdot}$, namely the Proposition \ref{prop:transpose-1}: 
\[ S_{I,f,\vec{\gamma}}^* = S_{I, f^\dagger, \vec{\gamma}^{-1}} \]
where $f \in \mathscr{O}(\hat{G} \bbslash (\Lgrp{G})^I \sslash \hat{G})$, the finite set $I$ and $\vec{\gamma} \in \Gal(\overline{\mathring{F}}|\mathring{F})^I$ are the data defining excursion operators, and $f^\dagger(\vec{g}) = f\left( \Lgrp{\theta}(\vec{g})^{-1}\right)$ for $\vec{g} \in (\Lgrp{G})^I$. This property entails that if $\nu: \mathcal{B} \to \overline{\Q_\ell}$ corresponds to $\sigma$, then $\nu^*: S \mapsto \nu(S^*)$ corresponds to $\Lgrp{\theta} \circ \sigma$ (Proposition \ref{prop:transpose-parameter}).

The starting point of the computation of transpose is the fact that $\lrangle{\cdot, \cdot}$ is of geometric origin: it stems from the Verdier duality on the moduli stack $\Cht^{(I_1, \ldots, I_k)}_{N,I}/\Xi$ of chtoucas. The Chevalley involution intervenes ultimately in describing the effect of Verdier duality in geometric Satake equivalence, which is in turn connected to  $\Cht^{(I_1, \ldots, I_k)}_{N,I}/\Xi$ via certain canonical smooth morphisms.

These geometric ingredients are already implicit in \cite{Laf18}. We just recast the relevant parts into our needs and supply some more details. In fact, the pairing $\lrangle{\cdot, \cdot}$ and its geometrization were used in a crucial way in older versions of \cite{Laf18}; that usage is now deprecated, and this article finds another application thereof.

Our third main result concerns the \emph{duality involution} proposed by Prasad in \cite[\S 3]{Pra18}. Assume that $G$ is quasisplit. Fix an additive character $\psi$ of $G$, an $F$-pinning $\mathcal{P}$ of $G$ and the corresponding Whittaker datum $\mathfrak{w}$; replacing $\psi$ by $\psi^{-1}$ yields another Whittaker datum $\mathfrak{w}'$. Prasad defined an involution $\iota_{G,\mathcal{P}}$ as the commuting product of the Chevalley involution $\theta = \theta_{\mathcal{P}}$ of $G$ and some inner automorphism $\iota_-$ which calibrates the Whittaker datum. Up to $G(F)$-conjugation, this recovers the MVW involutions on classical groups \cite[Chapitre 4]{MVW87} as well as the transpose-inverse on $\GL(n)$, whose relation with contragredient is well-known.

\begin{thm}[Theorem \ref{prop:duality-involution}]
	Let $\phi \in \Phi(G)$ be a semisimple parameter such that $\Pi_\phi$ contains a unique $\mathfrak{w}$-generic member $\pi$. Then $\Pi_{\Lgrp{\theta} \circ \phi}$ satisfies the same property with respect to $\mathfrak{w}'$, and $\check{\pi} \simeq \pi \circ \iota_{G, \mathcal{P}} \in \Pi_{\Lgrp{\theta} \circ \phi}$.
\end{thm}

Besides the crucial assumption which is expected to hold for tempered parameters if one works over $\CC$ with true $L$-packets (called Shahidi's tempered $L$-packet conjecture \cite{Sh90}), the main inputs are Theorem \ref{prop:A} and the local ``trivial functoriality'' applied to $\iota_{G,\mathcal{P}}$ (see \cite[Théorèmes 0.1 et 8.1]{GL17}). Due to these assumptions and the coarseness of our LLC, one should regard this result merely as some heuristic for Prasad's conjectures in \cite{Pra18}.

To conclude this introduction, let us mention two important issues that are left unanswered in this article.
\begin{itemize}
	\item As in \cite{Laf18, GL17}, these techniques can be generalized to some \emph{metaplectic coverings}, i.e.\ central extensions of locally compact groups
	\[ 1 \to \mu_m(F) \to \tilde{G} \to G(F) \to 1 \]
	where $\mu_m(R) = \left\{ z \in R^\times: z^m=1 \right\}$ as usual; it is customary to assume $\mu_m(F) = \mu_m(\overline{F})$ here. Fix a character $\zeta: \mu_m \hookrightarrow \overline{\Q_\ell}^\times$. One studies the irreducible smooth representations $\pi$ of $\tilde{G}$ that are $\zeta$-\emph{genuine}, i.e.\ $\pi(\varepsilon) = \zeta(\varepsilon) \cdot \identity$ for all $\varepsilon \in \mu_m(F)$. The most satisfactory setting for metaplectic coverings is due to Brylinski--Deligne \cite{BD01} that classifies the central extensions of $G$ by $\mathbf{K}_2$ as sheaves over $(\Spec F)_{\mathrm{Zar}}$. Taking $F$-points and pushing-out from $\mathrm{K}_2(F)$ by norm-residue symbols yields a central extension above.

	The $L$-group $\Lgrp{\tilde{G}}_\zeta$ associated to a Brylinski--Deligne $\mathbf{K}_2$-central extension, $m$ and $\zeta$ has been constructed in many situations; see the references in \cite[\S 14]{Laf18}. Now consider the metaplectic variant of Conjecture \ref{conj:AVP}. If $\pi$ is $\zeta$-genuine, $\check{\pi}$ will be $\zeta^{-1}$-genuine so one needs a canonical $L$-isomorphism $\Lgrp{\theta}: \Lgrp{\tilde{G}}_\zeta \to \Lgrp{\tilde{G}}_{\zeta^{-1}}$; this is further complicated by the fact that $\Lgrp{G}_\zeta$ is not necessarily a split extension of groups. Although some results seem within reach when $G$ is split, it seems more reasonable to work in the broader $\mathbf{K}_2$-setting and incorporate the framework of Gaitsgory--Lysenko \cite{GL18} for the geometric part. Nonetheless, this goes beyond the scope of the present article.
	\item With powerful tools from $p$-adic Hodge theory, Fargues and Scholze proposed a program to obtain a local Langlands parameterization in characteristic zero, akin to that of Genestier--Lafforgue; see \cite{Far16} for an overview. It would certainly be interesting to try to adopt our techniques to characteristic zero. However, our key tools are global adélic in nature, whilst the setting of Fargues--Scholze is global in a different sense (over the Fargues--Fontaine curve). This hinders a direct translation into the characteristic zero setting.
\end{itemize}

\subsection*{Organization of this article}
In \S\ref{sec:review}, we collect the basic backgrounds on cusp forms, the integration pairing $\lrangle{\cdot, \cdot}$, contragredient representations and $L$-parameters, all in the $\ell$-adic setting.

In \S\ref{sec:statement}, we begin by defining the Chevalley involution with respect to a chosen pinning and its extension to the $L$-group. Then we state the main Theorems \ref{prop:local-contragredient}, \ref{prop:global-contragredient} in the local and global cases, respectively. The local--global argument and the heuristic on duality involutions (Theorem \ref{prop:duality-involution}) are also given there.

We give a brief overview of some basic vocabularies of \cite{Laf18} in \S\ref{sec:overview}. The only purpose of this section is to fix notation and serves as a preparation of the next section. As in \cite{Laf18, GL17}, we allow non-split groups as well.

The transposes of excursion operators are described in \S\ref{sec:transposes}. It boils down to explicating the interplay between Verdier duality and partial Frobenius morphisms on the moduli stack of chtoucas. As mentioned before, a substantial part of this section can be viewed as annotations to \cite{Laf18}, together with a few new computations. The original approach in \S\ref{sec:transposes} in an earlier manuscript has been substantially simplified following suggestions of V.\ Lafforgue.

In \S\S\ref{sec:overview}---\ref{sec:transposes}, we will work exclusively in the global setting.

\subsection*{Acknowledgements}
The author is deeply grateful to Alain Genestier, Vincent Lafforgue, Dipendra Prasad and Changjian Su for their helpful comments, answers and corrections. Thanks also goes to the referees for pertinent suggestions.

\subsection*{Conventions}
\addcontentsline{toc}{subsection}{Conventions}
Throughout this article, we fix a prime number $\ell$ distinct from the characteristic $p > 0$ of the fields under consideration. We also fix an algebraic closure $\overline{\Q_\ell}$ of the field $\Q_\ell$ of $\ell$-adic numbers.

The six operations on $\ell$-adic complexes are those defined in \cite{LO1,LO2}, for algebraic stacks locally of finite type over a reasonable base scheme, for example over $\Spec\F_q$ where $q$ is some power of $p$. Given a morphism $f$ of finite type between such stacks, the symbols $f_!$, $f_*$, etc.\ will always stand for the functors between derived categories $\cate{D}^b_c(\cdots, E)$ unless otherwise specified, where the field of coefficients $E$ is some algebraic extension of $\Q_\ell$. The perverse $t$-structure on such stacks is defined in \cite{LO3}; further normalizations will be recalled in \S\ref{sec:geometric-setup}. The constant sheaf associated to $E$ on such a stack $\mathcal{X}$ is denoted by $E_{\mathcal{X}}$.

We use the notation $C_c(X; E)$ the indicate the space of compactly supported smooth $E$-valued functions on a topological space $X$, where $E$ is any ring. Since we work exclusively over totally disconnected locally compact spaces, smoothness here means locally constant.

For a local or global field $F$, we denote by $\Weil{F}$ the Weil group $F$ with respect to a choice of separable closure $\overline{F}|F$, and by $I_F \subset \Gal(\overline{F}|F)$ the inertia subgroup. The arithmetic Frobenius automorphism is denoted by $\Frob$. If $F$ is local non-Archimedean, $\mathfrak{o}_F$ will stand for its ring of integers.

If $\mathring{F}$ is a global field, we write $\A = \A_{\mathring{F}} := \prod'_v \mathring{F}_v$ for its ring of adèles, where $v$ ranges over the places of $\mathring{F}$. We also write $\mathfrak{o}_v = \mathfrak{o}_{\mathring{F}_v}$ in this setting.

For a scheme $T$, we write:
\begin{compactitem}
	\item $\Delta: T \hookrightarrow T^I$ for the diagonal morphism, where $I$ is any set;
	\item $\pi_1(T, \overline{t})$ for the étale fundamental group with respect to a geometric point $\overline{t} \to T$, when $T$ is connected, normal and locally Noetherian;
	\item $\mathscr{O}(T)$ for the ring of regular functions on $T$;
	\item $T_B := T \dtimes{\Spec A} \Spec B$ if $T$ is a scheme over $\Spec A$, and $B$ is a commutative $A$-algebra;
	\item $E(T) := \text{Frac}\,\mathscr{O}(T)$ for the function field, when $T$ is an irreducible variety over a field $E$.
\end{compactitem}
Suppose that $T$ is a variety over a field. The geometric invariant-theoretic quotient of $T$ under the right action of some group variety $Q$, if it exists, is written as $T \sslash Q$. Similar notation pertains to left or bilateral actions.

Let $G$ be a connected reductive group over a field $F$. For any $F$-algebra $A$, denote the group of $A$-points of $G$ by $G(A)$, endowed with a topology whenever $A$ is. Denote by $Z_G$, $G^\text{der}$, $G^\text{ad}$ for the center, derived subgroup and the the adjoint group of $G$, respectively. Normalizers (resp.\ centralizers) in $G$ are written as $N_G(\cdot)$ (resp.\ $Z_G(\cdot)$). If $T \subset G$ is a maximal torus, we write $T^\text{ad}$, etc.\ for the corresponding subgroups in $G ^\text{ad}$, etc. The character and cocharacter groups of a torus $T$ are denoted by $X^*(T)$ and $X_*(T)$ as $\Z$-modules, respectively.

The $L$-group (resp.\ Langlands dual group) of $G$ is denoted by $\Lgrp{G}$ (resp.\ $\hat{G}$). We use the Galois form of $L$-groups: details will be given in \S\ref{sec:L-parameters}.

For an affine algebraic group $H$ over some field $E$, the additive category of finite-dimensional algebraic representations of $H$ will be denoted as $\cate{Rep}_E(H)$. The trivial representation is denoted by $\mathbf{1}$. For any object $W \in \cate{Rep}_E(H)$, we write $\check{W}$ or $W^\vee$ for its contragredient representation on $\Hom_E(W, E)$. For any automorphism $\theta$ of $H$, write $W^\theta$ for the representation on $W$ such that every $h \in H$ acts by $w \mapsto \theta(h) \cdot w$.

The same notation $\check{\pi}$ applies to the contragredient of a smooth representation $\pi$ of a locally compact totally disconnected group. This will be the topic of \S\ref{sec:rep}. We denote the central character of an irreducible smooth representation $\pi$ as $\omega_\pi$.

\section{Review of representation theory}\label{sec:review}
\subsection{Cusp forms}\label{sec:cusp-forms}
Let $\mathring{F}$ be a global field of characteristic $p > 0$. We may write $\mathring{F} = \F_q(X)$ where $q$ is some power of $p$, and $X$ is a smooth, geometrically irreducible proper curve over $\F_q$. Denote $\A = \A_{\mathring{F}}$. Fix a closed subscheme $N \subset X$ which is finite over $\F_q$, known as the level.

Let $G$ be a connected reductive group over $\mathring{F}$. We associate to $N$ a compact open subgroup
\[ K_N := \Ker\left[ G\left(\prod_{v \in |X|} \mathfrak{o}_v \right) \to G(\mathscr{O}(N)) \right] \; \subset G(\A). \]
Denote the maximal split central torus in $G$ by $A_G$. It is also known that there is a co-compact lattice
\[ \Xi \subset A_G(\mathring{F}) \backslash A_G(\A), \]
which we fix once and for all. The space $G(\mathring{F}) \backslash G(\A) /\Xi$ is known to have finite volume with respect to any Haar measure on $G(\A)$.

In what follows, we use a Haar measure on $G(\A)$ such that $\mathrm{mes}(K) \in \Q$ for any compact open subgroup $K$. The existence of such measures is established in \cite[Théorème 2.4]{Vig96}. The same convention pertains to the subgroups of $G$.

For all subextension $E|\Q_\ell$ of $\overline{\Q_\ell}|\Q_\ell$, we have the space
\[ C_c(G(\mathring{F}) \backslash G(\A) /\Xi; E) = \bigcup_{N: \text{levels}} C_c(G(\mathring{F}) \backslash G(\A)/K_N \Xi; E) \]
of smooth $E$-valued functions of compact support on $G(\mathring{F}) \backslash G(\A)/ \Xi$. Then $G(\A)$ acts on the left of $C_c(G(\mathring{F}) \backslash G(\A) /\Xi; E)$ by $(gf)(x) = f(xg)$. Accordingly, the space of $K_N$-invariants $C_c(G(\mathring{F}) \backslash G(\A)/K_N \Xi; E)$ is a left module under the unital $E$-algebra $C_c(K_N \backslash G(\A) / K_N; E)$, the Hecke algebra under convolution $\star$.

Our convention on Haar measures means that we can integrate $E$-valued smooth functions on $G(\mathring{F}) \backslash G(\A) /\Xi$, etc.

The subspace of $C_c(G(\mathring{F}) \backslash G(\A) /\Xi; E)$ of cuspidal functions
\[ \Ccusp(G(\mathring{F}) \backslash G(\A) / \Xi; E) = \bigcup_{N: \text{levels}} \Ccusp(G(\mathring{F}) \backslash G(\A) /K_N \Xi; E) \]
is defined by either
\begin{compactitem}
	\item requiring that the constant terms $f_P(x) = \int_{U(\mathring{F}) \backslash U(\A)} f(ux) \dd u$ are zero whenever $P = MU \subsetneq G$ is a parabolic subgroup;
	\item or using the criterion in terms of Hecke-finiteness in \cite[Proposition 8.23]{Laf18}.
\end{compactitem} 

We record two more basic facts.
\begin{itemize}
	\item The $E$-vector space $\Ccusp(G(\mathring{F}) \backslash G(\A) /K_N \Xi; E)$ is finite-dimensional. This result is originally due to Harder, and can be deduced from the uniform bound on supports of such functions in \cite[I.2.9]{MW94}.

	\item As a smooth $G(\A)$-representation, $\Ccusp(G(\mathring{F}) \backslash G(\A) /\Xi; E)$ is absolutely semisimple, i.e.\ it is semisimple after $- \otimes_E \overline{\Q_\ell}$; see \cite[VIII.226]{BouAlg-8}. Indeed, the semisimplicity in the case $E = \overline{\Q_\ell} \simeq \CC$ is well-known.
\end{itemize}
In parallel, $\Ccusp(G(\mathring{F}) \backslash G(\A)/ K_N \Xi; E)$ is also absolutely semisimple as a $C_c(K_N \backslash G(\A)/K_N; E)$-module. Recall the module structure: $f \in C_c(K_N \backslash G(\A)/K_N; E)$ acts on $h$ as
\begin{equation}\label{eqn:Hecke-alg-convolution}
	(f \cdot h)(x) := \int_{K_N \backslash G(\A) / K_N} h(xg) f(g) \dd g = (h \star \check{f})(x), \quad x \in G(\A)
\end{equation}
where $\check{f}(g) = f(g^{-1})$ and the convolution $\star$ is defined in the usual manner.

We record the following standard result for later use.
\begin{proposition}\label{prop:K_N-invariants}
	For every $G(\A)$-representation $\mathring{\pi}$, assumed to be smooth, let $\mathring{\pi}^{K_N}$ be the space of $K_N$-invariant vectors. It is a left module under $C_c(K_N \backslash G(\A)/K_N; \overline{\Q_\ell})$.
	\begin{enumerate}[(i)]
		\item For all irreducible $G(\A)$-representations $\mathring{\pi}_1, \mathring{\pi}_2$ generated by $K_N$-invariants, we have $\mathring{\pi}_1 \simeq \mathring{\pi}_2 \iff \mathring{\pi}_1^{K_N} \simeq \mathring{\pi}_2^{K_N}$ as simple $C_c(K_N \backslash G(\A)/K_N; \overline{\Q_\ell})$-modules.
		\item Given any decomposition $\Ccusp(G(\mathring{F}) \backslash G(\A) /\Xi; \overline{\Q_\ell}) = \bigoplus_{\mathring{\pi} \in \Pi} \mathring{\pi}$ into irreducibles, where $\Pi$ is a set (with multiplicities) of irreducible subrepresentations, we have
		\[ \Ccusp(G(\mathring{F}) \backslash G(\A)/ K_N \Xi; \overline{\Q_\ell}) = \bigoplus_{\substack{\mathring{\pi} \in \Pi \\ \mathring{\pi}^{K_N} \neq 0 }} \mathring{\pi}^{K_N} \]
		in which each $\mathring{\pi}^{K_N}$ is simple.
		\item For every irreducible $G(\A)$-representation $\mathring{\pi}$ generated by $K_N$-invariants, we have a natural isomorphism of multiplicity spaces
		\begin{multline*}
			\Hom_{G(\A)\dcate{Rep}}\left( \mathring{\pi}, \Ccusp(G(\mathring{F}) \backslash G(\A) /\Xi; \overline{\Q_\ell}) \right) \rightiso \\
			\Hom_{C_c(K_N \backslash G(\A)/K_N; \overline{\Q_\ell})\dcate{Mod} }\left( \mathring{\pi}^{K_N}, \Ccusp(G(\mathring{F}) \backslash G(\A)/ K_N \Xi; \overline{\Q_\ell}) \right).
		\end{multline*}
	\end{enumerate}
	The property (i) actually holds for representations of $G(\mathring{F}_v)$ and of its Hecke algebras, for any place $v$ of $\mathring{F}$.
	
	The $\Ccusp$ in (ii), (iii) can be replaced by $\bigoplus_{\alpha \in \Ker^1(\mathring{F}, G)} \Ccusp(G_\alpha(\mathring{F}) \backslash G(\A) /\Xi; \overline{\Q_\ell})$; see \eqref{eqn:alpha}.
\end{proposition}
\begin{proof}
	By semisimplicity, $\Ccusp(G(\mathring{F}) \backslash G(\A) /\Xi; \overline{\Q_\ell})$ (or the $\bigoplus_\alpha$ version) decomposes uniquely into $W \oplus W'$ such that
	\begin{compactitem}
		\item $W$ is a subrepresentation isomorphic to a direct sum of irreducibles, each summand is generated by $K_N$-invariants;
		\item $W'$ is a subrepresentation satisfying $(W')^{K_N} = \{0\}$.
	\end{compactitem}
	For (ii)---(iii), it suffices to look at the $G(\A)$-representation $W$ and the $C_c(K_N \backslash G(\A)/K_N; \overline{\Q_\ell})$-module $W^{K_N}$; both are semisimple. The required assertions follow from the standard equivalences between categories in \cite[I.3 and III.1.5]{Ren10} and Schur's Lemma \cite[III.1.8 and B.II]{Ren10}.
\end{proof}

Next, we introduce the moduli stack $\Bun_{G,N}$ over $\F_q$ of $G$-torsors on $X$ with level $N$ structures: it maps any $\F_q$-scheme $S$ to the groupoid
\begin{align*}
	\Bun_{G,N}(S) & = \left\{ \begin{array}{r|l}
		(\mathcal{G}, \psi) & \mathcal{G}: G\text{-torsor over } X \times S \\
		& \psi: \mathcal{G}|_{N \times S} \rightiso G|_{N \times S} \\
		& \quad \text{a trivialization over $N$}
	\end{array}\right\}, \\
	\Bun_G & := \Bun_{G, \emptyset}.
\end{align*}
For this purpose, we need suitable models of $G$ over $X$. Let $U \subset X$ be the maximal open subscheme such that $G$ extends to a connected reductive $U$-group scheme. We follow \cite[\S 12.1]{Laf18} to take parahoric models at the formal neighborhoods of all points of $X \smallsetminus U$. Glue these parahoric models with the smooth model over $U$, à la Beauville--Laszlo, to yield a smooth affine $X$-group scheme with geometrically connected fibers, known as a \emph{Bruhat--Tits group scheme} over $X$; see also \cite[\S 1]{He10}.

Regard $\Bun_{G,N}(\F_q)$ as a set, on which $\Xi$ acts naturally. As explained in \cite{Laf18}, we have
\begin{equation}\label{eqn:alpha}
	\Bun_{G,N}(\F_q) = \bigsqcup_{\alpha \in \Ker^1(\mathring{F}, G)} G_\alpha(\mathring{F}) \backslash G_\alpha(\A) / K_N
\end{equation}
where
\begin{compactitem}
	\item $\Ker^1(\mathring{F}, G)$ is the kernel of $\mathrm{H}^1(\mathring{F}, G) \to \prod_{v \in |X|} \mathrm{H}^1(\mathring{F}_v, G)$;
	\item to each $\alpha \in \Ker^1(\mathring{F}, G)$ is attached a locally trivial pure inner twist $G_\alpha$ of $G$, and we fix an identification $G_\alpha(\A) \simeq G(\A)$.
\end{compactitem}
The decomposition is compatible with $\Xi$-actions. The pointed set $\Ker^1(\mathring{F}, G)$ is finite; it is actually trivial when $G$ is split. As before, we have the spaces
\[ \Ccusp\left( \Bun_{G,N}(\F_q)/\Xi; E \right) = \bigoplus_{\alpha \in \Ker^1(\mathring{F}, G)} \Ccusp\left( G_\alpha(\mathring{F}) \backslash G_\alpha(\A) / K_N \Xi; E \right). \]
The cuspidality on the left-hand side can be defined in terms of Hecke-finiteness as before. We shall also use compatible Haar measures on various $G_\alpha(\A)$.

From the viewpoint of harmonic analysis, the mere effect of working with $\Bun_{G,N}(\F_q)$ is to consider all the inner twists from $\Ker^1(\mathring{F}, G)$ at once. See also \cite[\S 12.2.5]{Laf18}.

\subsection{Integration pairing}\label{sec:integration-pairing}
Let $E$ be a subextension of $\overline{\Q_\ell}|\Q_\ell$.

\begin{definition}\label{def:integration-pairing}
	With the Haar measures as in \S\ref{sec:cusp-forms}, we define the \emph{integration pairing}
	\begin{gather*}
		\lrangle{\cdot, \cdot}: \Ccusp(G(\mathring{F}) \backslash G(\A) /\Xi; E) \dotimes{E} \Ccusp(G(\mathring{F}) \backslash G(\A) / \Xi; E) \longrightarrow E \\
		h \otimes h' \longmapsto \lrangle{h, h'} := \int_{G(\mathring{F}) \backslash G(\A) / \Xi} hh'.
	\end{gather*}
\end{definition}
The pairing is clearly $E$-bilinear, symmetric and $G(\A)$-invariant. There is an obvious variant for not necessarily cuspidal functions.

\begin{lemma}
	The pairing $\lrangle{\cdot, \cdot}$ above is absolutely non-degenerate, i.e.\ its radical equals $\{0\}$ after $- \otimes_E \overline{\Q_\ell}$.
\end{lemma}
\begin{proof}
	It is legitimate to assume $E = \overline{\Q_\ell}$, and there exists an isomorphism of fields $\overline{\Q_\ell} \simeq \CC$. The non-degeneracy over $\CC$ is well-known: we have $\int h\overline{h} \geq 0$, and equality holds if and only if $h=0$.
\end{proof}

\begin{remark}\label{rem:pairing-level}
	For a chosen level $N \subset X$, we have an analogous pairing
	\begin{gather*}
		\lrangle{\cdot, \cdot}: \Ccusp(G(\mathring{F}) \backslash G(\A) /K_N \Xi; E) \dotimes{E} \Ccusp(G(\mathring{F}) \backslash G(\A) / K_N \Xi; E) \longrightarrow E \\
		h \otimes h' \longmapsto \lrangle{h, h'} := \int_{G(\mathring{F}) \backslash G(\A) /K_N \Xi} hh'.
	\end{gather*}
	The integration here is actually a ``stacky'' sum over $G(\mathring{F}) \backslash G(\A) /K_N \Xi$, i.e.\ $\lrangle{h,h'}$ equals that of Definition \ref{def:integration-pairing} if one starts with a Haar measure on $G(\A)$ with $\mathrm{mes}(K_N)=1$. It is also $E$-bilinear, symmetric, absolutely non-degenerate and invariant in the sense that
	\[ \lrangle{f \cdot h, \; h'} = \lrangle{h, \; \check{f} \cdot h'}, \quad f \in C_c(K_N \backslash G(\A) /K_N; E), \]
	see \eqref{eqn:Hecke-alg-convolution}. There is an obvious variant for not necessarily cuspidal functions.
\end{remark}

The spaces in question being finite-dimensional, it makes sense to talk about the \emph{transpose} of a linear operator. For example, the transpose of the left multiplication by $f$ is given by that of $\check{f}$.

As in \S\ref{sec:cusp-forms}, the integration pairing extends to
\begin{gather*}
	\lrangle{\cdot, \cdot}: \Ccusp\left( \Bun_{G,N}(\F_q)/\Xi; E \right) \dotimes{E} \Ccusp\left( \Bun_{G,N}(\F_q)/\Xi; E \right) \longrightarrow E \\
	h \otimes h' \longmapsto \int_{\Bun_{G,N}(\F_q)/\Xi} hh'.
\end{gather*}
This is the orthogonal sum of the integrations pairings on various $G_\alpha(\A)$.

\subsection{Representations}\label{sec:rep}
In this subsection, we let $F$ be a local field of characteristic $p > 0$. Denote the cardinality of the residue field of $F$ as $q$. Let $G$ be a connected reductive $F$-group. The smooth representations of $G(F)$ will always be realized on $\overline{\Q_\ell}$-vector spaces. Irreducible smooth representation of $G(F)$ are admissible; see \cite[VI.2.2]{Ren10}.

The smooth characters of $G(F)$ are homomorphisms $G(F) \to \overline{\Q_\ell}^\times$ with open kernel. We will need to look into a class of particularly simple characters, namely those trivial on the open subgroup
\begin{equation}\label{eqn:G(F)1}
	G(F)^1 := \bigcap_{\chi \in X^*(G)} \Ker|\chi|_F
\end{equation}
of $G(F)$, where $X^*(G) := \Hom_{\text{alg.\ grp}}(G, \Gm)$ and
\[ |\cdot|_F: F^\times \twoheadrightarrow q^\Z \subset \Q^\times \]
is the normalized absolute value on $F$. Note that $G(F)/G(F)^1 \simeq \Z^r$ with $r := \mathrm{rk}_\Z X^*(G)$. Moreover, $G(F)^1 \supset G^\mathrm{der}(F) = G^\mathrm{der}(F)^1$.

For any smooth character $\omega$ of $Z_G(F)$, denote by $C_c(G(F), \omega)$ the space of functions $f: G(F) \to \overline{\Q_\ell}$ such that $f(zg) = \omega(z)f(g)$ for all $z \in Z_G(F)$ and $\Supp(f)$ is compact modulo $Z_G(F)$.

Let $\Ind^G_P(\cdot)$ denote the unnormalized parabolic induction from the Levi quotient of $P \subset G$. Let $\delta_P$ denote the modulus character of $P(F)$ taking values in $q^\Z$. Upon choosing $q^{1/2} \in \overline{\Q_\ell}$, we can also form the \emph{normalized parabolic induction} $I^G_P(\cdot) := \Ind^G_P\left( \cdot \otimes \delta_P^{1/2} \right)$.

We need the notion \cite[VI.7.1]{Ren10} of the \emph{cuspidal support} $(M,\tau)$ of an irreducible smooth representation $\pi$. Here $M \subset G$ is a Levi subgroup and $\tau$ is a supercuspidal irreducible representation of $M(F)$, such that $\pi$ is a subquotient of $I^G_P(\tau)$ for any parabolic subgroup $P \subset G$ with Levi component $M$. The cuspidal support is unique up to $G(F)$-conjugacy. It is known that one can choose $P$ with Levi component $M$ such that $\pi \hookrightarrow I^G_P(\tau)$. See \cite[VI.5.4]{Ren10}.

We collect below a few properties of an irreducible smooth representation $\pi$ of $G(F)$.

\begin{enumerate}
	\item Suppose that $\pi$ is supercuspidal. There exists a finite extension $E$ of $\Q_\ell$ such that $\pi$ is defined over $E$. Indeed, since the central character $\omega_\pi$ can be defined over some finite extension of $\Q_\ell$, so is $\pi \hookrightarrow C_c(G(F), \omega_\pi)$.

	From this and the discussion on cuspidal supports, it follows that every $\pi$ can be defined over some finite extension $E$ of $\Q_\ell$.

	\item We say $\pi$ is \emph{integral} if it admits an $\mathfrak{o}_E$-model of finite type, where $E$ is a finite extension of $\Q_\ell$. See \cite[\S 1.4]{Vig01} for details. Then an irreducible supercuspidal $\pi$ is integral if and only if $\omega_\pi$ has $\ell$-adically bounded image in $\overline{\Q_\ell}^\times$.
	
	Again, this is a consequence of $\pi \hookrightarrow C_c(G(F), \omega_\pi)$. It also implies the notion of integrality stated in the beginning of \cite{GL17}.

	\item Let $V$ be the underlying vector space of $\pi$. The \emph{contragredient representation} $\check{\pi}$ of a smooth representation  $\pi$ is realized on the space $V^\vee$ of the smooth vectors in $\Hom_{\overline{\Q_\ell}}(V, \overline{\Q_\ell})$. It satisfies $\lrangle{\check{\rho}(g)\check{v}, v} = \lrangle{\check{v}, \rho(g^{-1})v}$. If $\pi$ is defined over $E$, so is $\check{\pi}$. Taking contragredient preserves irreducibility and supercuspidality. It is clear that $(\pi \otimes \chi)^\vee = \check{\pi} \otimes \chi^{-1}$ for any smooth character $\chi: G(F) \to \overline{\Q_\ell}^\times$.
	
	\item Moreover, $(\pi)^{\vee\vee} \simeq \pi$ for all smooth irreducible $\pi$; see \cite[III.1.7]{Ren10}. Also, $\omega_{\check{\pi}} = \omega_\pi^{-1}$
\end{enumerate}

\begin{proposition}\label{prop:sc-support}
	If $\pi$ is an irreducible smooth representation of $G(F)$ with cuspidal support $(M, \tau)$, then $\check{\pi}$ has cuspidal support $(M, \check{\tau})$.
\end{proposition}
\begin{proof}
	Choose a parabolic subgroup $P \subset G$ with $M$ as Levi component such that $\pi \hookrightarrow I^G_P(\tau)$. Once the Haar measures are chosen, we have $I^G_P(\tau)^\vee \simeq I^G_P(\check{\tau})$ canonically; see \cite[\S 3.5]{BH06}. Dualizing, we deduce $I^G_P(\check{\tau}) \twoheadrightarrow \check{\pi}$. Thus $\check{\pi}$ is a subquotient of $I^G_P(\check{\tau})$.
\end{proof}

\subsection{\texorpdfstring{$L$}{L}-parameters}\label{sec:L-parameters}
Let $F$ be a local or global field of characteristic $p > 0$. For a connected reductive $F$-group $G$, we denote by $\tilde{F}|F$ the splitting field of $G$, which is a finite Galois extension inside a chosen separable closure $\bar{F}$.

Denote by $\Weil{F}$ the absolute Weil group of $F$. It comes with canonical continuous homomorphisms
\begin{inparaenum}[(i)]
	\item $\Weil{F} \to \Gal(\bar{F}|F)$, and
	\item $\Weil{F_v} \to \Weil{F}$ if $F$ is global and $v$ is a place of $F$. For (ii) we choose an embedding $\overline{F} \hookrightarrow \overline{F_v}$ of separable closures.
\end{inparaenum}

\begin{definition}
	The \emph{Langlands dual group} $\hat{G}$ of $G$ is a pinned connected reductive $\Q_\ell$-group (in fact, definable over $\Z$), on which $\Gal(\tilde{F}|F)$ operates by pinned automorphisms. Throughout this article, we use the \emph{finite Galois forms} of the $L$-group of $G$, namely
	\[ \Lgrp{G} := \hat{G} \rtimes \Gal(\tilde{F}|F) \]
	viewed as an affine algebraic group.
%	By a standard abuse of notation, we will often identify $\Lgrp{G}$ (resp.\ $\hat{G}$) with $\Lgrp{G}(\overline{\Q_\ell}) = \hat{G}(\overline{\Q_\ell}) \rtimes \Gal(\tilde{F}|F)$ (resp.\ $\hat{G}(\overline{\Q_\ell})$), unless otherwise specified.
\end{definition}

If $M \hookrightarrow G$ is a Levi subgroup, we obtain a corresponding embedding $\Lgrp{M} \to \Lgrp{G}$ of standard Levi subgroup.

\begin{definition}
	An $L$-parameter for $G$ is a homomorphism $\sigma: \Weil{F} \to \Lgrp{G}(\overline{\Q_\ell})$ such that
	\begin{itemize}
		\item the diagram
		$\begin{tikzcd}[row sep=small, column sep=tiny]
			\Weil{F} \arrow[rr, "\sigma"] \arrow[rd] & & \Lgrp{G} \arrow[ld] \\
			& \Gal(\tilde{F}|F) &
		\end{tikzcd}$ commutes;
		\item $\sigma$ is continuous with respect to the $\ell$-adic topology on $\Lgrp{G}(\overline{\Q_\ell})$;
		\item $\sigma$ is \emph{relevant} in the sense of \cite[\S 8.2]{Bo79}, which matters only when $G$ is not quasisplit;
		\item (the local case) $\sigma$ is Frobenius semisimple: $\rho(\sigma(\Frob))$ is semisimple for every algebraic representation $\rho: \Lgrp{G}(\overline{\Q_\ell}) \to \GL(N, \overline{\Q_\ell})$, where $\Frob$ stands for any Frobenius element in $\Weil{F}$ (see \cite[32.7 Proposition]{BH06} for more discussions on Frobenius-semisimplicity);
		\item (the global case) $\sigma$ is semisimple in the sense of \cite{Se05}, to be described below. We do not require Frobenius-semisimplicity here because for $\ell$-adic representations of geometric origin, that property is a long-standing conjecture in étale cohomology. 
	\end{itemize}
	The set of $\hat{G}(\overline{\Q_\ell})$-conjugacy classes of $L$-parameters is denoted as $\Phi(G)$. By \cite[\S 3.4]{Bo79}, there is a natural map $\Phi(M) \to \Phi(G)$ for any Levi subgroup $M$.
\end{definition}

\begin{remark}
	Since $\Lgrp{G}(\overline{\Q_\ell})$ carries the $\ell$-adic topology and $\sigma$ is required to be continuous, when $F$ is local we get rid of the Weil--Deligne group in the usual formulation in terms of $\Lgrp{G}(\CC)$. Besides, we do not consider Arthur parameters in this article.
\end{remark}

As recalled earlier, the structure of Weil groups allows us to
\begin{itemize}
	\item localize a global $L$-parameter at a place $v$;
	\item talk about $L$-parameters of the form $\Gal(\bar{F}|F) \to \Lgrp{G}(\overline{\Q_\ell})$ and their localizations when $F$ is global.
\end{itemize}

Next, we recall the \emph{semisimplicity} of $L$-parameters following \cite{Laf18, Se05}: a continuous homomorphism $\sigma: \Weil{F} \to \Lgrp{G}(\overline{\Q_\ell})$ is called semisimple if the Zariski closure of $\Image(\sigma)$ is reductive in $\Lgrp{G}(\overline{\Q_\ell})$, in the sense that its identity connected component is reductive. When $G$ is split, this is exactly the definition of complete reducibility in \cite[3.2.1]{Se05}, say by applying \cite[Proposition 4.2]{Se05}.

\begin{lemma}\label{prop:semisimple-equiv}
	Assume $F$ is local. The following are equivalent for any $L$-parameter $\sigma$ for $G$:
	\begin{enumerate}[(i)]
		\item $\sigma$ is semisimple;
		\item the Weil--Deligne parameter associated to $\sigma$ has trivial nilpotent part.
	\end{enumerate}
\end{lemma}
\begin{proof}
	By composing $\sigma$ with any faithful algebraic representation $\rho: \Lgrp{G}(\overline{\Q_\ell}) \hookrightarrow \GL(N, \overline{\Q_\ell})$, we may assume that $\sigma$ is an $\ell$-adic representation $\Weil{F} \to \GL(N, \overline{\Q_\ell})$. To $\sigma$ is associated the Weil--Deligne representation $\mathrm{WD}(\sigma)$: it comes with a nilpotent operator $\mathfrak{n}$. For details, see \cite[32.5]{BH06}.

	(i) $\implies$ (ii): The line $\overline{\Q_\ell} \mathfrak{n}$ is preserved by $\Image(\sigma)$-conjugation. Since $\exp(t\mathfrak{n}) \in \Image(\sigma)$ for $t \in \Z_\ell$ with $|t| \ll 1$, the semisimplicity of $\sigma$ forces $\mathfrak{n} = 0$.

	(ii) $\implies$ (i). As $\mathfrak{n} = 0$, the smooth representation underlying $\mathrm{WD}(\sigma)$ is just $\sigma$, hence $\sigma$ is semisimple as a smooth representation of $\Weil{F}$ by \cite[32.7 Theorem]{BH06}. The reductivity (or complete reducibility) of the Zariski closure of $\Image(\sigma)$ then follows from the theory in \cite{Se05}.
\end{proof}

Finally, we define parabolic subgroups of $\Lgrp{G}$ as in \cite[3.2]{Bo79}. They are subgroups of the form $N_{\Lgrp{G}}(\hat{P})$ where $\hat{P} \subset \hat{G}$ is a parabolic subgroups, and whose projection to $\Gal(\tilde{F}|F)$ has full image. Define the unipotent radical of such a parabolic subgroup to be that of $\hat{P}$. We still have the notion of Levi decomposition in this setting; see \cite[3.4]{Bo79}.

Following \cite[\S 13]{Laf18}, the \emph{semi-simplification} $\sigma^{\mathrm{ss}}$ of an $L$-parameter $\sigma$ is defined as follows
\begin{compactitem}
	\item first, take the smallest parabolic subgroup $\Lgrp{P} \subset \Lgrp{G}$ containing $\Image(\sigma)$;
	\item project to the Levi quotient;
	\item then embed back into $\Lgrp{G}$ using some Levi decomposition.
\end{compactitem}
The resulting parameter is well-defined up to $\hat{G}(\overline{\Q_\ell})$-conjugacy.

By definition, an $L$-homomorphism $\Lgrp{H} \to \Lgrp{G}$ between $L$-groups is an algebraic homomorphism respecting the projections to $\Gal(\tilde{F}|F)$.
\begin{lemma}\label{prop:commutation-ss}
	Up to $\hat{G}(\overline{\Q_\ell})$-conjugacy, semi-simplification commutes with $L$-automorphisms of $\Lgrp{G}$.
\end{lemma}
\begin{proof}
	Indeed, an $L$-automorphism permutes the parabolic subgroups of $\Lgrp{G}$ together with their Levi decompositions.
\end{proof}

\section{Statement of a variant of the conjecture}\label{sec:statement}
\subsection{Chevalley involutions}\label{sec:Chevalley}
To begin with, we consider a split connected reductive group $H$ over a field, equipped with a pinning $\mathcal{P} = (B,T, (X_\alpha)_{\alpha \in \Delta_0})$, where
\begin{compactitem}
	\item $(B,T)$ is a Borel pair of $H$, and
	\item $X_\alpha$ is a nonzero vector in the root subspace $\mathfrak{h}_\alpha$, where $\alpha$ ranges over the set $\Delta_0$ of $B$-simple roots.
\end{compactitem}

\begin{definition}
	The \emph{Chevalley involution} $\theta = \theta_{\mathcal{P}}$ is the unique pinned automorphism of $H$ acting as $t \mapsto w_0(t^{-1})$ on $T$, where $w_0$ stands for the longest element in the Weyl group associated to $T$.
\end{definition}
This is the definition in \cite[\S 4]{Pra18}, and it is clear that $\theta^2 = \identity_H$.

The Chevalley involution will be considered in the following settings. Let $F$ be a field with separable closure $\bar{F}$.
\begin{enumerate}
	\item Let $H = \hat{G}$ be the dual group of $G$, which is connected reductive over $F$. The dual group is endowed with a pinning and we obtain $\theta: \hat{G} \to \hat{G}$. Since $\Gal(\tilde{F}|F)$ operates by pinned automorphisms on $\hat{G}$, the Chevalley involution extends to
	\[ \Lgrp{\theta}: \Lgrp{G} \to \Lgrp{G}, \quad g \rtimes \sigma \mapsto \theta(g) \rtimes \sigma, \]
	which is still an involution.
%	Originally: the discussions preceding \cite[\S 4, Conjecture 2]{Pra18} imply that $\theta$ extends to an $L$-automorphism...
	\item Let $G$ be a quasisplit connected reductive group over $F$. Then $G$ admits an $F$-pinning $\mathcal{P}$, i.e.\ a Galois-invariant pinning of $H := G_{\overline{F}}$. Therefore the Chevalley involution $\theta = \theta_{\mathcal{P}}$ for $G_{\overline{F}}$ descends to $G$.
\end{enumerate}

Furthermore, observe that if $H = \prod_{i=1}^r H_i$ and $\mathcal{P}$ decomposes into $(\mathcal{P}_1, \ldots, \mathcal{P}_r)$ accordingly, the corresponding Chevalley involution $\theta_{\mathcal{P}}$ equals $\prod_{i=1}^r \theta_{\mathcal{P}_i}$.

\subsection{The local statement}\label{sec:local-statement}
Let $F$ be a local field of characteristic $p > 0$. Let $G$ be a connected reductive $F$-group. The set of isomorphism classes of irreducible smooth representations over $\overline{\Q_\ell}$ of $G(F)$ will be denoted by $\Pi(G)$. The local statement to follow presumes a given \emph{Langlands parametrization} of representations, namely an arrow
\begin{align*}
	\Pi(G) & \to \Phi(G) \\
	\pi & \mapsto \phi.
\end{align*}
This is the ``automorphic to Galois'' direction of the local Langlands correspondence for $G$. We say that $\phi$ is the parameter of $\pi$, and denote by $\Pi_\phi \subset \Pi(G)$ the fiber over $\phi$, called the \emph{packet} associated to $\phi$.

For the local statement, we employ the Langlands parameterization furnished by Genestier--Lafforgue \cite{GL17}. It is actually an arrow
\[ \Pi(G) \to \left\{ \text{semisimple $L$-parameters} \right\} \big/ \hat{G}(\overline{\Q_\ell})\text{-conj.} \subset \Phi(G). \]

\begin{remark}\label{rem:GL-semisimplified}
	The Genestier--Lafforgue parameters are expected to be the semi-simplifications of authentic (yet hypothetical) Langlands parameters. As a consequence, the packets $\Pi_\phi$ for general Genestier--Lafforgue parameters are expected to be a disjoint union of authentic $L$-packets, unless when $\phi$ is an elliptic parameter (see Lemma \ref{prop:semisimple-equiv}), i.e.\ $\Image(\phi)$ is $\Lgrp{G}$-ir in the sense of \cite[3.2.1]{Se05}.
\end{remark}

Further descriptions and properties of the Genestier--Lafforgue parameterization will be reviewed in due course. Let us move directly to the main local statement.

\begin{theorem}\label{prop:local-contragredient}
	Let $\phi \in \Phi(G)$ be a semisimple $L$-parameter. In terms of the Langlands parameterization of Genestier--Lafforgue, we have
	\[ \left\{ \check{\pi} : \pi \in \Pi_\phi \right\} = \Pi_{\Lgrp{\theta} \circ \phi}, \]
	where $\Lgrp{\theta}: \Lgrp{G} \to \Lgrp{G}$ is the Chevalley involution in \S\ref{sec:Chevalley}.
\end{theorem}

If the Genestier--Lafforgue parameterization is replaced by an authentic Langlands parameterization, the statement above becomes \cite[Conjecture 1.1]{AV16} by Adams--Vogan; it is also a part of \cite[\S 4, Conjecture 2]{Pra18} by D.\ Prasad, but Prasad's conjecture also predicates on the internal structure of $L$-packets. The conjecture of Adams--Vogan--Prasad applies to any local field $F$; known cases in this generality include
\begin{itemize}
	\item the case $F = \R$ in \cite[Theorem 7.1 (a)]{AV16}, with admissible representations of $G(\R)$ over $\CC$;
	\item the tempered $L$-packets for symplectic groups $\Sp(2n)$ and quasisplit $\SO$ groups over non-Archimedean local fields $F$ of characteristic zero in terms of Arthur's endoscopic classification, see \cite[Corollary 5.10]{Kal13};
	\item the depth-zero and epipelagic $L$-packets for many $p$-adic groups \cite[\S 6]{Kal13}.
\end{itemize}
Each case above requires a different construction of $L$-packets, applicable to different groups or parameters, whereas the Theorem \ref{prop:local-contragredient} furnishes a uniform statement. On the other hand, Theorem \ref{prop:local-contragredient} is weaker since the Langlands parametrization here is coarser, in view of the Remark \ref{rem:GL-semisimplified}.

The proof of Theorem \ref{prop:local-contragredient} will occupy \S\ref{sec:local-global}.

\subsection{The global statement}
Theorem \ref{prop:local-contragredient} will be connected to the global result below.

Let $\mathring{F} = \F_q(X)$ and fix the level $N \subset X$ as in \S\ref{sec:cusp-forms}. Let $G$, $K_N$ and $\Xi$ be as in \S\ref{sec:cusp-forms}, so that $\Bun_{G,N}$ is defined. Note that we need to choose a model of $G$ over $X$ which is a Bruhat--Tits group scheme, still denoted as $G$. Let $U \subset X$ denote the (open) locus of good reduction of $G$, and set
\begin{equation}\label{eqn:hat-N}
	\hat{N} := N \cup (X \smallsetminus U).
\end{equation}
This is a finite closed $\F_q$-subscheme of $X$, the ``unramified locus''. Let $\eta \to X$ be the generic point of $X$; fix a geometric generic point $\overline{\eta} \to \eta$ of $X$.

The main global result of V.\ Lafforgue \cite[Théorème 12.3]{Laf18} gives a canonical decomposition of $C_c(K_N \backslash G(\A) /K_N; \overline{\Q_\ell})$-modules
\begin{equation}\label{eqn:global-decomp}
	\Ccusp\left( \Bun_{G,N}(\F_q)/\Xi; \overline{\Q_\ell} \right) = \bigoplus_\sigma \mathfrak{H}_\sigma
\end{equation}
indexed by $L$-parameters $\sigma: \Gal( \overline{\mathring{F}} | \mathring{F}) \to \Lgrp{G}(\overline{\Q_\ell})$ up to $\hat{G}(\overline{\Q_\ell})$-conjugacy that
\begin{compactitem}
	\item are semisimple, and
	\item factor continuously through $\Gal( \overline{\mathring{F}} | \mathring{F}) \to \pi_1(X \smallsetminus \hat{N}, \overline{\eta})$.
\end{compactitem}
% Note that the left-hand side of \eqref{eqn:global-decomp} can be decomposed according to the finite set $\Ker^1(\mathring{F}, G)$.

\begin{remark}\label{rem:realize-in-H}
	Since the left-hand side of \eqref{eqn:global-decomp} is a semisimple module, of finite dimension over $\overline{\Q_\ell}$, so are its submodules $\mathfrak{H}_\sigma$. To each $\sigma$ we may associate a set (with multiplicities) of simple submodules $C_\sigma$, such that
	\[ \mathfrak{H}_\sigma = \bigoplus_{\mathcal{L} \in C_\sigma} \mathcal{L} , \quad \text{hence} \quad \Ccusp\left( \Bun_{G,N}(\F_q)/\Xi; \overline{\Q_\ell} \right) = \bigoplus_\sigma \bigoplus_{\mathcal{L} \in C_\sigma} \mathcal{L} \]
	as $C_c(K_N \backslash G(\A) /K_N; \overline{\Q_\ell})$-modules.
\end{remark}

The decomposition \eqref{eqn:global-decomp} is built on two pillars: the theories of excursion operators and pseudo-characters for $\Lgrp{G}$. As in the local case, we defer the necessary details of \cite{Laf18} to \S\ref{sec:overview}.

\begin{theorem}\label{prop:global-contragredient}
	Suppose that $\mathfrak{H}_{\sigma}$, $\mathfrak{H}_{\sigma'}$ are two nonzero summands in \eqref{eqn:global-decomp} such that the restriction
	\[ \lrangle{\cdot, \cdot}_{\sigma, \sigma'} : \mathfrak{H}_\sigma \dotimes{\overline{\Q_\ell}} \mathfrak{H}_{\sigma'} \to \overline{\Q_\ell} \]
	of the integration pairing $\lrangle{\cdot, \cdot}$ of Remark \ref{rem:pairing-level} (extended to $\Bun_{G,N}(\F_q)/\Xi$) is not identically zero. Then we have
	\[ \sigma' = \Lgrp{\theta} \circ \sigma \quad \text{in}\; \Phi(G); \]
	here $\Lgrp{\theta}$ is the Chevalley involution of $\Lgrp{G}$.
\end{theorem}

The proof of Theorem \ref{prop:global-contragredient} will be accomplished at the end of \S\ref{sec:computation-transpose}.

\subsection{Local--global argument}\label{sec:local-global}
Consider a connected reductive group $G$ over a local field $F$ of characteristic $p$ as in the local setting \S\ref{sec:local-statement}. As usual, $A_G$ stands for the maximal central split torus in $G$, and $\tilde{F}|F$ stands for the splitting field of $G$. Take a maximal torus $T \subset G$ with splitting field equal to $\tilde{F}$. Let $\mathrm{H}^1(W_F, Z_{\hat{G}})$ denote the continuous cohomology with values in $Z_{\hat{G}}(\overline{\Q_\ell})$ with discrete topology. 

The first lemma concerns the Langlands parameterization of smooth characters of $G(F)$. The general case turns out to be delicate: by the discussion in \cite[Appendix A]{LM15}, the usual cohomological construction actually yields an arrow in the opposite direction:
\[\begin{tikzcd}[row sep=small]
	\mathrm{H}^1(\Weil{F}, Z_{\widehat{G}}) \arrow[r] \arrow[hookrightarrow, d] & \left\{ \eta: G(F) \to \overline{\Q_\ell}^\times, \; \text{smooth character} \right\} \\
	\Phi(G) &
\end{tikzcd}\]
It is injective but not necessarily surjective. However, we only need the invert it when $\eta|_{G(F)^1}$ is trivial. This is well-known to experts, and below is a sketch.

\begin{lemma}\label{prop:character-parameter}
	For $G$ as above, there is a canonical homomorphism of groups
	\[ \left\{ \begin{array}{rl}
		\eta: & G(F)/G(F)^1 \to \overline{\Q_\ell}^\times \\
		& \text{smooth character}
		\end{array} \right\}
	\to \mathrm{H}^1\left( \Weil{F}, Z_{\widehat{G}} \right). \]
	Here we do not assume $\mathrm{char}(F) > 0$.
\end{lemma}
\begin{proof}
	Fix $\eta$. First, one can take a $z$-extension of $G$ as in \cite[Proof of Lemma A.1]{LM15}, i.e. a central extension
	\[ 1 \to C \to G_1 \xrightarrow{p} G \to 1, \quad C: \text{induced torus}, \quad G^\mathrm{der}_1: \text{simply connected}.  \]
	Then $\eta_1 := \eta \circ p$ is trivial on $G_1(F)^1$. We know that $\mathrm{H}^1(F, G_1^\text{der})$ is trivial. Put $S := G_1/G_1^\text{der}$ so that $G_1(F)/G_1^\text{der}(F) \rightiso S(F)$ and $\widehat{S} \simeq Z_{\widehat{G_1}} = Z_{\widehat{G_1}}^\circ$. Then $G_1^\text{der}(F) \subset G_1(F)^1$ implies that $\eta_1$ factors through $S(F)$. The local classfield theory affords an element $a \in \mathrm{H}^1(\Weil{F}, Z_{\widehat{G_1}})$. Since $\eta_1|_C = 1$, we infer that $a$ has trivial image in $\mathrm{H}^1(\Weil{F}, \widehat{C})$. 

	Furthermore, using that fact that $C$ is induced, in \textit{loc.\ cit.} is constructed a natural isomorphism
	\[ \mathrm{H}^1(\Weil{F}, Z_{\widehat{G}}) \simeq \Ker\left[ \mathrm{H}^1(\Weil{F}, Z_{\widehat{G_1}}) \to \mathrm{H}^1(\Weil{F}, \widehat{C}) \right]. \]
	All in all, we obtain $a \in \mathrm{H}^1(\Weil{F}, Z_{\widehat{G}})$. It is routine to check that $\eta \mapsto a$ is independent of the choice of $z$-extensions, cf.\ \textit{loc.\ cit.}
\end{proof}
In fact, $\eta$ corresponds to some class in $\mathrm{H}^1(\Weil{F}/I_F, Z_{\widehat{G}}^{I_F})$. To see this, one readily reduces to the case of a torus $S$ as above. Since $S(F)^1$ contains the parahoric subgroup, one can infer, for example by the Satake isomorphism \cite[Proposition 1.0.2]{HR10} for $S$, that we obtain a parameter in $\mathrm{H}^1(\Weil{F}/I_F, \widehat{S}^{I_F})$.

The second lemma concerns the globalization of groups.
\begin{lemma}\label{prop:globalization-group}
	Given $G$ and $F$ as above, one can choose
	\begin{compactitem}
		\item $\mathring{F}$: a global field of characteristic $p$;
		\item $\mathring{G}$: a connected reductive $\mathring{F}$-group with maximal $\mathring{F}$-torus $\mathring{T}$, sharing the same splitting field $\widetilde{\mathring{F}} | \mathring{F}$;
		\item $v$: a place of $\mathring{F}$, and $w$ is the unique place of $\widetilde{\mathring{F}}$ lying over $v$, in particular $\Gal(\widetilde{\mathring{F}} | \mathring{F})$ equals the decomposition group $\Gamma_w := \Gal(\widetilde{\mathring{F}}_w | \mathring{F}_v)$;
	\end{compactitem}
	such that
	\begin{compactitem}
		\item there exist isomorphisms $\mathring{F}_v \simeq F$, $\widetilde{\mathring{F}}_w \simeq \tilde{F}$, which identify $\Gamma := \Gal(\tilde{F}|F)$ with $\Gamma_w$;
		\item under the identifications above, there is an isomorphism
		$\begin{tikzcd}[row sep=tiny, column sep=small]
			\mathring{G}_{\mathring{F}_v} \arrow[r, "\sim"] & G \\
			\mathring{T}_{\mathring{F}_v} \arrow[r, "\sim"] \arrow[phantom, u, "\subset" description, sloped] & T \arrow[phantom, u, "\subset" description, sloped]
		\end{tikzcd}$, i.e.\ $\mathring{G} \supset \mathring{T}$ is an $\mathring{F}$-model of $G \supset T$;
		\item $\mathring{G}$ and $G$ share the same root datum endowed with actions of $\Gamma \simeq \Gamma_w$, relative to $\mathring{T}$ and $T$ respectively.
	\end{compactitem}
\end{lemma}
\begin{proof}
	Standard. See for instance \cite[p.526]{Ar88-2} or \cite[3.12]{Vig01}.
\end{proof}

\begin{remark}\label{rem:globalization-group}
	The matching of root data in Lemma \ref{prop:globalization-group} also implies that $A_{\mathring{G}}$ is ``the same'' as $A_G$. Hereafter, we shall drop the clumsy notation $\mathring{G}$, $\mathring{T}$ or $A_{\mathring{G}}$, and denote them abusively as $G$, $T$ or $A_G$ instead.
	
	For any closed discrete subgroup $\Xi \subset A_G(F)$ isomorphic to $\Z^{\dim A_G}$, its isomorphic image in $A_G(\mathring{F}) \backslash A_G(\A)$ will also be denoted by $\Xi$. Another consequence of Lemma \ref{prop:globalization-group} is that $\Xi$ is a co-compact lattice in $A_G(\mathring{F}) \backslash A_G(\A)$ satisfying the requirements in \S\ref{sec:cusp-forms}.
\end{remark}

\begin{proof}[Proof of Theorem \ref{prop:local-contragredient} from Theorem \ref{prop:global-contragredient}]
	In what follows, we write $\pi \leadsto \phi$ if $\pi \in \Pi(G)$ has Genestier--Lafforgue parameter $\phi \in \Phi(G)$. It suffices to show that for every $\pi \in \Pi(G)$,
	\begin{equation}\label{eqn:local-global-aux}
		\left( \pi \leadsto \phi \right) \implies \left( \check{\pi} \leadsto \Lgrp{\theta} \circ \phi \right).
	\end{equation}
	Indeed, this assertion amounts to $\{\check{\pi} : \pi \in \Pi_\phi \} \subset \Pi_{\Lgrp{\theta} \circ \phi}$. The reverse inclusion will follow by applying \eqref{eqn:local-global-aux} to any $\pi_1 \in \Pi(G)$ with $\pi_1 \leadsto \Lgrp{\theta} \circ \phi$, which in turn yields $\pi := \check{\pi}_1 \leadsto \Lgrp{\theta} \circ \Lgrp{\theta} \circ \phi = \phi$ whilst $\pi_1 = \check{\pi}$.

	The assertion \eqref{eqn:local-global-aux} will be established in steps.

	\textit{Step 1.} We reduce \eqref{eqn:local-global-aux} to the case $\pi$ supercuspidal. Indeed, let $(M,\tau)$ be the cuspidal support of $\pi$ reviewed in \S\ref{sec:rep}. By Proposition \ref{prop:sc-support}, $\check{\pi}$ has cuspidal support $(M,\check{\tau})$.

	On the dual side, choose an $L$-embedding $\iota: \Lgrp{M} \hookrightarrow \Lgrp{G}$ as reviewed in \S\ref{sec:L-parameters}. Suppose that $\tau \leadsto \phi_\tau$ in $M$. By \cite[Théorème 0.1]{GL17}, $\phi$ equals to the composite of $\Weil{F} \xrightarrow{\phi_\tau} \Lgrp{M} \hookrightarrow \Lgrp{G}$ up to $\hat{G}(\overline{\Q_\ell})$-conjugacy. The same relation holds for the parameters for $\check{\pi}$ and $\check{\tau}$. Denoting $\Lgrp{\theta_M}$ the Chevalley involution on $\Lgrp{M}$, the diagram
	\[\begin{tikzcd}
		\Lgrp{M} \arrow[d, "{\Lgrp{\theta_M}}"'] \arrow[hookrightarrow, r, "\iota"] & \Lgrp{G} \arrow[d, "{\Lgrp{\theta}}"] \\
		\Lgrp{M} \arrow[hookrightarrow, r, "\iota"] & \Lgrp{G}
	\end{tikzcd}\]
	is commutative up to an explicit $\hat{G}(\overline{\Q_\ell})$-conjugacy, by \cite[\S 5, Lemma 4]{Pra18}. Upon replacing $(G, \pi)$ by $(M, \tau)$, we have reduced \eqref{eqn:local-global-aux} to the supercuspidal case.
		
	\textit{Step 2.} Consider the smooth character $\omega := \omega_\pi|_{A_G(F)}$. We reduce \eqref{eqn:local-global-aux} to the case that $\omega$ is of finite order as follows (see also Remark \ref{rem:new-reduction}). First, recalling \eqref{eqn:G(F)1}, there exists a character
	\[ \eta_0: A_G(F)/A_G(F)^1 \to \overline{\Q_\ell}^\times \]
	such that $\eta_0 \otimes \omega$ is of finite order. Indeed, this is easily reduced to the case $A_G \simeq \Gm$, and it suffices to take $\eta_0(\varpi) = \omega(\varpi)^{-1}$ where $\varpi \in F^\times$ is some uniformizer.

	Secondly, the inclusion of discrete free commutative groups of finite type
	\[ A_G(F)/ A_G(F)^1 = A_G(F)/A_G(F) \cap G(F)^1 \hookrightarrow G(F)/G(F)^1 \]
	has finite cokernel, whereas $\overline{\Q_\ell}^\times$ is divisible. Therefore $\eta_0$ extends to a smooth character $\eta: G(F)/G(F)^1 \to \overline{\Q_\ell}^\times$. The central character of $\pi \otimes \eta$ has finite order when restricted to $A_G(F)$.

	Attach $a \in \mathrm{H}^1(\Weil{F}, Z_{\hat{G}})$ to $\eta$ by Lemma \ref{prop:character-parameter}; it can be used to twist elements of $\Phi(G)$ by the homomorphism
	\[ \Weil{F} \ltimes (Z_{\widehat{G}} \times \widehat{G}) \to \Weil{F} \ltimes \widehat{G}, \quad w \ltimes (z, g) \mapsto w \ltimes (zg) \]
	by choosing any cocycle representative of $a$; see \cite[Remarque 0.2]{GL17}.

	In the construction above, $- \otimes \eta^{-1}$ corresponds to twisting a parameter by $a^{-1}$. We have $(\pi \otimes \eta)^\vee \simeq \check{\pi} \otimes \eta^{-1}$. Concurrently, $\Lgrp{\theta} \circ (\phi \cdot a) = (\Lgrp{\theta} \circ \phi) \cdot a^{-1}$ since Chevalley involution acts as $z \mapsto z^{-1}$ on the center. Therefore, by replacing $\pi$ by $\pi \otimes \eta$, it suffices to prove \eqref{eqn:local-global-aux} when $\omega$ has finite order.

	\textit{Step 3.} Now we can assume $\pi$ to be integral supercuspidal (see \S\ref{sec:rep}) with $\omega := \omega_\pi|_{A_G(F)}$ of finite order. By \cite{GL17}, we know that the parameter $\phi$ of $\pi$ factors through $\Gal(\overline{F}|F)$. Take a global $\mathring{F}$-model of $G \supset A_G$ as in Lemma \ref{prop:globalization-group} with $\mathring{F}_v \simeq F$. As $A_G$ is split over $\mathring{F}$, by reducing to $\Gm$ and applying \cite[Chapter X, \S 2, Theorem 5]{AT}, there exists an automorphic character
	\[ \mathring{\omega} = \bigotimes_u \mathring{\omega}_u: A_G(\mathring{F}) \backslash A_G(\A) \to \overline{\Q_\ell}^\times \]
	of finite order, such that $\mathring{\omega}_v = \omega$.
	
	Since $\omega$ is smooth, there exists a closed discrete subgroup $\Xi \subset A_G(F)$ such that $\omega|_\Xi = 1$ and $\Xi \simeq \Z^{\dim A_G}$. In view of Remark \ref{rem:globalization-group}, $\Xi$ also affords the co-compact lattice in $A_G(\mathring{F}) \backslash A_G(\A)$ required in \S\ref{sec:cusp-forms}.
	
	Claim: there exists a cuspidal automorphic representation $\mathring{\pi} = \bigotimes_u \mathring{\pi}_u$ of $G(\A)$ (in the extended sense that we consider all $G_\alpha$ simultaneously, $\alpha \in \Ker^1(\mathring{F}, G)$) such that
	\begin{itemize}
		\item the central character of $\mathring{\pi}$ equals $\mathring{\omega}$ on $A_G(\A)$;
		\item we have $\mathring{\pi}_v \simeq \pi$;
		\item relative to the chosen lattice $\Xi$ and a sufficiently deep level $N$, the $C_c(K_N \backslash G(\A)/K_N; \overline{\Q_\ell})$-module $\mathring{\pi}^{K_N}$ can be embedded in some summand $\mathfrak{H}_\sigma$ in \eqref{eqn:global-decomp}.
	\end{itemize}
	This can be achieved by the following variant of the argument in \cite[Appendice 1]{He84} (which works over $\CC$) via Poincaré series; see also the proof of \cite[Lemme 1.4]{GL17}. For each place $u$ of $\mathring{F}$, choose a smooth function $f_u \in C_c(G(\mathring{F}_u), \mathring{\omega}_u)$ such that
	\begin{compactitem}
		\item there exists a finite set $S$ of places of $\mathring{F}$ containing $v$ and the ramification locus of $G$, such that when $u \notin S$, the function $f_u$ is right $G(\mathfrak{o}_u)$-invariant, supported on $A_G(\mathring{F}_u)G(\mathfrak{o}_u)$ and $f_u(1) = 1$, where $G(\mathfrak{o}_u)$ is the hyperspecial subgroup arsing from some reductive model of $G$ over the ring of $S$-integers in $\mathring{F}$;
		\item we require $f_v$ to a matrix coefficient of $\pi$ and assume $f_v(1) \neq 0$;
		\item for every $u \in S \smallsetminus \{v\}$, we require that $C_u := \Supp|f_u|$ is a sufficiently small neighborhood of $1$ modulo $A_G(F_u)$, so that the image of
		\[ \Supp(f_v) \times \prod_{u \in S \smallsetminus \{v\}} C_u \times \prod_{u \notin S} G(\mathfrak{o}_u) A_G(\mathring{F}_u) \]
		in $A_G(\A) \backslash G(\A) = (A_G \backslash G)(\A)$ intersects $A_G(\mathring{F}) \backslash G(\mathring{F}) = (A_G \backslash G)(\mathring{F})$ only at $1$. To see why this can be achieved, embed $A_G \backslash G$ into some affine space over $F$.
	\end{compactitem}
	Take $f := \prod_u f_u: G(\A) \to \overline{\Q_\ell}$ and form
	\[ P_f(g) = \sum_{\gamma \in (A_G \backslash G)(\mathring{F})} f(\gamma g), \quad g \in G(\A). \]
	The sum is finite when $g$ is constrained in any compact subset modulo $A_G(\A)$. By choosing $N$ sufficiently deep, it furnishes an element of $C_c(G(\mathring{F}) \backslash G(\A)/K_N \Xi ; \overline{\Q_\ell})$. Moreover, $P_f(1) = f(1) \neq 0$ by the condition on supports. By looking at $f_v$, we see that $P_f$ is a cusp form.
	
	Decompose $\Ccusp(\Bun_{G, N}(\F_q)/\Xi; \overline{\Q_\ell})$ into simple submodules as in Remark \ref{rem:realize-in-H}. There exists a summand $\mathcal{L}$ contained in some $\mathfrak{H}_\sigma$ such that $P_f$ has nonzero component in $\mathcal{L}$. Let $\mathring{\pi}$ be the cuspidal automorphic representation corresponding to $\mathcal{L}$ via Proposition \ref{prop:K_N-invariants} (realized in $\bigoplus_\alpha \Ccusp(G_\alpha(\mathring{F}) \backslash G(\A) \cdots)$) where $\alpha \in \Ker^1(\mathring{F}, G)$) so that $\mathring{\pi}^{K_N} = \mathcal{L} \hookrightarrow \mathfrak{H}_\sigma$. Then $\mathring{\pi}$ has central character $\mathring{\omega}$ on $A_G(\A)$ and $\mathring{\pi}_v \simeq \pi$, since $P_f$ and $\mathcal{L}$ have similar properties under $C_c(K_N \backslash G(\A)/K_N; \overline{\Q_\ell})$.

	\textit{Step 4.} Since the integration pairings $\lrangle{\cdot, \cdot}$ of Remark \ref{rem:pairing-level} are non-degenerate, $\mathring{\pi}^{K_N} \subset \mathfrak{H}_\sigma$ must pair non-trivially with some simple $C_c(K_N \backslash G(\A) /K_N; \overline{\Q_\ell})$-submodule of some $\mathfrak{H}_{\sigma'}$. Proposition \ref{prop:K_N-invariants} implies that the simple submodule takes the form $(\mathring{\pi}')^{K_N} \subset \mathfrak{H}_{\sigma'}$ for some cuspidal automorphic representation $\mathring{\pi}'$. 

	Theorem \ref{prop:global-contragredient} then asserts $\sigma' = \Lgrp{\theta} \circ \sigma$ in $\Phi(G)$ (global version). On the other hand, $\mathring{\pi}'$ pairs non-trivially with $\mathring{\pi}$ under the integration pairing $\lrangle{\cdot, \cdot}$ of Definition \ref{def:integration-pairing}. The invariance of $\lrangle{\cdot, \cdot}$ therefore implies
	\begin{gather*}
		\bigotimes_u (\mathring{\pi}_u)^\vee = \mathring{\pi}^\vee \simeq \mathring{\pi}' \quad \text{as $G(\A)$-representations}.
	\end{gather*}
	The local-global compatibility in \cite[Théorème 0.1 (b)]{GL17} says that
	\begin{gather*}
		\pi \simeq \mathring{\pi}_v \leadsto (\sigma|_{\Gal(\overline{F}|F)} )^\mathrm{ss}, \\ 
		\check{\pi} \simeq (\mathring{\pi}_v)^\vee \simeq \mathring{\pi}'_v \leadsto (\sigma'|_{\Gal(\overline{F}|F)} )^\mathrm{ss} = \left( \Lgrp{\theta} \circ \sigma|_{\Gal(\overline{F}|F)} \right)^\mathrm{ss}.
	\end{gather*}
	Here we choose an embedding of the separable closure of $\mathring{F}$ into $\overline{F}$, and the semi-simplification is defined as in \S\ref{sec:L-parameters}. In particular, $\phi = (\sigma|_{\Gal(\overline{F}|F)} )^\mathrm{ss}$ in $\Phi(G)$ (local version).
	
	By Lemma \ref{prop:commutation-ss} we have $\left( \Lgrp{\theta} \circ \sigma|_{\Gal(\overline{F}|F)} \right)^\mathrm{ss} = \Lgrp{\theta} \circ \left( \sigma|_{\Gal(\overline{F}|F)} \right)^\mathrm{ss}$. Summarizing,
	\[ \check{\pi} \leadsto \Lgrp{\theta} \circ \left( \sigma|_{\Gal(\overline{F}|F)} \right)^\mathrm{ss} = \Lgrp{\theta} \circ \phi \]
	holds in $\Phi(G)$ (local version). This establishes \eqref{eqn:local-global-aux} and the Theorem \ref{prop:local-contragredient} follows.
\end{proof}

\begin{remark}\label{rem:new-reduction}
	As pointed out by a referee, Lemma \ref{prop:character-parameter} can be avoided in Step 2 by the following arguments. Restrict $\mathrm{quot}: G \to T := G/G_{\mathrm{der}}$ to an isogeny $A_G \to T$. The same arguments show that some smooth character $\eta: T(F) \to \overline{\Q_\ell}^\times$ pulls back to our given $\eta_0: A_G(F) \to \overline{\Q_\ell}^\times$. To complete Step 2, it remains to compare (a) the parameters of $\eta$ and $\eta^{-1}$, (b) the parameters of $\pi$ and $\pi \otimes \eta$. For (a), apply local trivial functoriality \cite[Théorème 8.1]{GL17} to the automorphism $t \mapsto t^{-1}$ of $T$. For (b), apply it to the homomorphism $G \xrightarrow{(\identity, \mathrm{quot})} G \times T$ with normal image, as performed in \cite[Remarque 0.3]{GL17}.
\end{remark}

From \S\ref{sec:overview} onwards, we will focus exclusively on Theorem \ref{prop:global-contragredient} and the underlying geometric considerations.

\subsection{Remarks on the duality involution}
Conserve the assumptions for the local statement in \S\ref{sec:local-statement} and assume $G$ is quasisplit. Fix an $F$-pinning $\mathcal{P} = (B, T, (X_\alpha)_\alpha)$ of $G$. Choose the unique $\kappa \in T^\text{ad}(F)$ such that $\kappa X_\alpha \kappa^{-1} = -X_\alpha$, for all simple root $\alpha$ with respect to $(B,T)$. Observe that $\kappa^2 = 1$ in $G^\text{ad}$.

Let $\theta = \theta_{\mathcal{P}}$ be the Chevalley involution of $G$, and $\iota_-$ be the inner involution $g \mapsto \kappa g \kappa^{-1}$. Observe that $\iota_- \theta = \theta \iota_-$. Indeed, $\iota_- \theta \iota_-$ is seen to preserve $\mathcal{P}$ and coincides with $\theta$ on $T$, hence $\iota_- \theta \iota_- = \theta$ by the characterization of the Chevalley involution.

\begin{definition}[{Cf.\ \cite[\S 3]{Pra18}}]
	Relative to the $F$-pinning $\mathcal{P}$, set $\iota_{G, \mathcal{P}} := \iota_- \theta = \theta \iota_-$. It is called the \emph{duality involution} of $G$.
\end{definition}

Recall that $\iota_{G, \mathcal{P}}$ induces a pinned automorphism of $\hat{G}$, called the \emph{dual} automorphism of $\iota_{G, \mathcal{P}}$, which depends only on $\iota_{G, \mathcal{P}}$ modulo $G^\mathrm{ad}(F)$; see \cite[\S 2.5]{Bo79} for the general set-up. This recipe applies to any base field $F$.

\begin{lemma}\label{prop:dual-Chevalley}
	The Chevalley involution on $\hat{G}$ is the dual of $\iota_{G, \mathcal{P}}$ in the sense above. This result holds over any field $F$. 
\end{lemma}
\begin{proof}
	Since $\iota_-$ comes from $G^\mathrm{ad}(F)$-action, $\iota_{G,\mathcal{P}}$ and $\theta$ have the same dual. It suffices to show that the Chevalley involution of $\hat{G}$ is dual to that of $G$. Since both automorphisms are pinned, it suffices to show that the induced automorphisms on $X_*(T_{\overline{F}})$ and $X^*(T_{\overline{F}})$ are mutually dual. Recall that the Chevalley involution of $G$ (resp.\ $\hat{G}$) acts on $X_*(T_{\overline{F}})$ (resp.\ $X^*(T_{\overline{F}})$) as $x \mapsto -w_0(x)$, where $w_0$ is the longest element in the Weyl group. Since $w_0^2 = 1$, these two automorphisms are indeed mutually dual.
\end{proof}

Fix a nontrivial smooth character $\psi: F \to \overline{\Q_\ell}^\times$. From the $F$-pinning $\mathcal{P} = (B,T,(X_\alpha)_\alpha)$ we produce a \emph{Whittaker datum} $\mathfrak{w} := (U, \chi)$ for $G$ taken up to $G(F)$-conjugacy, that is,
\begin{compactitem}
	\item $U$ is the unipotent radical of $B$,
	\item $\chi: U(F) \to \overline{\Q_\ell}^\times$ is the composition of $\psi$ with the algebraic character $U \to \Ga$ mapping each $X_\alpha$ to $1$.
\end{compactitem}
The automorphisms of $G$ act on $F$-pinnings, thereby act on Whittaker data. Put
\[ \mathfrak{w}' := (U, \chi^{-1}) = \iota_- \mathfrak{w}. \]

Fix $\psi$, $\mathcal{P}$ and the associated Whittaker datum $\mathfrak{w}$ for $G$. Let $\phi \in \Phi(G)$ be a semisimple parameter. Define the Genestier--Lafforgue packet $\Pi_\phi$ as in \S\ref{sec:local-statement}. We say that \emph{Shahidi's property} holds for $\Pi_\phi$ and $\mathfrak{w}$, if
\begin{equation}\label{eqn:Shahidi}
	\exists! \pi \in \Pi_\phi \quad \text{such that}\; \pi \text{ is $\mathfrak{w}$-generic}.
\end{equation}
Further discussions about this property will be given in Remark \ref{rem:Shahidi}.

\begin{lemma}\label{prop:equiv-genericity}
	The following are equivalent for an irreducible smooth representation $\pi$ of $G(F)$:
	\begin{inparaenum}[(i)]
		\item $\pi$ is $\mathfrak{w}$-generic;
		\item $\pi \circ \theta$ is $\mathfrak{w}$-generic;
		\item $\check{\pi}$ is $\mathfrak{w}'$-generic;
		\item $\pi \circ \iota_-$ is $\mathfrak{w}'$-generic.
	\end{inparaenum}
\end{lemma}
\begin{proof}
	(i) $\iff$ (ii) since $\theta$ preserves $\mathcal{P}$. (i) $\iff$ (iii) is \cite[\S 4, Lemma 2]{Pra18}. (i) $\iff$ (iv) follows from transport of structure by the involution $\iota_-$.
\end{proof}

The following result serves as a partial heuristic for Prasad's \cite[\S 3, Conjecture 1]{Pra18}.
\begin{theorem}\label{prop:duality-involution}
	Define the Whittaker data $\mathfrak{w}$ and $\mathfrak{w}'$ as above. Let $\phi \in \Phi(G)$ be a semisimple parameter such that $\Pi_\phi$ satisfies Shahidi's property \eqref{eqn:Shahidi} with respect to $\mathfrak{w}$. Then the following hold.
	\begin{enumerate}[(i)]
		\item The packet $\Pi_{\Lgrp{\theta} \circ \phi}$ satisfies Shahidi's property \eqref{eqn:Shahidi} with respect to $\mathfrak{w}'$.
		\item Let $\pi$ be the unique $\mathfrak{w}$-generic member of $\Pi_\phi$, then $\check{\pi}$ is the unique $\mathfrak{w}'$-generic member of $\Pi_{\Lgrp{\theta} \circ \phi}$.
		\item If $\pi \in \Pi_\phi$ is $\mathfrak{w}$-generic, then $\check{\pi} \simeq \pi \circ \iota_{G, \mathcal{P}}$.
	\end{enumerate}
\end{theorem}
\begin{proof}
	Parts (i), (ii) follow immediately from Lemma \ref{prop:equiv-genericity} and Theorem \ref{prop:local-contragredient}, which says that $\Pi_{\Lgrp{\theta} \circ \phi} = \left\{ \check{\pi} : \pi \in \Pi_\phi \right\}$.
	
	Now consider (iii). We claim that $\pi \circ \iota_{G, \mathcal{P}} \in \Pi_{\Lgrp{\theta} \circ \phi}$. In view of Lemma \ref{prop:dual-Chevalley}, the corresponding statement for global Langlands parameterization of cuspidal automorphic representations follows from the ``trivial functoriality'' (under the dual of $\iota_{G, \mathcal{P}}$) in \cite[Théorème 0.1 et 8.1]{GL17}.

	Using Lemma \ref{prop:equiv-genericity}, we see that $\pi \circ \iota_{G, \mathcal{P}} = (\pi \circ \theta) \circ \iota_-$ is also a $\mathfrak{w}'$-generic member of $\Pi_{\Lgrp{\theta} \circ \phi}$. It follows from (ii) that $\check{\pi} \simeq \pi \circ \iota_{G, \mathcal{P}}$.
\end{proof}

\begin{remark}\label{rem:Shahidi}
	Choose an isomorphism $\overline{\Q_\ell} \rightiso \CC$ and let $\phi \in \Phi(G)$ be semisimple. By a conjecture of Shahidi \cite[Conjecture 9.4]{Sh90}, one expects that when $\phi$ is a tempered $L$-parameter, \eqref{eqn:Shahidi} will hold for the \emph{authentic $L$-packet} associated to $\Pi_\phi$ and for any $\mathfrak{w}$.
	
	On the other hand, \cite[Conjecture 2.6]{GP92} proposes a characterization of $L$-parameters satisfying \eqref{eqn:Shahidi}. It is stated in terms of adjoint $L$-factors, thus applies directly to the $\ell$-adic case. The author is grateful to Yeansu Kim for this comment.

	Because of the semi-simplified nature of our packet $\Pi_\phi$, see Remark \ref{rem:GL-semisimplified}, we expect \eqref{eqn:Shahidi} to hold only when $\phi$ is not the semi-simplification of any other $L$-parameter. This occurs when $\phi$ is elliptic, in which case every $\pi \in \Pi_\phi$ is supercuspidal: otherwise the compatibility of the parameterization $\pi \leadsto \phi$ with cuspidal supports will force $\phi$ to factor through some proper Levi. It is believed that the authentic $L$-packets for elliptic $\phi$ have the same property. Many constructions of such $L$-packets have been proposed, such as in \cite{Kal16}. Nonetheless, the precise relation of these packets to the Langlands parameterization of Genestier--Lafforgue \cite{GL17} remains to be settled.
\end{remark}

\begin{remark}
	As shown in \cite{Pra18}, up to $G(F)$-conjugacy, $\iota_{G, \mathcal{P}}$ reduces to the well-known MVW involution when $G$ is classical; it reduces to $g \mapsto {}^t g^{-1}$ when $G = \GL(n)$. According to \cite[\S 3, Corollary 1]{Pra18}, when $Z_G$ is an elementary $2$-group, $\iota_{G, \mathcal{P}}$ is independent of $\mathcal{P}$ up to $G(F)$-conjugacy.
\end{remark}

\section{Overview of the global Langlands parameterization}\label{sec:overview}
\subsection{Geometric setup}\label{sec:geometric-setup}
Fix some power $q$ of a prime number $p$. Take $E \subset \overline{\Q_\ell}$ to be a finite extension of $\Q_\ell$ containing a square root $q^{1/2}$ of $q$, which we fix once and for all. The sheaves and complexes under consideration will be $E$-linear.

Suppose that $S$ is a smooth $\F_q$-scheme of finite type and of pure dimension $d$. For any reasonable algebraic stack $\mathcal{X}$ equipped with a morphism $\mathfrak{p}: \mathcal{X} \to S$, define the \emph{normalized perverse sheaves} on $\mathcal{X}$ with respect to $S$ to be of the form
\[ \mathcal{F}[-d]\left( -\frac{d}{2} \right), \quad \mathcal{F}: \text{non-normalized perverse sheaf}. \]
The usual operations on constructible complexes continue to hold in the normalized setting, with the proviso that the dualizing complex in \cite[\S 7.3]{LO2} becomes
\[ \Omega_{\mathcal{X}} := \left( \text{the non-normalized one} \right)[-2d](-d) \simeq \mathfrak{p}^! \left( E_S \right) \]
and the duality operator becomes $\VerD = \iHom(-, \Omega_{\mathcal{X}})$ accordingly. This formalism extends to ind-stacks, etc.\ with a morphism to $S$. When $S = \Spec \F_q$, we revert to the usual definitions.

Next, assume $\mathring{F} = \F_q(X)$ is a global field and $G$ is a connected reductive $\mathring{F}$-group with a chosen Bruhat--Tits model over $X$, as in \S\ref{sec:cusp-forms}. Fix a maximal $\mathring{F}$-torus $T \subset G$. Also recall that $\hat{G}$ carries a Galois-stable pinning $(\hat{B}, \hat{T}, (X_\alpha)_\alpha)$. Enlarging $E$ if necessary, we can assume that:
\begin{center}
	All irreducible $\overline{\Q_\ell}$-representations of $\Lgrp{G}$ are realized over $E$.
\end{center}

Fix a partition of a finite set
\[ I = I_1 \sqcup \cdots \sqcup I_k \]
used to label points on $X$ and a level $N \subset X$. Set $\hat{N} = |N| \cup (X \smallsetminus U)$ as in \eqref{eqn:hat-N}. Define the \emph{Hecke stack} $\Hecke^{(I_1, \ldots, I_k)}_{N,I}$ that maps each $\F_q$-scheme $S$ to the groupoid
\begin{equation}\label{eqn:Hecke-stack} \Hecke^{(I_1, \ldots, I_k)}_{N,I}(S) = \left\{
	\begin{array}{c|l}
		(x_i)_i \in (X \smallsetminus \hat{N})(S)^I & \phi_j: \text{defined off } \bigcup_{i \in I_j} \Gamma_{x_i} \\
		\left( (\mathcal{G}_j, \psi_j) \in \Bun_{G,N}(S) \right)_{j=0}^k & \psi_j \phi_j|_{N \times S} = \psi_{j-1} \\
		\phi_j: \mathcal{G}_{j-1} \dashrightarrow \mathcal{G}_j & \forall j = 1, \ldots, k
\end{array}\right\} \end{equation}
where $\Gamma_{x_i}$ stands for the graph of $x_i: S \to X$. The points $(x_i)_{i \in I}$ are known as the ``paws''.

The reason for partitioning $I$ into $I_1, \ldots, I_k$ is to define \emph{partial Frobenius morphisms}, see \S\ref{sec:partial-Frobenius}.

The ind-scheme $\Gra^{(I_1, \ldots, I_k)}_I$, the factorization version of affine Grassmannian of Beilinson--Drinfeld, is the space classifying the same data \eqref{eqn:Hecke-stack} as $\Hecke^{(I_1, \ldots, I_k)}_{I, \emptyset}$ together with a trivialization $\theta$ of $\mathcal{G}_k$. It also admits a morphism of ``paws'' to $X^I$. In fact $\Gra^{(I_1, \ldots, I_k)}_I$ is ind-projective; we refer to \cite[\S 1]{Laf18} for further details. When $I$ is a singleton and $k=1$, the usual Beilinson--Drinfeld Grassmannian over $X$ is recovered.

The \emph{factorization structure} here means that given a surjection $\zeta: I \to J$, we have the canonical isomorphism below
\begin{gather*}
	U_\zeta := \left\{ (x_i)_{i \in I}: \zeta(a) \neq \zeta(b) \implies x_a \neq x_b \right\} \subset X^I, \quad I'_a := I_a \cap \zeta^{-1}(j), \; \forall j, \\
	\Gra^{(I_1, \ldots, I_k)}_I \dtimes{X^I} U_\zeta \rightiso \prod_{j \in J} \Gra^{(I'_1, \ldots, I'_k)}_{\zeta^{-1}(j)} \quad \text{ over } U_\zeta;
\end{gather*}
see \cite[Remarque 1.9]{Laf18}. The factorization structure is mainly to be employed together with the complexes that are \emph{universally locally acyclic}, hereafter abbreviated as \emph{ULA}, with respect to the base (say $X^I$). This property (cf.\ \cite[\S 3.2]{Ric14} or \cite[\S 5.1]{BG02}) is immensely useful for ``spreading out'' certain properties of complexes from some open subset in the base, see eg.\ \cite[Theorem 3.16]{Ric14}.

Given $(n_i)_i \in \Z_{\geq 0}^I$. Define $\Gamma_{\sum_i n_i x_i} \subset X \times X^I$ as the closed subscheme Zariski--locally defined by $\prod_{i \in I} t_i^{n_i}$, with $t_i$ being a local equation for $x_i$ in $X$, where $(x_i)_{i \in I}$ are the aforementioned ``paws''. Then define
\[ G_{\sum_{i \in I} n_i x_i} := \text{the Weil restriction of}\; G \quad \text{w.r.t.}\quad \Gamma_{\sum_i n_i x_i} \to X^I. \]
One interprets $G_{\sum_i \infty x_i}$ in the same manner by considering formal neighborhoods, but we won't go into the details.

As in the discussion preceding \cite[Proposition 1.10]{Laf18}, there is a notion of $G_{\sum_i \infty x_i}$-action on $\Gra^{(I_1, \ldots, I_k)}_I$, namely by altering the trivialization $\theta$ of $\mathcal{G}_k$ at $\Gamma_{\sum_i \infty x_i}$.

%Next, assume temporarily that $G$ is split and $\Lgrp{G} = \hat{G}$ for simplicity. Fix a Borel pair $B \supset T$ for $G$. Given $\underline{\omega} = (\omega_i)_{i \in I}$, where each coweight $\omega_i \in X_*(T)$ is a dominant, one can truncate $\Gra^{(I_1, \ldots, I_k)}_I$ into closed subschemes $\Gra^{(I_1, \ldots, I_k)}_{I, \lesssim \underline{\omega}}$ by bounding the ``relative position'' of the datum $\phi_j: \mathcal{G}_{j-1} \dashrightarrow \mathcal{G}_j$ in \eqref{eqn:Hecke-stack} around $x_i$ by $\omega_i$, for $I_j \ni i$. By \textit{loc.\ cit.}, when restricted to $\Gra^{(I_1, \ldots, I_k)}_{I, \lesssim \underline{\omega}} \hookrightarrow \Gra^{(I_1, \ldots, I_k)}_I$, the $G_{\sum_i \infty x_i}$-action factors through $G_{\sum_{i \in I} n_i x_i}$ provided that $\forall n_i \gg 0$ with respect to $\underline{\omega}$. Define the reduced closed subscheme $\Gra^{(I_1, \ldots, I_k)}_{I, \underline{\omega}} \subset \Gra^{(I_1, \ldots, I_k)}_{I, \lesssim \underline{\omega}}$ as in \cite[Définition 1.12]{Laf18}.

Let $\cate{Perv}_{G_{\sum_i \infty x_i}}\left(\Gra^{(I_1, \ldots, I_k)}_I\right)$ denote the category of $G_{\sum_i \infty x_i}$-equivariant normalized perverse sheaves on the ind-scheme $\Gra^{(I_1, \ldots, I_k)}_I$ relative to $X^I$; for non-split $G$, we confine ourselves to $(X \smallsetminus \hat{N})^I$ as in \cite[\S 12.3.1]{Laf18}. The factorization version of \emph{geometric Satake equivalence} \cite[Théorèmes 1.17 + 12.16]{Laf18} gives an additive functor
\begin{align*}
	\cate{Rep}_E((\Lgrp{G})^I) & \longrightarrow \cate{Perv}_{G_{\sum_i \infty x_i}}\left(\Gra^{(I_1, \ldots, I_k)}_I\right) \\
	W & \longmapsto \EuScript{S}^{(I_1, \ldots, I_k)}_{I,W,E}.
\end{align*}
For later reference, we record some of the basic properties of this functor, all of which can be found in \textit{loc.\ cit.}
\begin{enumerate}
	\item The normalized perverse sheaves $\EuScript{S}^{(I_1, \ldots, I_k)}_{I,W,E}$ are ULA relative to the morphism to $X^I$ (or $(X \smallsetminus \hat{N})^I$).
	\item When $|I| = 1$, the geometric Satake equivalence \cite{Ric14, Zh15} yields $W \mapsto \EuScript{S}^{(I)}_{I, W, E}$. This extends to general $I$ and ``factorizable'' $W$ using the factorization structure on affine Grassmannians, cf.\ \cite{Laf18}. Namely, for any family $(W_i)_{i \in I}$ of objects in $\cate{Rep}_E(\Lgrp{G})$, one can associate $\EuScript{S}^{(I_1, \ldots, I_k)}_{I, \boxtimes_i W_i, E}$ in $\cate{Perv}_{G_{\sum_i \infty x_i}}\left(\Gra^{(I_1, \ldots, I_k)}_I\right)$.
	\item Write $\Lgrp{G}^I$ for $(\Lgrp{G})^I$. In order to obtain a functorial construction in all $W \in \cate{Rep}_E(\Lgrp{G}^I)$, we take the $\Lgrp{G}^I \times \Lgrp{G}^I$-representation $\EuScript{R} := \boxtimes_{i \in I} \mathscr{O}(\Lgrp{G})$ over $E$. This becomes an ind-object of $\cate{Rep}_E(\Lgrp{G}^I)$ using the $\Lgrp{G}^I$-action on the first slot, and this ind-object carries a $\Lgrp{G}^I$-action from the second slot. Take a system of representatives of irreducible objects $V \in \cate{Rep}_E(\Lgrp{G})$. As $\Lgrp{G}^I \times \Lgrp{G}^I$-representations, we have
	\[ \bigoplus_{V: \text{irred}} V \dotimes{E} V^\vee \rightiso \EuScript{R} \quad \text{by taking matrix coefficients}. \]
	The decomposition above and the available $\EuScript{S}^{(I_1, \ldots, I_k)}_{I, V, E}$ define a normalized ind-perverse sheaf $\EuScript{S}^{(I_1, \ldots, I_k)}_{I, \EuScript{R}, E}$, with the $\hat{G}^I$-action inherited from the second slot of $\EuScript{R}$. Now we can define, for each $W \in \cate{Rep}_E(\Lgrp{G}^I)$,
	\begin{equation}\begin{aligned}
		\EuScript{S}^{(I_1, \ldots, I_k)}_{I,W,E} & := \left( \EuScript{S}^{(I_1, \ldots, I_k)}_{I, \EuScript{R}, E} \dotimes{E} W \right)^{\Lgrp{G}^I} \\
		& \simeq \bigoplus_{V: \text{irred}} \EuScript{S}^{(I_1, \ldots, I_k)}_{I, V, E} \dotimes{E} \left( V^\vee \dotimes{E} W \right)^{\Lgrp{G}^I} \simeq \bigoplus_{V: \text{irred}} \EuScript{S}^{(I_1, \ldots, I_k)}_{I, V, E} \dotimes{E} \mathfrak{W}_V,
	\end{aligned}\end{equation}
	where
	\begin{inparaenum}[(a)]
		\item $W$ is viewed as a constant sheaf on $\Gra^{(I_1, \ldots, I_k)}_I$,
		\item $\Lgrp{G}^I$ acts diagonally, and
		\item $\mathfrak{W}_V$ stands for the multiplicity space of $V$ in $W$.
	\end{inparaenum}
	Functoriality in $W$ is clear, and it is readily seen to agree with the previous step if $W = V$, up to isomorphism.

\end{enumerate}

Given $W \in \cate{Rep}_E(\Lgrp{G}^I)$, we define the reduced closed subscheme
\[ \Gra^{(I_1, \ldots, I_k)}_{I,W} := \Supp \EuScript{S}^{(I_1, \ldots, I_k)}_{I,W,E} \; \subset \Gra^{(I_1, \ldots, I_k)}_I . \]
In this manner, the objects of $\cate{Rep}_E(\Lgrp{G}^I)$ will serve as truncation parameters for $\Gra^{(I_1, \ldots, I_k)}_{I, W}$. For the traditional definition in terms of weights and relative positions, see \cite[Définition 1.12]{Laf18}.

When $|I|=1$ and $W$ is irreducible, $\EuScript{S}^{(I)}_{I, W, E}$ is well-known to be isomorphic to the normalized IC-complex of the stratum $\Gra^{(I)}_{I, W}$.

We move to the \emph{moduli stack of chtoucas} with level structures, whose the details can be found in \cite[\S 2, \S 12.3.2]{Laf18}. For $I = I_1 \sqcup \cdots \sqcup I_k$ and $N$ as before, $\Cht^{(I_1, \ldots, I_k)}_{N,I}$ is defined by a pull-back diagram
\[\begin{tikzcd}
	\Cht^{(I_1, \ldots, I_k)}_{N,I} \arrow[r] \arrow[d] \arrow[phantom, rd, "\Box" description] & \Hecke^{(I_1, \ldots, I_k)}_{N,I} \arrow[d, "{(\mathcal{G}_0, \mathcal{G}_k)}"] \\
	\Bun_{G,N} \arrow[r, "{ \identity \times \Frob }"'] & \Bun_{G,N} \times \Bun_{G,N}
\end{tikzcd}\]
of ind-stacks over $\F_q$. It classifies the chains
\[ (\mathcal{G}_0, \psi_0) \stackrel{\phi_1}{\dashrightarrow} \cdots \stackrel{\phi_{k-1}}{\dashrightarrow} (\mathcal{G}_k, \psi_k) \stackrel{\phi_k}{\dashrightarrow} ({}^\tau \mathcal{G}_0, {}^\tau \psi_0) \]
of $G$-torsors with $N$-level structures (cf.\ \eqref{eqn:Hecke-stack}). Here, for every $\F_q$-scheme $S$ and $(\mathcal{G}, \psi) \in \Bun_{G,N}(S)$ we set
\[ ({}^\tau \mathcal{G}, {}^\tau \psi) := (\identity_X \times \Frob_S)^* (\mathcal{G}, \psi) \]
and similarly for the morphisms in $\Bun_{G,N}(S)$. Note that $\Cht^{(I_1, \ldots, I_k)}_{N,I}$ is an ind-stack of ind-finite type over $\F_q$ endowed with a morphism of ``paws''
\[ \mathfrak{p}^{(I_1, \ldots, I_k)}_{N,I}: \Cht^{(I_1, \ldots, I_k)}_{N,I} \to (X \smallsetminus \hat{N})^I \]
coming from that of $\Hecke^{(I_1, \ldots, I_k)}_{N,I}$. Stability conditions of Harder--Narasimhan type attached to dominant coweights $\mu \in X_*(T^{\mathrm{ad}})$ of $G^{\mathrm{ad}}$ on the datum $\mathcal{G}_0$ gives rise to the truncated piece $\Cht^{(I_1, \ldots, I_k), \leq \mu}_{N,I}$. Choose any Borel subgroup (over the separable closure) of $G$ containing $T$. For coweights $\mu, \mu'$, write
\[ \mu' \geq \mu \iff \mu' - \mu \in \sum_{\check{\alpha}: \text{simple coroot}} \Q_{\geq 0} \cdot \check{\alpha}. \]
As $\mu$ grows with respect to $\geq$, we have the filtered limit
\[ \Cht^{(I_1, \ldots, I_k)}_{N,I} = \varinjlim_\mu \Cht^{(I_1, \ldots, I_k), \leq \mu}_{N,I}. \]
Exactly as in the case of affine Grassmannians, there is another truncation indexed by $W \in \cate{Rep}_E((\Lgrp{G})^I)$; see \cite[\S 2]{Laf18} for details. They give rise to
\[ \Cht^{(I_1, \ldots, I_k)}_{N,I,W} \stackrel{\text{open}}{\supset} \Cht^{(I_1, \ldots, I_k), \leq \mu}_{N,I,W}. \]
%We shall adhere to the version with $W$. Again, only the set of highest weights of $W$ matters for truncation.

By \cite[Proposition 2.6]{Laf18}, $\Cht^{(I_1, \ldots, I_k)}_{N,I,W}$ is a reduced Deligne--Mumford stack locally of finite type over $(X \smallsetminus \hat{N})^I$, for any $W$. The connected components of an open substack of the form $\Cht^{(I_1, \ldots, I_k), \leq \mu}_{N,I,W}$ are quotients of quasi-projective $(X \smallsetminus \hat{N})^I$-schemes by finite groups; when $N$ is large relative to $\mu$ and to the highest weigts of $W$, those connected components are even quasi-projective $(X \smallsetminus \hat{N})^I$-schemes. The last property can serve to justify some geometric reasoning over such stacks, by reducing them to the usual scheme-theoretic setting.

% Note that $\Cht^{(I_1, \ldots, I_k), \leq \mu}_{N,I,W} = \Cht^{(I_1, \ldots, I_k), \leq \mu}_{N,I,W^{\vee,\theta}}$ and $\Gra^{(I_1, \ldots, I_k)}_{I,W} = \Gra^{(I_1, \ldots, I_k)}_{I,W^{\vee, \theta}}$, since $W^{\vee, \theta}$ and $W$ have the same weight-multiplicities. Here $W^{\vee,\theta}$ is $W^\vee$ twisted by the Chevalley involution $\Lgrp{\theta}$ of $(\Lgrp{G})^I$.

We have $Z_G(\mathring{F}) \backslash Z_G(\A) \hookrightarrow \Bun_{Z_G, N}(\F_q)$, and the latter acts on $\Cht^{(I_1, \ldots, I_k)}_{N,I}$ by twisting $G$-torsors by $Z_G$-torsors. This action leaves each truncated piece invariant. In particular, for a lattice $\Xi \subset Z_G(\mathring{F}) \backslash Z_G(\A)$ chosen as in \S\ref{sec:cusp-forms}, we have $\Xi$-action on $\Cht^{(I_1, \ldots, I_k), \leq \mu}_{N,I,W}$, etc. One can shrink $\Xi$ to make it act freely, and consider the quotients $\Cht^{(I_1, \ldots, I_k), \leq \mu}_{N,I,W}/\Xi$, etc.

By the discussions before \cite[Définition 2.14]{Laf18}, $\Cht^{(I_1, \ldots, I_k), \leq\mu}_{N,I,W}/\Xi$ is a Deligne--Mumford stack of finite type.

\subsection{Cohomologies}\label{sec:local-models}
We keep the notation from \S\ref{sec:geometric-setup}. In what follows, normalization of perverse sheaves will always be with respect to the base $(X \smallsetminus \hat{N})^I$.

The first ingredient \cite[Proposition 2.8]{Laf18} is a canonical smooth morphism
% Originally: of relative dimension $\sum_i n_i \dim G$
\[ \epsilon^{(I_1, \ldots, I_k)}_{(I),W,\underline{n}}: \Cht^{(I_1, \ldots, I_k)}_{N,I,W} \to \Gra^{(I_1, \ldots, I_k)}_{I,W} / G_{\sum_{i \in I} n_i x_i} \]
where $\underline{n} = (n_i)_{i \in I} \in \Z_{\geq 0}^I$ is sufficiently positive with respect to $W \in \cate{Rep}_E((\Lgrp{G})^I)$, so that the $G_{\sum_i \infty x_i}$-action factors through $G_{\sum_i n_i x_i}$. Assume furthermore that $W = \boxtimes_{j=1}^k W_j$ where each $W_j \in \cate{Rep}_E((\Lgrp{G})^{I_j})$ is irreducible. In \cite[(2.5)]{Laf18} is constructed the canonical smooth morphism
\[ \epsilon^{(I_1, \ldots, I_k)}_{(I_1, \ldots, I_k), W, \underline{n}}: \Cht^{(I_1, \ldots, I_k)}_{N,I,W} \to \prod_{j=1}^k \Gra^{(I_j)}_{I_j, W_j} / G_{\sum_{i \in I_j} n_i x_i}. \]
These two are related by the canonical smooth morphism \cite[(1.12)]{Laf18}
\[ \kappa^{(I_1, \ldots, I_k)}_{I, W}: \Gra^{(I_1, \ldots, I_k)}_{I, W} \to \prod_{j=1}^k \Gra^{(I_j)}_{I_j, W_j} / G_{\sum_{i \in I_j} n_i x_i} \]
that chops a chain $\mathcal{G}_0 \dashrightarrow \mathcal{G}_1 \dashrightarrow \cdots \dashrightarrow \mathcal{G}_k$ (the trivialization forgotten) classified by $\Gra^{(I_1, \ldots, I_k)}_I$ into segments indexed by $I_j$. By \cite[(1.13)]{Laf18}, when $m_i \gg n_i$ it factorizes through a smooth
\[ \tilde{\kappa}^{(I_1, \ldots, I_k)}_{I, W}: \Gra^{(I_1, \ldots, I_k)}_{I, W} / G_{\sum_{i \in I} m_i x_i} \to \prod_{j=1}^k \Gra^{(I_j)}_{I_j, W_j} / G_{\sum_{i \in I_j} n_i x_i}. \]

For an interesting result on \emph{local models} of $\Cht^{(I_1, \ldots, I_k)}_{N,I,W}$ based on these morphisms, see \cite[Proposition 2.11]{Laf18}. However, we do not need that result in this article.

% Originally: local model, which is not used anymore.
%The result \cite[Proposition 2.11]{Laf18} on local models asserts the existence of an étale covering of $\Cht^{(I_1, \ldots, I_k)}_{N,I,W}$ consisting of $\left\{ \alpha: \mathcal{U} \to \Cht^{(I_1, \ldots, I_k)}_{N,I,W} \right\}_\alpha$, such that each $\alpha$ fits into a commutative diagram of stacks over $(X \smallsetminus \hat{N})^I$
%\begin{equation}\label{eqn:local-model} \begin{tikzcd}
%	& \mathcal{U} \arrow[ld, "\alpha:\; \text{étale}"'] \arrow[rd, "\beta:\; \text{étale}"] & \\
%	\Cht^{(I_1, \ldots, I_k)}_{N,I,W} \arrow[rd, "{\epsilon^{(I_1, \ldots, I_k)}_{(I_1, \ldots, I_k), W, \underline{n}} }"'] & & \prod_{j=1}^k \Gra^{(I_j)}_{I_j, W_j} \arrow[ld, "\prod_{j=1}^k \;\text{quotient}"] \\
%	& \prod_{j=1}^k \Gra^{(I_j)}_{I_j, W_j} / G_{\sum_{i \in I_j} n_i x_i} &
%\end{tikzcd}\end{equation}
%In other words, $\Cht^{(I_1, \ldots, I_k)}_{N,I,W}$ is étale-locally isomorphic to $\prod_{j=1}^k \Gra^{(I_j)}_{I_j, W_j}$, in a manner that is compatible with the canonical maps down to $\prod_{j=1}^k \Gra^{(I_j)}_{I_j, W_j} / G_{\sum_{i \in I_j} n_i x_i}$.

As an application, for each $W \in \cate{Rep}_E((\Lgrp{G})^I)$ we take the normalized perverse sheaf $\EuScript{S}^{(I_1, \ldots, I_k)}_{I,W,E}$ on $\Gra^{(I_1, \ldots, I_k)}_{I,W}$. Descend this complex to $\Gra^{(I_1, \ldots, I_k)}_{I,W} / G_{\sum_{i \in I} n_i x_i}$ by its equivariance given by geometric Satake. Hence on can form the complex $(\epsilon^{(I_1, \ldots, I_k)}_{I,W,\underline{n}})^* \EuScript{S}^{(I_1, \ldots, I_k)}_{I,W,E}$.

Since $\EuScript{S}^{(I_1, \ldots, I_k)}_{I,W,E}$ is ULA with respect to $(X \smallsetminus \hat{N})^I$, so is its inverse image via the smooth morphism $\epsilon^{(I_1, \ldots, I_k)}_{I,W,\underline{n}}$; see \cite[5.1.2, item 2]{BG02}. We claim that the complex $(\epsilon^{(I_1, \ldots, I_k)}_{I,W,\underline{n}})^* \EuScript{S}^{(I_1, \ldots, I_k)}_{I,W,E}$ is moreover normalized perverse on $\Cht^{(I_1, \ldots, I_k)}_{N,I,W}$ for irreducible $W = \boxtimes_{j=1}^k W_j$.

Indeed, the claim is a routine consequence of the factorization structure on $\Gra^{(I_1, \ldots, I_k)}_{I,W}$ and the ULA property, smoothness, etc.

This completes our construction when $G$ is semisimple. In general, one has to consider a lattice $\Xi$ as in \S\ref{sec:cusp-forms}. According to \cite[Remarque 1.20]{Laf18}, $\EuScript{S}^{(I_1, \ldots, I_k)}_{I,W,E}$ descends to $\Gra^{(I_1, \ldots, I_k)}_{I,W}/G^\text{ad}_{\sum_i n_i x_i}$. By the discussions after \cite[Définition 2.14]{Laf18}, $\epsilon^{(I_1, \ldots, I_k)}_{I,W,\underline{n}}$ induces
\[ \epsilon^{(I_1, \ldots, I_k), \Xi}_{N,I,W,\underline{n}}: \Cht^{(I_1, \ldots, I_k)}_{N,I,W}/\Xi \to \Gra^{(I_1, \ldots, I_k)}_{I,W}/G^\text{ad}_{\sum_i n_i x_i} \]
which is smooth of relative dimension equal to $\dim G^\text{ad}_{\sum_i n_i x_i}$. We define accordingly
\begin{equation}\label{eqn:F-sheaf}
	\EuScript{F}^{(I_1, \ldots, I_k)}_{N,I,W,\Xi,E} := (\epsilon^{(I_1, \ldots, I_k), \Xi}_{N,I,W,\underline{n}})^*  \EuScript{S}^{(I_1, \ldots, I_k)}_{I,W,E}.
\end{equation}
This is still a normalized perverse sheaf on $\Cht^{(I_1, \ldots, I_k)}_{N,I,W}/\Xi$. In \textit{loc.\ cit.}, one actually deduces that $\EuScript{F}^{(I_1, \ldots, I_k)}_{N,I,W,\Xi,E}$ is isomorphic to the normalized IC-sheaf on $\Cht^{(I_1, \ldots, I_k)}_{N,I,W}/\Xi$.

Thus far we have assumed $W = \boxtimes_{j=1}^k W_j$. A general definition, functorial in arbitrary $W \in \cate{Rep}_E((\Lgrp{G})^I)$, can be crafted by repeating the construction for $W \mapsto \EuScript{S}^{(I_1, \ldots, I_k)}_{I,W,E}$ reviewed in \S\ref{sec:geometric-setup}. The result still takes the form \eqref{eqn:F-sheaf}, except that the right-hand side is now constructed functorially in $W \in \cate{Rep}_E((\Lgrp{G})^I)$; see \cite[\S 4.5]{Laf18}.

Next, introduce the other truncation parameter $\mu$ from \S\ref{sec:geometric-setup}. The morphism of paws induces
\[ \mathfrak{p}^{(I_1, \ldots, I_k), \leq \mu}_{N,I}: \Cht^{(I_1, \ldots, I_k), \leq \mu}_{N,I,W} /\Xi \to (X \smallsetminus \hat{N})^I. \]
Recall that $\Cht^{(I_1, \ldots, I_k), \leq \mu}_{N,I,W}$ is open and $\Xi$-invariant in $\Cht^{(I_1, \ldots, I_k)}_{N,I,W}$. Define
\begin{equation}\label{eqn:H-cplx}\begin{aligned}
	\EuScript{H}^{\leq \mu, E}_{N,I,W} & := \left(\mathfrak{p}^{(I_1, \ldots, I_k), \leq \mu}_{N,I}\right)_! \; \EuScript{F}^{(I_1, \ldots, I_k)}_{N,I,W,\Xi,E}\bigg|_{\Cht^{(I_1, \ldots, I_k), \leq \mu}_{N,I,W}/\Xi}, \\
	\EuScript{H}^{i, \leq \mu, E}_{N,I,W} & := \mathrm{H}^i \,\EuScript{H}^{\leq \mu, E}_{N,I,W}, \qquad i \in \Z;
\end{aligned}\end{equation}
here $\mathrm{H}^i$ is taken with respect to the ordinary $t$-structure on $\cate{D}_c((X \smallsetminus \hat{N})^I, E)$.
\begin{itemize}
	\item By using the forgetful morphisms as in \cite[Construction 2.7 + Corollaire 2.18]{Laf18}, these complexes are seen to be independent of the partition $(I_1, \ldots, I_k)$. The notation in \eqref{eqn:H-cplx} is thus justified.
	\item For $\mu \leq \mu'$, the open immersion $j: \Cht^{(I_1, \ldots, I_k), \leq \mu}_{N,I,W}/\Xi \to \Cht^{(I_1, \ldots, I_k), \leq \mu'}_{N,I,W}/\Xi$ induces a canonical arrow $\EuScript{H}^{\leq \mu, E}_{N,I,W} \to \EuScript{H}^{\leq \mu', E}_{N,I,W}$. This is a standard consequence of the formalism of six operations as $j^* = j^!$.
	\item They also respect the \emph{coalescence} of paws with respect to any map $\zeta: I \to J$. We refer to \cite[Proposition 4.12]{Laf18} for further explanations.
\end{itemize}

Let $I$ be a finite set and $W \in \cate{Rep}_E((\Lgrp{G})^I)$ arbitrary. Denote the generic point of $X$ (resp.\ $X^I$) by $\eta$ (resp.\ $\eta^I$) and choose geometric points over them
\[ \overline{\eta} \to \eta, \quad \overline{\eta^I} \to \eta^I. \]
Let $\Delta: X \to X^I$ be the diagonal embedding. Following \cite[\S 8]{Laf18} or \cite[\S 1.3]{Var07}, we choose an \emph{arrow of specialization}
\[ \mathfrak{sp}: \overline{\eta^I} \to \Delta(\overline{\eta}), \]
i.e.\ a morphism $(X^I)_{(\overline{\eta^I})} \to (X^I)_{(\Delta(\overline{\eta}))}$ or equivalently $\overline{\eta^I} \to (X^I)_{(\Delta(\overline{\eta}))}$, where the subscripts indicate strict Henselizations at the corresponding geometric points. By \cite[Proposition 8.24]{Laf18}, the induced pull-back morphism
\[ \mathfrak{sp}^*: \varinjlim_\mu \EuScript{H}^{0, \leq\mu, E}_{N,I,W} \bigg|_{\Delta(\overline{\eta})} \to \varinjlim_\mu \EuScript{H}^{0, \leq\mu, E}_{N,I,W} \bigg|_{\overline{\eta^I}} \]
between $E$-vector spaces is injective.

Now comes the Hecke action. Let $f \in C_c(K_N \backslash G(\A)/K_N; E)$. According to \cite[Corollaire 6.5]{Laf18}, taking a coweight $\kappa \gg 0$ with respect to $f$, there is an induced morphism
\begin{equation}\label{eqn:Hecke-action}
	T(f): \EuScript{H}^{\leq \mu, E}_{N,I,W} \to \EuScript{H}^{\leq \mu + \kappa, E}_{N,I,W}
\end{equation}
in $\cate{D}^b_c((X \smallsetminus \hat{N})^I, E)$, with various compatibilities. It is $E$-linear in $f$ and satisfies $T(ff') = T(f)T(f')$. After passing to $\varinjlim_\mu$, we are led to the left $C_c(K_N \backslash G(\A)/K_N; E)$-module
\begin{equation}\label{eqn:Hf}
	H_{I,W} := \left( \varinjlim_\mu \EuScript{H}^{0, \leq\mu, E}_{N,I,W} \bigg|_{\Delta(\overline{\eta})} \right)^{\mathrm{Hf}}
\end{equation}
where ``Hf'' means Hecke-finite with respect to the action \eqref{eqn:Hecke-action}. This definition is clearly functorial in $W$. The following properties are established in \cite[\S\S 8---9]{Laf18}.
\begin{itemize}
	\item Compatibility with coalescence of paws. Namely, every map $\zeta: J \to I$ induces a canonical isomorphism $\chi_\zeta: H_{I,W} \rightiso H_{J,W^\zeta}$, where $W^\zeta \in \cate{Rep}_E(\Lgrp{G}^J)$ denotes the pull-back of $W$ via $\zeta$.
	\item The arrow $\mathfrak{sp}^*$ commutes with Hecke action since the latter is defined on the level of $\cate{D}^b_c((X \smallsetminus \hat{N})^I, E)$. Moreover, it induces an isomorphism
	\[ \mathfrak{sp}^*: H_{I,W} \rightiso \left( \varinjlim_\mu \EuScript{H}^{0, \leq\mu, E}_{N,I,W} \bigg|_{\overline{\eta^I}} \right)^{\mathrm{Hf}}. \]
	\item We have $\Lgrp{G}^\emptyset = \{1\}$, $\eta^\emptyset = \Spec\F_q$ when $I = \emptyset$. There are natural isomorphisms
	\begin{equation}\label{eqn:H-vs-Ccusp}
		H_{\{0\}, \mathbf{1}} \stackrel{\chi}{\leftiso} H_{\emptyset, \mathbf{1}} \rightiso \Ccusp\left(\Bun_{G,N}(\F_q)/\Xi; E \right).
	\end{equation}
	The arrow $\chi$ is induced by coalescence via the unique map $\emptyset \to \{0\}$. The rightward arrow stems from the fact \cite[Proposition 2.16 (c)]{Var04} that $\Cht_{N, \emptyset, \mathbf{1}}/\Xi$ is the constant stack $\Bun_{G,N}(\F_q)/\Xi$ over $\Spec\F_q$, which implies a canonical isomorphism
	\begin{equation}\label{eqn:H-vs-Cc}
		\varinjlim_\mu \EuScript{H}^{0, \leq\mu, E}_{N, \emptyset, \mathbf{1}} \bigg|_{\Delta(\overline{\eta})} \rightiso C_c\left(\Bun_{G,N}(\F_q)/\Xi; E \right)
	\end{equation}
	of $C_c(K_N \backslash G(\A)/K_N; E)$-modules.
	\item For $I = \{1\}$, $W = \mathbf{1}$, coalescence induces $\Cht^{(\{0\})}_{N, \{1\}, \mathbf{1}}/\Xi \rightiso (\Cht_{N, \emptyset, \mathbf{1}}/\Xi) \dtimes{\Spec\F_q} (X \smallsetminus \hat{N})$ by \cite[(8.4)]{Laf18}. In this case, $\mathcal{H}^{0, \leq\mu, E}_{N, \{0\}, \mathbf{1}}$ is a constant sheaf and the $\varinjlim_\mu$ of its stalk at $\overline{\eta}$ is still $C_c\left(\Bun_{G,N}(\F_q)/\Xi; E \right)$.

	\item Via these isomorphisms, the $C_c(K_N \backslash G(\A)/K_N; E)$-module structures on $H_{\emptyset, \mathbf{1}}$ and $H_{\{0\}, \mathbf{1}}$ match the one on $\Ccusp\left(\Bun_{G,N}(\F_q)/\Xi; E \right)$ recorded in \S\ref{sec:cusp-forms}. See \cite[\S 8]{Laf18}.
\end{itemize}
The last item above is how harmonic analysis enters the geometric picture.

\subsection{Partial Frobenius morphisms and Galois actions}\label{sec:partial-Frobenius}
We conserve the previous conventions and review the partial Frobenius morphisms. Let $J \subset I$ be finite sets. Choose a partition $I = I_1 \sqcup \cdots \sqcup I_k$ with $I_1 = J$, together with a specialization arrow $\mathfrak{sp}: \overline{\eta^I} \to \Delta(\overline{\eta})$. The choice of partition intervenes in the constructions, but will disappear in the final results.

Let $\Frob_J = \Frob_{I_1}: (X \smallsetminus \hat{N})^I \to (X \smallsetminus \hat{N})^I$ be the morphism that equals $\Frob$ on the coordinates indexed by $I_1$, and $\identity$ elsewhere.

Take $W \in \cate{Rep}_E((\Lgrp{G})^I)$ as well the lattice $\Xi$ as in \S\ref{sec:local-models}. In \cite[\S 3]{Laf18} is defined the partial Frobenius morphism
\begin{equation}\label{eqn:partial-Frob-Cht}
	\Frob^{(I_1, \ldots, I_k)}_{I_1, N}: \Cht^{(I_1, \ldots, I_k)}_{N,I,W} \to \Cht^{(I_2, \ldots, I_k, I_1)}_{N,I,W}
\end{equation}
covering $\Frob_{I_1}$, that respects $\Xi$-actions. In terms of the notations in \S\ref{sec:geometric-setup}, it sends the chain
\[ (\mathcal{G}_0, \psi_0) \stackrel{\phi_1}{\dashrightarrow} \cdots \dashrightarrow (\mathcal{G}_k, \psi_k) \dashrightarrow ({}^\tau \mathcal{G}_0, {}^\tau \psi_0) \]
into
\[ (\mathcal{G}_1, \psi_1) \stackrel{\phi_2}{\dashrightarrow} \cdots \dashrightarrow (\mathcal{G}_k, \psi_k) \dashrightarrow ({}^\tau \mathcal{G}_0, {}^\tau \psi_0) \stackrel{{}^\tau \phi_1}{\dashrightarrow} ({}^\tau \mathcal{G}_1, {}^\tau \psi_1) \]
whereas the paws are transformed accordingly by $\Frob_{I_1}$. The cyclic composition of $k$ partial Frobenius morphism equals the total Frobenius endomorphism of $\Cht^{(I_1, \ldots, I_k)}_{N,I,W}$. An easy consequence is that $\Frob_{I_1}$ is a \emph{universal homeomorphism}; see \cite[Tag 04DC]{stacks-project}.

The induced morphism between the quotients by $\Xi$ is also named $\Frob^{(I_1, \ldots, I_k)}_{I_1, N}$. Now introduce the dominant coweight $\mu$ of $G^\mathrm{ad}$ in \S\ref{sec:local-models} as truncation parameter. A basic fact is that whenever $\mu' \gg \mu$ with respect to $W$, 
\begin{equation}\label{eqn:mu-mu'}
	\left(\Frob^{(I_1, \ldots, I_k)}_{I_1, N}\right)^{-1} \Cht^{(I_2, \ldots, I_k, I_1), \leq \mu}_{N,I,W} \subset \Cht^{(I_1, \ldots, I_k), \leq \mu'}_{N,I,W}.
\end{equation}

When $k=1$, we have the usual Frobenius correspondence $\Phi: \Frob^* \EuScript{S}^{(I)}_{I,W,E} \rightiso \EuScript{S}^{(I)}_{I,W,E}$ between normalized perverse sheaves on $\Gra^{(I)}_{I,W}/G^\mathrm{ad}_{\sum_i n_i x_i}$. In general, by writing
\[ \Frob_{I_1}(x_i)_{i \in I} = (x'_i)_{i \in I}, \quad (x_i)_{i \in I} \in (X \smallsetminus \hat{N})^I(S), \quad \forall S: \F_q\text{-scheme} \]
and supposing $W = \boxtimes_{j=1}^k W_j$ is irreducible, there is a commutative diagram:
\[\begin{tikzcd}[row sep=large]
	\Cht^{(I_1, \ldots, I_k)}_{N,I,W}/\Xi \arrow[r, "{\Frob^{(I_1, \ldots, I_k)}_{I_1, N}}"] \arrow[d, "{\epsilon^{(I_1, \ldots, I_k)}_{N, (I_1, \ldots, I_k), \underline{n}} }"'] & \Cht^{(I_2, \ldots, I_k, I_1)}_{N,I,W}/\Xi \arrow[d, "{\epsilon^{(I_2, \ldots, I_k, I_1)}_{N, (I_2, \ldots, I_k, I_1), \underline{n}} }"] \\
	\prod_{j=1}^k \Gra^{(I_j)}_{I_j, W_j}/G^\mathrm{ad}_{\sum_{i \in I_j} n_i x_i } \arrow[r, "{\Frob \times \identity \times \cdots \times \identity}"' inner sep=1.5em] & \prod_{j=1}^k \Gra^{(I_j)}_{I_j, W_j}/G^\mathrm{ad}_{\sum_{i \in I_j} n_i x'_i }
\end{tikzcd}\]
In view of the constructions in \S\ref{sec:local-models} using the smooth morphisms $\epsilon^{\cdots}_{\cdots}$, the ULA property, etc., we obtain a canonical isomorphism in $\cate{D}^b_c(\Cht^{(I_1, \ldots, I_k)}_{N,I,W}/\Xi, E)$
\begin{equation}\label{eqn:partial-Frob-F}
	F^{(I_1, \ldots, I_k)}_{I_1, N, W}: \left( \Frob^{(I_1, \ldots, I_k)}_{I_1, N} \right)^* \EuScript{F}^{(I_2, \ldots, I_1)}_{N,I,W,\Xi,E} \rightiso \EuScript{F}^{(I_1, \ldots, I_k)}_{N,I,W,\Xi,E}
\end{equation}
extending the previous case $k=1$. Cf.\ \cite[Proposition 3.4]{Laf18}. This isomorphism can be extended functorially to arbitrary $W \in \cate{Rep}_E((\Lgrp{G})^I)$ by repeating the construction for $W \mapsto \EuScript{S}^{(I_1, \ldots, I_k)}_{I,W,E}$.

Abbreviate the $\Frob^{(I_1, \ldots, I_k)}_{I_1, N}$ on $\Cht^{(I_1, \ldots, I_k)}_{N,I,W}/\Xi$ as $a_1$. It fits into the commutative diagram
\begin{equation}\label{eqn:corr-diagram}\begin{tikzcd}
	\Cht^{(I_2, \ldots, I_1), \leq \mu}_{N,I,W}/\Xi \arrow[d, "\mathfrak{p}"] & a_1^{-1}\left( \Cht^{(I_2, \ldots, I_1), \leq \mu}_{N,I,W}/\Xi \right) \arrow[l, "a_1"'] \arrow[hookrightarrow, r, "a_2"] \arrow[d, "\mathfrak{p}"] & \Cht^{(I_1, \ldots, I_k), \leq \mu'}_{N,I,W}/\Xi \arrow[d, "\mathfrak{p}"] \\
	(X \smallsetminus \hat{N})^I & (X \smallsetminus \hat{N})^I \arrow[l, "{\Frob_{I_1}}"] \arrow[equal, r] & (X \smallsetminus \hat{N})^I
\end{tikzcd}\end{equation}
where $\mathfrak{p}$ denotes the self-evident morphisms of paws, $a_1$ is a universal homeomorphism and $a_2$ is an open immersion. Hence \eqref{eqn:partial-Frob-F} affords a \emph{cohomological correspondence} between bounded constructible complexes in the sense of \cite[\S 1]{Var07}: for $a_1$, $a_2$ in \eqref{eqn:corr-diagram},
\begin{equation}\label{eqn:F-a1-a2}
	\Frob^{(I_1, \ldots, I_k)}_{I_1, N, W}: a_1^* \underbracket{ \EuScript{F}^{(I_2, \ldots, I_1)}_{N,I,W,\Xi,E} }_{\text{on the $\leq \mu$ part}} \to a_2^! \underbracket{ \EuScript{F}^{(I_1, \ldots, I_k)}_{N,I,W,\Xi,E} }_{\text{on the $\leq \mu'$ part}} , \quad a_2^* = a_2^! .
\end{equation}

The left square of \eqref{eqn:corr-diagram} is not Cartesian; however, in the commutative diagram defined with Cartesian square
\[\begin{tikzcd}
	& & a_1^{-1}\left( \Cht^{(I_2, \ldots, I_1), \leq \mu}_{N,I,W}/\Xi \right) \arrow[ld, "\exists! \, \varphi"] \arrow[bend right=10, lld, "a_1"'] \arrow[bend left, ldd, "\mathfrak{p}"] \\
	\Cht^{(I_2, \ldots, I_1), \leq \mu}_{N,I,W}/\Xi \arrow[d, "\mathfrak{p}"'] & \Frob_{I_1}^*\left( \Cht^{(I_2, \ldots, I_1), \leq \mu}_{N,I,W}/\Xi \right) \arrow[l, "{\tilde{a}_1}"'] \arrow[d, "\tilde{\mathfrak{p}}"] & \\
	(X \smallsetminus \hat{N})^I \arrow[phantom, ur, "\Box" description] & (X \smallsetminus \hat{N})^I \arrow[l, "{\Frob_{I_1}}"] &
\end{tikzcd}\]
the arrow $\varphi$ is a universal homeomorphism since both $\Frob_{I_1}$ and $a_1$ are. Therefore we obtain
\begin{equation}\label{eqn:BC}
	\mathrm{BC}: \Frob_{I_1}^* \mathfrak{p}_! \xrightarrow[\mathrm{bc}]{\sim} \tilde{\mathfrak{p}}_! \tilde{a}_1^* \leftiso \tilde{\mathfrak{p}}_! \varphi_! \varphi^* \tilde{a}_1^* \simeq \mathfrak{p}_! a_1^*.
\end{equation}
Indeed, $\mathrm{bc}$ is the isomorphism of base change by the universal homeomorphism $\Frob_{I_1}$ \cite[12.2]{LO2}; the second isomorphism is induced by $\varphi_! \varphi^* \rightiso \identity$, which is in turn due to the topological invariance of the étale topos (see \cite[Exp VIII, Théorème 1.1]{SGA4} or \cite[Tag 04DY]{stacks-project}) under the universal homeomorphism $\varphi$.

In view of \eqref{eqn:H-cplx}, we can now define
\begin{equation}\label{eqn:F_J}
	F_J = F_{I_1}: \Frob_{I_1}^* \EuScript{H}^{\leq \mu, E}_{N,I,W} \to \EuScript{H}^{\leq \mu', E}_{N,I,W}
\end{equation}
as the composite in $\cate{D}^b_c((X \smallsetminus \hat{N})^I, E)$
\[\begin{tikzcd}
	\Frob_{I_1}^* \mathfrak{p}_! \underbracket{ \EuScript{F}^{(I_2, \ldots, I_1)}_{N,I,W,\Xi,E} }_{\text{on $\leq \mu$}} \arrow[r, "\sim"', "\mathrm{BC}"] & \mathfrak{p}_! a_1^* \underbracket{ \EuScript{F}^{(I_2, \ldots, I_1)}_{N,I,W,\Xi,E} }_{\text{on $\leq \mu$}} \arrow[r, "{\mathfrak{p}_! \Frob^{(I_1, \ldots, I_k)}_{I_1,N,W}}" inner sep=1em] & \mathfrak{p}_! a_2^! \underbracket{ \EuScript{F}^{(I_1, \ldots, I_k)}_{N,I,W,\Xi,E} }_{\text{on $\leq \mu'$}} \arrow[r] & \mathfrak{p}_! \underbracket{ \EuScript{F}^{(I_1, \ldots, I_k)}_{N,I,W,\Xi,E}}_{\text{on $\leq \mu'$}}
\end{tikzcd}\]
the last arrow arising from $\mathfrak{p}_! a_2^! = \mathfrak{p}_! (a_2)_! a_2^! \xrightarrow{(a_2)_! a_2^! \to \identity} \mathfrak{p}_!$. It is functorial in $W \in \cate{Rep}_E((\Lgrp{G})^I)$ and is shown to be compatible with the coalescence of paws in \cite[\S\S 3---4]{Laf18}. Hence the dependence is only on $J \subset I$.

Consequently, the morphism $F_J$ also acts $E$-linearly on $\varinjlim_\mu \EuScript{H}^{0, \leq\mu, E}_{N,I,W} \bigg|_{\overline{\eta^I}}$. Given any partition $I = I_1 \sqcup \ldots \sqcup I_k$, the actions of $F_{I_1}, \ldots, F_{I_k}$ form a commuting family whose cyclic composition equals the total Frobenius action on $\varinjlim_\mu \EuScript{H}^{0, \leq\mu, E}_{N,I,W} \bigg|_{\overline{\eta^I}}$.

On the other hand, the standard theory \cite[Tag 03QW]{stacks-project} yields a continuous representation of $\pi_1(\eta^I, \overline{\eta^I})$ on $\EuScript{H}^{0, \leq\mu, E}_{N,I,W} \bigg|_{\overline{\eta^I}}$ which passes to $\varinjlim_\mu$.

To conclude this subsection, we recall briefly the following extension of groups
\[ 1 \to \Ker\left[ \pi_1(\eta^I, \overline{\eta^I}) \to \hat{\Z} \right] \to \mathrm{FWeil}(\eta^I, \overline{\eta^I}) \to \Z^I \to 0. \]
We refer to \cite[Remarque 8.18]{Laf18} and the subsequent discussions for all further details. When $|I|=1$, it becomes the Weil group $\Weil{\mathring{F}}$ of $\mathring{F} = \F_q(X)$; in general there is a surjection $\mathrm{FWeil}(\eta^I, \overline{\eta^I}) \twoheadrightarrow \Weil{\mathring{F}}^I$ depending on the choice of $\mathfrak{sp}$. The surjection induces an isomorphism from the pro-finite completion $\mathrm{FWeil}(\eta^I, \overline{\eta^I})$ to that of $\Weil{\mathring{F}}^I$, i.e.\ $\pi_1(\eta, \overline{\eta})^I$. As mentioned in \textit{loc.\ cit.}, the action on $\varinjlim_\mu \EuScript{H}^{0, \leq\mu, E}_{N,I,W} \bigg|_{\overline{\eta^I}}$ of
\begin{compactitem}
	\item the partial Frobenius morphisms $F_J$ for various $J \subset I$, and
	\item that of $\pi_1(\eta^I, \overline{\eta^I})$
\end{compactitem}
meld into an action of $\mathrm{FWeil}(\eta^I, \overline{\eta^I})$. The upshot of the \S 8 of \textit{loc.\ cit.} is to produce a continuous $E$-linear $\pi_1(\eta, \overline{\eta})^I$-action on $\left(\varinjlim_\mu \EuScript{H}^{0, \leq\mu, E}_{N,I,W} \bigg|_{\overline{\eta^I}} \right)^{\text{Hf}}$ therefrom. In other words, one wants to factorize the $\mathrm{FWeil}(\eta^I, \overline{\eta^I})$-action through its pro-finite completion continuously.

The key for the passage to $\pi_1(\eta, \overline{\eta})^I$-action is \emph{Drinfeld's Lemma}. This method requires some finiteness conditions which in turn involve the Eichler--Shimura relations. These important issues are addressed at length in \cite[\S 8]{Laf18}, but they are not needed in this article.

The aforementioned continuous representation transports to $H_{I,W}$, namely
\[ \vec{\gamma} \cdot f := (\mathfrak{sp}^*)^{-1} \left( \vec{\gamma} \cdot (\mathfrak{sp}^* f) \right), \quad \vec{\gamma} \in \pi_1(\eta, \overline{\eta})^I, \; f \in H_{I,W}. \]
This action turns out to be independent of the choice of $\overline{\eta^I}$ and $\mathfrak{sp}$, by \cite[Lemme 9.4]{Laf18}.

\subsection{Excursion operators and pseudo-characters}\label{sec:excursion-operator}
Consider finite sets $I, J$ and $W \in \cate{Rep}_E((\Lgrp{G})^I)$, $U \in \cate{Rep}_E((\Lgrp{G})^J)$. Let $\zeta_I, \zeta_J$ be the unique maps from $I,J$ into the singleton $\{0\}$.
The diagonal action on $W$ gives $W^{\zeta_I} \in \cate{Rep}_E(\Lgrp{G})$; the space of $\hat{G}$-invariants $(W^{\zeta_I})^{\hat{G}}$ is therefore a representation of $\Gal(\widetilde{F}|F)$. Denote by $(W^{\zeta_I})^{\hat{G}} \big|_{X \smallsetminus \hat{N}}$ the $E$-lisse sheaf on $X \smallsetminus \hat{N}$ obtained by descent. Likewise, we have $(W^{\zeta_I})_{\hat{G}} \big|_{X \smallsetminus \hat{N}}$ by taking the maximal quotient of $W^{\zeta_I}$ on which $\hat{G}$ acts trivially. In \cite[(12.18), (12.19)]{Laf18} is constructed a pair of morphisms
\[\begin{tikzcd}[row sep=tiny]
	\EuScript{H}^{\leq \mu, E}_{N,J,U} \boxtimes (W^{\zeta_I})^{\hat{G}} \big|_{X \smallsetminus \hat{N}} \arrow[r] & \EuScript{H}^{\leq \mu, E}_{N, J \sqcup I, U \boxtimes W}\big|_{(X \smallsetminus \hat{N})^J \times \Delta(X \smallsetminus \hat{N})} \\
	\EuScript{H}^{\leq \mu, E}_{N,J,U} \boxtimes (W^{\zeta_I})_{\hat{G}} \big|_{X \smallsetminus \hat{N}} & \EuScript{H}^{\leq \mu, E}_{N, J \sqcup I, U \boxtimes W}\big|_{(X \smallsetminus \hat{N})^J \times \Delta(X \smallsetminus \hat{N})} \arrow[l]
\end{tikzcd}\]
in $\cate{D}^b_c\left((X \smallsetminus \hat{N})^{J \sqcup \{0\}}, E\right)$. Roughly speaking, they are defined via coalescence and the functoriality of $\EuScript{H}$ with respect to $(W^{\zeta_I})^{\hat{G}} \hookrightarrow W^{\zeta_I} \twoheadrightarrow (W^{\zeta_I})_{\hat{G}}$.

% % % % % % % % % % % % % % % % % %
%For all
%\begin{itemize}
%	\item $x \in W$ invariant under the diagonal action of $\hat{G}$, i.e.\ a morphism $\mathbf{1} \to W^{\zeta_I}$ in $\cate{Rep}_E(\Lgrp{G})$ (recall the discussion in \S\ref{sec:local-models} on coalescence);
%	\item $\xi \in W^\vee$ invariant under the diagonal action of $\hat{G}$, i.e.\ a morphism $W^{\zeta_I} \to \mathbf{1}$ in $\cate{Rep}_E(\Lgrp{G})$,
%\end{itemize}
%the corresponding \emph{creation} and \emph{annihilation} operators \cite[Définitions 5.1, 5.2 and 12.3.4]{Laf18} are a pair of arrows
%\[\begin{tikzcd}
%	\EuScript{H}^{\leq \mu, E}_{N,J,U} \boxtimes E_{X \smallsetminus \hat{N}} \arrow[r, yshift=0.5em, "{\EuScript{C}^\sharp_x}"] & \EuScript{H}^{\leq \mu, E}_{N, J \sqcup I, U \boxtimes W}\big|_{(X \smallsetminus \hat{N})^J \times \Delta(X \smallsetminus \hat{N})} \arrow[l, yshift=-0.5em, "{\EuScript{C}^\flat_\xi}"]
%\end{tikzcd}\]
%in $\cate{D}^b_c\left((X \smallsetminus \hat{N})^{J \sqcup \{0\}}, E\right)$; roughly speaking they are defined in terms of coalescence and the morphisms induced from $x,\xi$ by functoriality. % Note that the invariance and co-invariance can be interchanged since $\mathrm{char}(E) = 0$.
%% % % % % % % % % % % % % % % % % % %

Now take $J = \emptyset$ and $U = \mathbf{1}$. Let $x \in W$ (resp.\ $\xi \in W^\vee$) be $\hat{G}$-invariant under the diagonal action, viewed as a map $E \to (W^{\zeta_I})^{\hat{G}}$ (resp. $(W^{\zeta_I})_{\hat{G}} \to E$). Taking $\varinjlim_\mu \mathrm{H}^0(\cdots \big|_{\overline{\eta}})$ yields the creation and annihilation operators (see \cite[Définitions 5.1, 5.2 and 12.3.4]{Laf18})
\[\begin{tikzcd}
	\varinjlim_\mu \EuScript{H}^{0, \leq\mu, E}_{N,\emptyset,\mathbf{1}} \arrow[r, yshift=0.5em, "{\EuScript{C}^\sharp_x}"] & \varinjlim_\mu \EuScript{H}^{0, \leq \mu, E}_{N,I,W}\bigg|_{\Delta(\overline{\eta})} \arrow[l, yshift=-0.5em, "{\EuScript{C}^\flat_\xi}"]
\end{tikzcd}\]
between $E$-vector spaces. Restriction to Hecke-finite parts yields arrows
\[\begin{tikzcd}
	H_{\emptyset, \mathbf{1}} \simeq H_{\{0\}, \mathbf{1}} \arrow[r, yshift=0.5em, "{\EuScript{C}^\sharp_x}"] & H_{I,W} \arrow[l, yshift=-0.5em, "{\EuScript{C}^\flat_\xi}"]
\end{tikzcd}, \qquad \text{cf.\ \eqref{eqn:Hf}.} \]

Given $I, W, x, \xi$ as above and $\vec{\gamma} = (\gamma_i)_{i \in I} \in \pi_1(\eta, \overline{\eta})^I$, the \emph{excursion operator} $S_{I,W,x,\xi,\vec{\gamma}}$ is the composite
\[\begin{tikzcd}[row sep=small]
	H_{\{0\}, \mathbf{1}} \arrow[d, "{\EuScript{C}^\sharp_x}"'] & & & H_{\{0\}, \mathbf{1}} \\
	H_{I,W} \arrow[r, "{\mathfrak{sp}^*}", "\sim"'] & \left( \varinjlim_\mu \EuScript{H}^{0, \leq \mu, E}_{N,I,W} \bigg|_{\overline{\eta^I}} \right)^\text{Hf} \arrow[r, "\vec{\gamma}", "\sim"'] & \left( \varinjlim_\mu \EuScript{H}^{0, \leq \mu, E}_{N,I,W} \bigg|_{\overline{\eta^I}} \right)^\text{Hf} \arrow[r, "{(\mathfrak{sp}^*)^{-1}}", "\sim"'] & H_{I,W} \arrow[u, "{\EuScript{C}^\flat_\xi}"']
\end{tikzcd}\]
Here $\vec{\gamma}$ acts in the manner reviewed in \S\ref{sec:partial-Frobenius}. Upon recalling \eqref{eqn:H-vs-Ccusp}, we obtain
\[ S_{I,W,x,\xi,\vec{\gamma}} \in \End_E\left( H_{\{0\}, \mathbf{1}} \right) \simeq \End_E\left( \Ccusp(\Bun_{G,N}(\F_q)/\Xi; E)\right). \]
Moreover, by \cite[Définition-Proposition 9.1]{Laf18}
\begin{itemize}
	\item we have $S_{I,W,x,\xi,\vec{\gamma}} \in \End_{C_c(K_N \backslash G(\A)/K_N; E)}\left(H_{\{0\}, \mathbf{1}}\right)$;
	\item the formation of $S_{I,W,x,\xi,\vec{\gamma}}$ is $E$-bilinear in $x, \xi$ and continuous in $\vec{\gamma}$ for the topology on the finite-dimensional space $\End_E\left( H_{\{0\}, \mathbf{1}} \right)$ induced by $E$;
	\item let $\mathcal{B}_E$ be the $E$-subalgebra of $\End_{C_c(K_N \backslash G(\A)/K_N; E)}\left(H_{\{0\}, \mathbf{1}}\right)$ generated by $S_{I,W,x,\xi,\vec{\gamma}}$ for all quintuples $(I,W,x,\xi,\vec{\gamma})$. Then $\mathcal{B}_E$ is a finite-dimensional commutative $E$-algebra by \cite[(10.2)]{Laf18}.
\end{itemize}

The foregoing constructions behave well under finite extensions of the field $E$ of coefficients. Define the $\overline{\Q_\ell}$-algebra
\[ \mathcal{B} := \mathcal{B}_E \otimes_E \overline{\Q_\ell} \; \subset \End_{\overline{\Q_\ell}}\left( H_{\{0\}, \mathbf{1}} \otimes_E \overline{\Q_\ell} \right) . \]
Upon enlarging $E$, we may assume that all homomorphisms $\nu: \mathcal{B} \to \overline{\Q_\ell}$ (finitely many) of $\overline{\Q_\ell}$-algebras are defined over $E$. There is a decomposition of $C_c(K_N \backslash G(\A)/K_N; E)$-modules into generalized eigenspaces
\begin{equation}\label{eqn:nu-decomp}
	H_{\{0\}, \mathbf{1}} = \bigoplus_\nu \mathfrak{H}_\nu, \quad \mathfrak{H}_\nu := \left\{ \begin{array}{rl}
		f \in H_{\{0\}, \mathbf{1}} : & \forall T \in \mathcal{B}_E, \; \exists d \geq 1 \\
		& \text{s.t.}\; (T - \nu(T))^d f = 0
	\end{array}\right\}.
\end{equation}
Here $\nu$ ranges over the characters of $\mathcal{B}$, and one may take $d = \dim_E H_{\{0\}, \mathbf{1}}$. The same holds after passing to $\overline{\Q_\ell}$. All in all,
\[ \Ccusp\left( \Bun_{G,N}(\F_q)/\Xi; \overline{\Q_\ell} \right) = \bigoplus_{\nu: \mathcal{B} \to \overline{\Q_\ell}} \mathfrak{H}_\nu \quad \text{in} \quad C_c(K_N \backslash G(\A)/K_N; \overline{\Q_\ell})\dcate{Mod}. \]
It is conjectured that $\mathcal{B}$ is reduced, which will imply that $d=1$ suffices.

The next step is to re-encode the excursion operators $S_{I,W,x,\xi,\vec{\gamma}}$. Let $f(\vec{g}) = \lrangle{\xi, \vec{g} \cdot x}_{W^\vee \otimes W}$ where $\vec{g} \in (\Lgrp{G})^I$ and $\lrangle{\cdot, \cdot}_{W^\vee \otimes W}$ is the duality pairing $W^\vee \otimes_E W \to E$. Then $f \in \mathscr{O}(\hat{G} \bbslash (\Lgrp{G})^I \sslash \hat{G})$, where $\hat{G}$ acts by bilateral translations through diagonal embedding. By \cite[Lemme 10.6]{Laf18}, $S_{I,W,x,\xi,\vec{\gamma}}$ depends only on $(I, f, \vec{\gamma})$. Using some algebraic version of the Peter--Weyl theorem, one can uniquely define the operators
\[ S_{I,f,\vec{\gamma}} \; \in \mathcal{B}_E , \quad f \in \mathscr{O}(\hat{G} \bbslash (\Lgrp{G})^I \sslash \hat{G}), \]
in a manner compatible with the original $S_{I,W,x,\xi,\vec{\gamma}}$, such that if $f$ comes from a function $\Gal(\widetilde{\mathring{F}}|\mathring{F})^I \to E$, then $S_{I,f,\vec{\gamma}} = f(\vec{\gamma}) \cdot \identity$. See \cite[Remarque 12.20]{Laf18} for further explanations.

Take $n \in \Z_{\geq 1}$ and $I := \{0, \ldots, n\}$. Then $\hat{G}$ acts on $(\Lgrp{G})^n$ by simultaneous conjugation. There is a natural map
\begin{equation}\label{eqn:f-tilde} \begin{aligned}
	\mathscr{O}((\Lgrp{G})^n \sslash \hat{G}) & \longrightarrow \mathscr{O}\left(\Lgrp{G} \bbslash (\Lgrp{G})^{\{0, \ldots, n\}} \sslash \hat{G}\right) \subset \mathscr{O}\left(\hat{G} \bbslash (\Lgrp{G})^{\{0, \ldots, n\}} \sslash \hat{G}\right) \\
	f & \longmapsto \left[ \tilde{f}: (g_0, \ldots, g_n) \mapsto f(g_0^{-1} g_1, \ldots, g_0^{-1} g_n) \right].
\end{aligned}\end{equation}

When $n$ is fixed, the operators
\begin{equation}\label{eqn:Theta-operator}
	\Theta_n(f)\left( \vec{\gamma} \right) := S_{\{0, \ldots, n\}, \tilde{f}, (1, \vec{\gamma})}, \quad n \in \Z_{\geq 1}, \; \vec{\gamma} \in \pi_1(\eta, \overline{\eta})^n, \; f \in \mathscr{O}((\Lgrp{G})^n \sslash \hat{G})
\end{equation}
in $\mathcal{B}_E$ afford a homomorphism $\mathscr{O}((\Lgrp{G})^n \sslash \hat{G}) \to C\left( \pi_1(X \smallsetminus \hat{N}, \overline{\eta})^n, \mathcal{B}_E \right)$ between $E$-algebras, where $C(\cdots)$ denotes the algebra of continuous functions under pointwise operations. See \cite[Proposition 10.10]{Laf18} for the passage to $\pi_1(X \smallsetminus \hat{N}, \overline{\eta})$.

Since $\mathscr{O}\left(\Lgrp{G} \bbslash (\Lgrp{G})^{\{0, \ldots, n\}} \sslash \hat{G}\right) \subsetneq \mathscr{O}\left(\hat{G} \bbslash (\Lgrp{G})^{\{0, \ldots, n\}} \sslash \hat{G}\right)$ in general, the map \eqref{eqn:f-tilde} is not always surjective. Nonetheless, the operators $S_{\{0, \ldots, n\}, \tilde{f}, (1, \vec{\gamma})}$ still generate $\mathcal{B}_E$ as $n, f, \vec{\gamma}$ vary: see \cite[Remarque 12.20]{Laf18}.

Finally, the machinery of $\Lgrp{G}$-\emph{pseudo-characters} associates a semisimple $L$-parameter $\sigma \in \Phi(G)$ to any character $\nu: \mathcal{B} \to \overline{\Q_\ell}$, characterized as follows.
\begin{itemize}
	\item Version 1: for all $n \in \Z_{\geq 1}$, $\vec{\gamma} = (\gamma_1, \ldots, \gamma_n)$ and $f \in \mathscr{O}((\Lgrp{G})^n \sslash \hat{G})$, we have (see \cite[Proposition 11.7]{Laf18})
	\[ f\left(\sigma(\gamma_1), \ldots, \sigma(\gamma_n)\right) = \nu \circ \Theta_n(f)(\vec{\gamma}). \]
	\item Version 2: for all $n \in \Z_{\geq 1}$, $\vec{\gamma} = (\gamma_1, \ldots, \gamma_n)$ and $\tilde{f} \in \mathscr{O}\left(\hat{G} \bbslash (\Lgrp{G})^{\{0, \ldots, n\}} \sslash \hat{G}\right)$, we have
	\begin{equation}\label{eqn:sigma-characterization}
		\tilde{f}(\sigma(1), \sigma(\gamma_1), \ldots, \sigma(\gamma_n)) = \nu \left( S_{\{0, \ldots, n\}, \tilde{f}, (1, \vec{\gamma})} \right).
	\end{equation}
	The version 2 above is \textit{a priori} stronger, but they are actually equivalent by the preceding remarks on generators.
\end{itemize}

By the discussions preceding \cite[Remarque 12.21]{Laf18}, the map $\Hom_{\overline{\Q_\ell}\dcate{Alg}}(\mathcal{B}, \overline{\Q_\ell}) \to \Phi(G)$ above is injective. Hence we may write $\mathfrak{H}_\sigma = \mathfrak{H}_\nu$ if $\nu \mapsto \sigma \in \Phi(G)$, and set $\mathfrak{H}_\sigma = \{0\}$ if $\sigma$ does not match any $\nu$. This leads to the desired decomposition \eqref{eqn:global-decomp}.

\section{The transposes of excursion operators}\label{sec:transposes}
\subsection{On Verdier duality}\label{sec:Verdier}
Retain the notation of \S\ref{sec:local-models}. Among them, we recall only two points:
\begin{inparaenum}[(i)]
	\item the duality operator $\VerD$ is normalized with respect to $(X \smallsetminus \hat{N})^I$, and
	\item $W^{\vee,\theta}$ denotes the contragredient of $W \in \cate{Rep}_E((\Lgrp{G})^I)$ twisted by the Chevalley involution of $(\Lgrp{G})^I$.
\end{inparaenum}

The following results are recorded in \cite[Remarque 5.4]{Laf18}. For the benefit of the readers, we will give some more details below.

\begin{proposition}\label{prop:S-Verdier}
	There is a canonical isomorphism
	\[ \VerD \EuScript{S}^{(I_1, \ldots, I_k)}_{I,W,E} \rightiso \EuScript{S}^{(I_1, \ldots, I_k)}_{I,W^{\vee,\theta},E} \]
	between functors from $W \in \cate{Rep}_E((\Lgrp{G})^I)^{\mathrm{op}}$ to $\cate{Perv}_{G_{\sum_i \infty x_i}}\left(\Gra^{(I_1, \ldots, I_k)}_I\right)$.
\end{proposition}
\begin{proof}
	As noted in \cite[\S B.6]{BG02}, $\VerD$ preserves the ULA property with respect to $\Gra^{(I_1, \ldots, I_k)}_I \to (X \smallsetminus \hat{N})^I$. Since $\EuScript{S}^{(I_1, \ldots, I_k)}_{I,W,E}$ is ULA, the factorization structure on $\Gra^{(I_1, \ldots, I_k)}_I$ reduces the affairs to the case $|I|=1$, i.e.\ the Beilinson--Drinfeld affine Grassmannian used in \cite{MV07, Ric14, Zh15}. Consider its fiber $\Gra_x$ over some point $x \in |X \smallsetminus \hat{N}|$. In the notation from \S\ref{sec:geometric-setup}, there is a left $G_{\infty x}$-action on $\Gra_x$.

% Below: the original version, in which the isomorphism is not canonical... (W: irreducible here)
%	In this setting, $W = V_\omega$ is realized canonically via the total cohomology functor in the geometric Satake equivalence.
%
%	By \cite[(11.10)]{MV07}, sending an irreducible $W$ to $W^\vee$ corresponds, up to isomorphism, to $\mathcal{A} \mapsto \VerD(i^* \mathcal{A})$ on the category of shifted perverse sheaves on $G_{\infty x} \backslash \Gra_x$ (or: $G_{\infty x}$-equivariant perverse sheaves on $\Gra_x$); here $i$ is the involution of $G_{\infty x} \backslash \Gra_x$ induced by $g \mapsto g^{-1}$ on $G$. We claim that $\mathcal{A} \mapsto i^* \mathcal{A}$ corresponds to the $\Lgrp{\theta}$-twist $W \mapsto W^\theta$ in $\cate{Rep}_E(\hat{G})$. Granting this, $W \mapsto W^{\vee,\theta}$ will correspond to $\mathcal{A} \mapsto i^* \VerD i^* \mathcal{A} \simeq \VerD \mathcal{A}$.
%	
%	To prove the claim, notice that in the geometric Satake equivalence, cf.\ \cite[Lemma B.5]{Ric14}, the based root datum of $\hat{G}$ is reconstructed from the monoid of dominant weights of $\hat{G}$ partially ordered by dominance. In turn, these data are interpreted in terms of the stratification of $\Gra_x$ according to $G_{\infty x}$-orbits; see \cite[p.234]{Ric14}. Therefore, the effect of $i$ reflected on the based root datum of $\hat{G}$ is seen to be ``taking minus''. In parallel, the involution $\theta^{\text{AV}}$ on $\hat{G}$ in Remark \ref{rem:Adams-involution} has the same effect. Since $\theta$ and $\theta^\text{AV}$ differ only by a $\hat{G}(E)$-conjugation, our claim follows.

	In the local setting above, denote the usual duality operator $\cate{Perv}_{G_{\infty}}(\Gra_x)$ by $\VerD$: normalization is not an issue here. The main ingredients are
	\begin{compactitem}
		\item the Satake functor $\cate{Rep}_E((\Lgrp{G})^I) \to \cate{Perv}_{G_{\infty x}}(\Gra_x)$, written as $W \mapsto \EuScript{S}_{W, E}$;
		\item a canonical isomorphism between functors in $W$:
		\[ \VerD \EuScript{S}_{W, E} \rightiso \EuScript{S}_{W^{\vee,\theta}, E}. \]
	\end{compactitem}
	Granting these ingredients, for general $|I|$ we obtain canonical isomorphisms $\VerD \EuScript{S}^{(I_1, \ldots, I_k)}_{I,W,E} \rightiso \EuScript{S}^{(I_1, \ldots, I_k)}_{I,W^{\vee,\theta},E}$

	Let us explain the two ingredients in the local setting. The functor $W \mapsto \EuScript{S}_{W, E}$ is obtained in \cite{Ric14} or \cite[Theorem A.12]{Zh15}, which are based on the case over separably closed fields in \cite{MV07}. In order to explain the effect of $\Lgrp{\theta}$, we shall review the case over the separable closure $\Bbbk$ of $\mathring{F}_x$ first. The canonical isomorphism $\VerD \EuScript{S}_{W, E} \rightiso \EuScript{S}_{W^{\vee,\theta}, E}$ over $\Bbbk$ can be found in \cite[Lemma 14]{BF08}, for example, where a stronger equivariant version is established; they work over $\CC$, but the argument is largely formal.

	Next, apply Galois descent as explicated in \cite[\S 6]{Ric14} or \cite[Appendix]{Zh15}. Let $\mathcal{C} := \cate{Perv}_{G_{\infty x}}(\Gra_x)$ and set $\mathcal{C}'$ to be its avatar over $\Bbbk$. The absolute Galois group $\Upsilon$ of $\mathring{F}_x$ operates on $\mathcal{C}'$ via $\otimes$-equivalences, in a manner compatible with the fiber functor (total cohomology), thus $\Upsilon$ acts on the Tannakian group $\hat{G}$ as well. By \cite[p.237]{Ric14}, $\mathcal{C}$ is equivalent as an abelian category to ($\mathcal{C}'$ + continuous descent data under $\Upsilon$). The Satake equivalence over $\Bbbk$ and the machinery from \textit{loc.\ cit.} furnish an equivalence of $\otimes$-categories
	\[ \cate{Perv}_{G_{\infty x}}(\Gra_x) \to \cate{Rep}_E(\hat{G} \rtimes_{\mathrm{geom}} \Upsilon), \]
	where $c$ means continuity, and ``geom'' means the Tannakian or ``geometric'' $\Upsilon$-action on $\hat{G}$. See \cite[Remarque 1.19]{Laf18} for the choice of commutativity constraints.

	By \cite[Proposition A.6]{Zh15} or \cite[Corollary 6.8]{Ric14}, the geometric $\Upsilon$-action on $\hat{G}$ differs from the familiar ``algebraic'' one by the adjoint action via $\rho_{\hat{B}} \circ \chi_{\mathrm{cycl}}: \Upsilon \to \Z_\ell^\times \to \hat{G}^{\mathrm{ad}}(\Q_\ell)$, where $\chi_{\mathrm{cycl}}$ is the $\ell$-adic cyclotomic character and $\rho_{\hat{B}}$ is the half-sum of positive roots in $\hat{B}$; in particular, $\hat{G} \rtimes_{\mathrm{geom}} \Upsilon \simeq \hat{G} \rtimes_{\mathrm{alg}} \Upsilon$ (= absolute Galois form of the $L$-group) continuously. Since $\theta \in \Aut(\hat{G})$ stabilizes $\rho_{\hat{B}}$, the isomorphism matches $\theta \rtimes_{\mathrm{geom}} \identity$ with the Chevalley involution $\theta \rtimes_{\mathrm{alg}} \identity =: \Lgrp{\theta}$.
	
	All in all, we obtain the Satake functor $W \mapsto \EuScript{S}_{W, E}$ as well as the canonical isomorphisms $\VerD \EuScript{S}_{W, E} \rightiso \EuScript{S}_{W^{\vee, \theta}, E}$. This completes the proof.
\end{proof}

Note that the equivariance can be upgraded to $G^\mathrm{ad}_{\sum_i \infty x_i}$ or $G^\mathrm{ad}_{\sum_i n_i x_i}$ where $n_i \gg 0$ relative to $W$, see \cite[Remaruqe 1.20]{Laf18}.

\begin{proposition}\label{prop:F-Verdier}
	There is a canonical isomorphism
	\[ \VerD \EuScript{F}^{(I_1, \ldots, I_k)}_{N,I,W,\Xi,E} \rightiso \EuScript{F}^{(I_1, \ldots, I_k)}_{N,I,W^{\vee,\theta},\Xi,E} \]
	between functors from $W \in \cate{Rep}_E((\Lgrp{G})^I)^{\mathrm{op}}$ to $\cate{Perv}\left(\Cht^{(I_1, \ldots, I_k)}_{N,I,W}/\Xi\right)$ that is compatible with coalescence of paws.
\end{proposition}
\begin{proof}
	First, by combining \cite[Proposition 2.8]{Laf18} and the explanations before Corollaire 2.15 of \textit{loc.\ cit.}, the smooth morphisms
	\[ \Cht^{(I_1, \ldots, I_k)}_{N,I,W}/\Xi \xrightarrow{\epsilon^{(I_1, \ldots, I_k), \Xi}_{N,I,W,\underline{n}}} \Gra^{(I_1, \ldots, I_k)}_{I,W}/G^\text{ad}_{\sum_i n_i x_i}, \quad \Gra^{(I_1, \ldots, I_k)}_{I,W} \to \Gra^{(I_1, \ldots, I_k)}_{I,W}/G^\text{ad}_{\sum_i n_i x_i} \]
	have the same relative dimension; denote it by $d$.
	
	Proposition \ref{prop:S-Verdier} gives a functorial isomorphism between descent data of shifted perverse sheaves from $\Gra^{(I_1, \ldots, I_k)}_{I,W}$ to $\Gra^{(I_1, \ldots, I_k)}_{I,W}/G^\text{ad}_{\sum_i n_i x_i}$, abbreviated as $\VerD \EuScript{S}_1 \rightiso \EuScript{S}_2$. Denote the corresponding shifted perverse sheaves on $\Gra^{(I_1, \ldots, I_k)}_{I,W}/G^\text{ad}_{\sum_i n_i x_i}$ as $\EuScript{S}_1^\flat$ and $\EuScript{S}_2^\flat$. By standard results, cf.\ \cite[9.1.2]{LO2}, the isomorphism above descends to
	\[ (\VerD \EuScript{S}_1^\flat) [2d](d) \rightiso \EuScript{S}_2^\flat. \]
	Since $\EuScript{F}^{(I_1, \ldots, I_k)}_{N,I,W,\Xi,E}$ (resp.\ $\EuScript{F}^{(I_1, \ldots, I_k)}_{N,I,W^{\vee,\theta},\Xi,E}$) is defined in \eqref{eqn:F-sheaf} as $(\epsilon^{(I_1, \ldots, I_k), \Xi}_{N,I,W,\underline{n}})^* \EuScript{S}_1^\flat$ (resp.\ $(\epsilon^{(I_1, \ldots, I_k), \Xi}_{N,I,W^{\vee,\theta},\underline{n}})^* \EuScript{S}_2^\flat$), the assertion follows immediately by the same standard result.
\end{proof}

Take any partition $I = I_1 \sqcup \cdots \sqcup I_k$, truncation parameter $\mu$ and $W \in \cate{Rep}_E((\Lgrp{G})^I)$. As a consequence of Propositions \ref{prop:S-Verdier} and \ref{prop:F-Verdier}, we deduce that $\Gra^{(I_1, \ldots, I_k)}_{I, W} = \Gra^{(I_1, \ldots, I_k)}_{I, W^{\vee,\theta}}$ and $\Cht^{(I_1, \ldots, I_k), \leq \mu}_{N, I, W} = \Cht^{(I_1, \ldots, I_k), \leq \mu}_{N, I, W^{\vee,\theta}}$.

\begin{remark}\label{rem:cup-product}
	Below is a review of the cup product of $!$-pushforward. Let $S$ be a regular scheme and let $\mathfrak{p}: \mathcal{X} \to S$ be an algebraic stack of finite type over $S$. Let $\mathcal{L}$, $\mathcal{L}'$ be in $\cate{D}^-_c(\mathcal{X}, E)$. Our goal is to define a canonical arrow
	\[ \mathfrak{p}_! \mathcal{L} \Lotimes \mathfrak{p}_! \mathcal{L}' \to \mathfrak{p}_! \left( \mathcal{L} \Lotimes \mathcal{L}' \right). \]
	Denote by $\Delta$ (resp.\ $\mathfrak{p} \times \mathfrak{p}$) the diagonal morphism $\mathcal{X} \to \mathcal{X} \dtimes{S} \mathcal{X}$ (resp. $\mathcal{X} \dtimes{S} \mathcal{X} \to S$). The Künneth formula \cite[11.0.14 Theorem]{LO2} yields a canonical isomorphism in $\cate{D}^-_c(S)$
	\[ \mathfrak{p}_! \mathcal{L} \Lotimes \mathfrak{p}_! \mathcal{L}' \simeq (\mathfrak{p} \times \mathfrak{p})_! \left( \mathcal{L} \boxtimes \mathcal{L}' \right), \]
	where $\boxtimes$ denotes the external tensor product. Since $\mathcal{L} \Lotimes \mathcal{L}' = \Delta^* \left( \mathcal{L} \boxtimes \mathcal{L}' \right)$, to obtain the desired arrow, it remains to use the
	\[ (\mathfrak{p} \times \mathfrak{p})_! \to (\mathfrak{p} \times \mathfrak{p})_! \Delta_! \Delta^* = \mathfrak{p}_! \Delta^* \]
	arising from $\identity \to \Delta_* \Delta^* = \Delta_! \Delta^*$, as $\Delta$ is a closed immersion.
\end{remark}

Consider the normalized dualizing complex $\Omega$ on $\Cht^{(I_1, \ldots, I_k), \leq \mu}_{N, I, W} / \Xi$. The \emph{trace map}
\[ \mathrm{Tr}: (\mathfrak{p}^{(I_1, \ldots, I_k), \leq \mu}_{N, I})_! \Omega \xrightarrow{\mathrm{Tr}} E_{(X \smallsetminus \hat{N})^I} \]
in Verdier duality is obtained by adjunction from $\Omega \rightiso (\mathfrak{p}^{(I_1, \ldots, I_k), \leq \mu}_{N, I})^! E_{(X \smallsetminus \hat{N})^I}$.

On the other hand, Proposition \ref{prop:F-Verdier} affords a canonical arrow $\EuScript{F}^{(I_1, \ldots, I_k)}_{N, I, W^{\vee,\theta}, \Xi, E} \Lotimes \EuScript{F}^{(I_1, \ldots, I_k)}_{N, I, W, \Xi, E} \to \Omega$. Apply the cup-product construction to the stack $\Cht^{(I_1, \ldots, I_k), \leq \mu}_{N, I, W} / \Xi$ over $(X \smallsetminus \hat{N})^I$ to obtain canonical arrows in $\cate{D}^b_c((X \smallsetminus \hat{N})^I, E)$:
\begin{equation}\label{eqn:H-Verdier} \begin{tikzcd}[row sep=small, column sep=tiny]
		\EuScript{H}^{\leq \mu, E}_{N, I, W^{\vee,\theta}} \Lotimes \EuScript{H}^{\leq \mu, E}_{N, I, W} \arrow[equal, d] & E_{(X \smallsetminus \hat{N})^I}\\
		\left( \mathfrak{p}^{(I_1, \ldots, I_k), \leq \mu}_{N, I} \right)_! \left( \EuScript{F}^{(I_1, \ldots, I_k)}_{N, I, W^{\vee,\theta}, \Xi, E} \right) \Lotimes \left( \mathfrak{p}^{(I_1, \ldots, I_k), \leq \mu}_{N, I} \right)_! \left(\EuScript{F}^{(I_1, \ldots, I_k)}_{N, I, W, \Xi, E} \right) \arrow[d] & \\
		\left( \mathfrak{p}^{(I_1, \ldots, I_k), \leq \mu}_{N, I} \right)_! \left( \EuScript{F}^{(I_1, \ldots, I_k)}_{N, I, W^{\vee,\theta}, \Xi, E} \Lotimes \EuScript{F}^{(I_1, \ldots, I_k)}_{N, I, W, \Xi, E} \right) \arrow[r] & \left( \mathfrak{p}^{(I_1, \ldots, I_k), \leq \mu}_{N, I} \right)_! \Omega \arrow[uu, "\mathrm{Tr}"]
\end{tikzcd}\end{equation}

By homological common sense (see \cite[Ex. I.24 (ii)]{KS90} for example), taking $\mathrm{H}^\bullet$ in \eqref{eqn:H-Verdier} with respect to the ordinary $t$-structure on $(X \smallsetminus \hat{N})^I$ yield natural arrows between $E$-sheaves over $(X \smallsetminus \hat{N})^I$
\begin{equation*}
	\mathfrak{B}^{\Xi,E}_{N,I,W}: \EuScript{H}^{i, \leq\mu, E}_{N, I, W^{\vee,\theta}} \dotimes{E} \EuScript{H}^{-i, \leq\mu, E}_{N, I, W} \to E_{(X \smallsetminus \hat{N})^I}, \quad i \in \Z;
\end{equation*}
we will only use the case $i = 0$ in this article.

Following \cite[Remarque 9.2]{Laf18}, we may even pass to $\varinjlim_\mu$ and look at the stalk at $\xi \in \left\{ \overline{\eta^I}, \Delta(\overline{\eta})\right\}$, thereby obtain from $\mathfrak{B}^{\Xi, E}_{N, I, W}$ the $E$-bilinear pairings
\[ \lrangle{\cdot, \cdot}_\xi: \varinjlim_\mu \EuScript{H}^{0, \leq\mu, E}_{N, I, W^{\vee,\theta}} \big|_\xi \dotimes{E} \varinjlim_\mu \EuScript{H}^{0, \leq\mu, E}_{N, I, W} \big|_\xi \longrightarrow E. \]
Their relation with the arrow $\mathfrak{sp}$ of specialization is given by \cite[(9.6)]{Laf18}
\begin{equation}\label{eqn:pairing-sp}
	\lrangle{ \mathfrak{sp}^* h, \mathfrak{sp}^* h'}_{\overline{\eta^I}} = \lrangle{h, h'}_{\Delta(\overline{\eta})}.
\end{equation} 

In the discussions surrounding \eqref{eqn:H-vs-Cc}, we have seen that $\Cht_{N, \emptyset, \mathbf{1}}/\Xi$ is the constant stack $\Bun_{G,N}(\F_q)/\Xi$ over $\Spec \F_q$. The upshot is that, as in \cite[Remarque 9.2]{Laf18}, the pairing $\lrangle{\cdot, \cdot}_{\Delta(\overline{\eta})}$ for $I=\emptyset$ reduces to the integration pairing on $C_c(G(\mathring{F}) \backslash G(\A)/K_N \Xi; E)$, assuming $\mathrm{mes}(K_N)=1$. Upon restriction to Hecke-finite part, we get the pairing $\lrangle{\cdot, \cdot}$ for $H_{\emptyset, \mathbf{1}}$ in Remark \ref{rem:pairing-level}. The same holds for $H_{\{0\}, \mathbf{1}}$ by coalescence \eqref{eqn:H-vs-Ccusp}.

\subsection{Frobenius invariance}
For every morphism $f$ between reasonable schemes or stacks, we will denote by ``$\mathrm{can}$'' the canonical isomorphisms exchanging $f^* \leftrightarrow f^!$ and $f_* \leftrightarrow f_!$ under $\VerD$. When $f$ is a universal homeomorphism (resp.\ open immersion), we have $f_* = f_!$ and $f^* = f^!$ (resp.\ only $f^* = f^!$).

Merge the conventions from \S\ref{sec:Verdier} and \S\ref{sec:partial-Frobenius}. Let $J \subset I$ be finite sets, $I = I_1 \sqcup \cdots \sqcup I_k$ with $J = I_1$. We are going to explicate the compatibility between $\mathfrak{B}^{\Xi, E}_{N, I, W}$ and $\Frob_J^* \EuScript{H}^{\leq\mu, E}_{N,I,W} \xrightarrow{F_J} \EuScript{H}^{\leq\mu', E}_{N,I,W}$, i.e.\ \eqref{eqn:F_J}.

\begin{lemma}\label{prop:invariance-1}
	In $\cate{D}^b_c \left( \Cht^{(I_1, \ldots, I_k)}_{N,I,W}/\Xi, E \right)$, there is a commutative diagram whose arrows are all invertible
	\[\begin{tikzcd}[column sep=small]
		(\Frob^{(I_1, \ldots, I_k)}_{I_1, N, W^{\vee,\theta}})^* \EuScript{F}^{(I_2, \ldots, I_1)}_{N, I, W^{\vee,\theta}, E} \arrow[r] \arrow[d, "{F^{(I_1, \ldots, I_k)}_{I,N,W^{\theta,\vee}} }"'] & (\Frob^{(I_1, \ldots, I_k)}_{I_1, N, W})^* \VerD \EuScript{F}^{(I_2, \ldots, I_1)}_{N,I,W,\Xi,E} \arrow[r, "\mathrm{can}", "\sim"'] & \VerD(\Frob^{(I_1, \ldots, I_k)}_{I_1, N, I})^* \EuScript{F}^{(I_2, \ldots, I_1)}_{N,I,W,\Xi,E} \\
		\EuScript{F}^{(I_1, \ldots, I_k)}_{N, I, W^{\vee,\theta}, \Xi, E} \arrow[r] & \VerD \EuScript{F}^{(I_1, \ldots, I_k)}_{N,I,W,\Xi,E} \arrow[ru, "{\VerD F^{(I_1, \ldots, I_k)}_{I,N,W} }"'] &
	\end{tikzcd}\]
	where $F^{(I_1, \ldots, I_k)}_{I_1, N, \ldots}$ is from \eqref{eqn:partial-Frob-F}, and the horizontal arrows except $\mathrm{can}$ are induced by Proposition \ref{prop:F-Verdier}.
\end{lemma}
\begin{proof}
	We may assume $W = \boxtimes_{j=1}^k W_j$ is irreducible. Using the definition \eqref{eqn:F-sheaf}, the smoothness of $\epsilon^{(I_1, \ldots, I_k), \Xi}_{N, I, W, \underline{n}}$ as well as the ULA properties of $\EuScript{F}$ and $\EuScript{S}$, the desired commutativity eventually reduces to that of
	\[\begin{tikzcd}
		\Frob^* \EuScript{S}^{(I_j)}_{I_j, W_j^{\vee,\theta}, E} \arrow[r] \arrow[d, "{\Phi}"'] & \Frob^* \VerD \EuScript{S}^{(I_j)}_{I_j, W_j, E} \arrow[r, "\text{can}", "\sim"'] \arrow[d, "{\Phi}"] & \VerD \Frob^* \EuScript{S}^{(I_j)}_{I_j, W_j, E} \\
		\EuScript{S}^{(I_j)}_{I_j, W_j^{\vee,\theta}, E} \arrow[r] & \VerD \EuScript{S}^{(I_j)}_{I_j, W_j, E} \arrow[ru, "{\VerD \Phi}"'] & 
	\end{tikzcd}\]
	in $\cate{D}^b_c\left( \Gra^{(I_j)}_{I_j, W_j}/ G^\mathrm{ad}_{\sum_{i \in I_j} n_i x_i}, E \right)$, for each $1 \leq j \leq k$. Here $\Phi$ stands for the usual Frobenius correspondences, and the horizontal arrows except $\mathrm{can}$ are from Proposition \ref{prop:S-Verdier}. The square commutes by the functoriality of $\Phi$.
	
	As for the triangular part, \cite[4.8.2 Corollary]{LO1} says that $\mathrm{can}: \Frob^* \VerD \rightiso \VerD \Frob^*$ equals
	\[ \Frob^* \iHom(-, \Omega) \xrightarrow{\text{natural}} \iHom(\Frob^*(-), \Frob^* \Omega) \xrightarrow{f_*} \iHom(\Frob^*(-), \Omega), \]
	where $\Omega$ is the dualizing complex and $f: \Frob^* \Omega \rightiso \Omega$ is the canonical isomorphism furnished by \textit{loc.\ cit.} Both $f$ and ``$\mathrm{can}$'' reflect the fact that universal homeomorphisms conserve duality. In our case, that fact is also realized by transport of structure via Frobenius, i.e.\ we have $f = \Phi_\Omega$, the Frobenius correspondence for $\Omega$. The desired commutativity thus reduces to that of
	\[\begin{tikzcd}
		\Frob^* \iHom(\EuScript{S}, \Omega) \arrow[d, "{\Phi_{\iHom(\EuScript{S}, \Omega)}}"'] \arrow[r, "\text{natural}" inner sep=0.5em] & \iHom(\Frob^* \EuScript{S}, \Frob^* \Omega) \arrow[ld, "{(\Phi_S^{-1})^* \circ (\Phi_\Omega)_* }"] \\
		\iHom(\EuScript{S}, \Omega) & 
	\end{tikzcd}\]
	for all $\EuScript{S} \in \cate{D}^b_c\left( \Gra^{(I_j)}_{I_j, W_j}/ G^\mathrm{ad}_{\sum_{i \in I_j} n_i x_i}, E \right)$. This is by now standard.
\end{proof}

Re-introduce the truncation parameters $\mu' \gg \mu$ so that \eqref{eqn:mu-mu'} holds with respect to both $W$ and $W^{\vee,\theta}$. Let $a_1$ (universal homeomorphism) and $a_2$ (open immersion) be as in \eqref{eqn:corr-diagram}. As $\mu$ increases, $\Cht^{\cdots, \leq \mu}_{N, I, W}$ (resp.\ $\Cht^{\cdots, \leq \mu'}_{\cdots}$) form an open covering of $\Cht^{\cdots}_{N, I, W}$.

To state the next result, we write $\Omega^{\leq \mu}$ (resp.\ $\Omega$) for the normalized dualizing complex on $\Cht^{\cdots, \leq \mu}_{N, I, W}/\Xi$ (resp.\ on $a_1^{-1} \Cht^{(I_2, \ldots, I_1), \leq\mu}_{N,I,W}/\Xi$). Recall that dualizing complexes are unique up to unique isomorphisms \cite[3.4.5]{LO1}. There are canonical isomorphisms $a_1^* \Omega^{\leq \mu} \rightiso \Omega \leftiso a_2^! \Omega^{\leq \mu'}$, since $a_1^* = a_1^!$.

\begin{lemma}\label{prop:invariance-2}
	In $\cate{D}^b_c\left( a_1^{-1} \Cht^{(I_2, \ldots, I_1), \leq\mu}_{N,I,W}/\Xi, E \right)$, there is a commutative diagram
	\[\begin{tikzcd}[column sep=small, row sep=tiny]
		a_1^* \underbracket{ \EuScript{F}^{(I_2, \ldots, I_1)}_{N, I, W^{\vee,\theta}, \Xi, E} }_{\text{on $\leq \mu$}} \; \Lotimes \; a_1^* \underbracket{ \EuScript{F}^{(I_2, \ldots, I_1)}_{N, I, W, \Xi, E} }_{\text{on $\leq \mu$}} \arrow[dd, "{\Frob^{(I_1, \ldots, I_k)}_{I_1, N, W^{\vee,\theta}} \Lotimes \Frob^{(I_1, \ldots, I_k)}_{I_1, N, W} }"'] \arrow[r] & a_1^* \Omega^{\leq \mu} \arrow[rd, "\simeq"] \arrow[dd, "\simeq"] & \\
		& & \Omega \\
		a_2^! \underbracket{ \EuScript{F}^{(I_1, \ldots, I_k)}_{N, I, W^{\vee,\theta}, \Xi, E} }_{\text{on $\leq \mu'$}} \; \Lotimes \; a_2^! \underbracket{ \EuScript{F}^{(I_1, \ldots, I_k)}_{N,I,W,\Xi,E} }_{\text{on $\leq \mu'$}} \arrow[r] & a_2^! \Omega^{\leq \mu'} \arrow[ru, "\simeq"'] &
	\end{tikzcd}\]
	where the arrows from $\cdots \Lotimes \cdots$ to $\Omega$ are induced from Lemma \ref{prop:invariance-1}.
\end{lemma}
\begin{proof}
	It suffices to show the commutativity of the outer pentagon, since the triangle is defined to be commutative. Recall the passage from \eqref{eqn:partial-Frob-F} to \eqref{eqn:F-a1-a2}: $\Frob^{(I_1, \ldots, I_k)}_{I_1, N, \ldots}$ is obtained by restricting $F^{(I_1, \ldots, I_k)}_{I_1, N, \ldots}$ to the open substacks cut out by the conditions $\leq \mu$ and $\leq \mu'$. It remains to apply Lemma \ref{prop:invariance-1}; note that the effect of arrows $a_1^* \Omega^{\leq \mu} \rightiso \Omega \leftiso a_2^! \Omega^{\leq \mu'}$ match the morphism ``$\mathrm{can}$'' in Lemma \ref{prop:invariance-1}.
\end{proof}

\begin{proposition}\label{prop:Frob-invariance}
	Write $J := I_1$. There is a commutative diagram in $\cate{D}^b_c((X \smallsetminus \hat{N}), E)$
	\[\begin{tikzcd}
		\Frob_J^* \EuScript{H}^{\leq \mu, E}_{N, I, W^{\vee,\theta}} \Lotimes \Frob_J^* \EuScript{H}^{\leq \mu, E}_{N, I, W} \arrow[r, "{\Frob_J^* \text{\eqref{eqn:H-Verdier}} }" inner sep=0.6em] \arrow[d, "{F_J \Lotimes F_J}"'] & \Frob_J^* E_{(X \smallsetminus \hat{N})^I} \arrow[d, "F_J"] \\
		\EuScript{H}^{\leq \mu', E}_{N, I, W^{\vee,\theta}} \Lotimes \EuScript{H}^{\leq \mu', E}_{N, I, W} \arrow[r, "\text{\eqref{eqn:H-Verdier}}"' inner sep=0.6em] & E_{(X \smallsetminus \hat{N})^I}
	\end{tikzcd}\]
	where
	\begin{compactitem}
		\item $F_J \Lotimes F_J$ is induced from the $F_J$ in \eqref{eqn:F_J},
		\item the $F_J$ on the right is the evident partial Frobenius morphism for $E_{(X \smallsetminus \hat{N})^I}$.
	\end{compactitem}
\end{proposition}
\begin{proof}
	Retain the notation for Lemma \ref{prop:invariance-2} and let $\mathfrak{p} := \mathfrak{p}^{(I_1, \ldots, I_k)}_{N, I}$. Upon recalling the formalism of Künneth formula, cup products (Remark \ref{rem:cup-product}) and the trace maps $\mathrm{Tr}$ (see \eqref{eqn:H-Verdier}), Lemma \ref{prop:invariance-2} produce a diagram in $\cate{D}^b_c((X \smallsetminus \hat{N})^I, E)$:
	\[\begin{tikzcd}[column sep=small]
		\Frob_J^* \mathfrak{p}_! \underbracket{ \EuScript{F}^{(I_2, \ldots, I_1)}_{N, I, W^{\vee,\theta}, \Xi, E} }_{\text{on $\leq \mu$}} \Lotimes \Frob_J^* \mathfrak{p}_! \underbracket{ \EuScript{F}^{(I_2, \ldots, I_1)}_{N, I, W, \Xi, E} }_{\text{on $\leq \mu$}} \arrow[d, "{\mathrm{BC} \Lotimes \mathrm{BC}}"', "\simeq"] \arrow[r] & \Frob_J^* \mathfrak{p}_! \Omega^{\leq \mu} \arrow[r, "{\Frob_J^* \mathrm{Tr}}" inner sep=0.7em] \arrow[d, "\mathrm{BC}", "\simeq"'] & \Frob_J^* E_{(X \smallsetminus \hat{N})^I} \arrow[d, "\simeq"', "F_J"] \\
		\mathfrak{p}_! a_1^* \EuScript{F}^{(I_2, \ldots, I_1)}_{N, I, W^{\vee,\theta}, \Xi, E} \Lotimes \mathfrak{p}_! a_1^* \EuScript{F}^{(I_2, \ldots, I_1)}_{N, I, W, \Xi, E} \arrow[r] \arrow[d, "{ \mathfrak{p}_! \Frob^{(I_1, \ldots, I_k)}_{I_1, N, W^{\vee,\theta}} \;\Lotimes\; \mathfrak{p}_! \Frob^{(I_1, \ldots, I_k)}_{I_1, N, W} }"'] & \mathfrak{p}_! a_1^* \Omega^{\leq \mu} \simeq \mathfrak{p}_! \Omega \arrow[r, "\mathrm{Tr}"] & E_{(X \smallsetminus \hat{N})^I} \\ 
		\mathfrak{p}_! a_2^! \EuScript{F}^{(I_1, \ldots, I_k)}_{N, I, W^{\vee,\theta}, \Xi, E} \Lotimes \mathfrak{p}_! a_2^! \EuScript{F}^{(I_1, \ldots, I_k)}_{N, I, W, \Xi, E} \arrow[r] \arrow[d] & \mathfrak{p}_! a_2^! \Omega^{\leq \mu'} \simeq \mathfrak{p}_! \Omega \arrow[equal, u] \arrow[d] & \\
		\mathfrak{p}_! \underbracket{ \EuScript{F}^{(I_1, \ldots, I_k)}_{N, I, W^{\vee,\theta}, \Xi, E} }_{\text{on $\leq \mu'$}} \Lotimes \mathfrak{p}_! \underbracket{ \EuScript{F}^{(I_1, \ldots, I_k)}_{N, I, W, \Xi, E} }_{\text{on $\leq \mu'$}} \arrow[r] & \mathfrak{p}_! \Omega^{\leq \mu'} \arrow[ruu, bend right, start anchor=east, "\mathrm{Tr}"'] &
	\end{tikzcd}\]
	where $\mathrm{BC}$ (resp.\ $\mathfrak{p}_! a_2^! \to \mathfrak{p}_!$) is the arrow in \eqref{eqn:BC} (resp.\ explained after \eqref{eqn:F_J}). The diagram commutes: indeed,
	\begin{compactitem}
		\item the first two rows form a commutative diagram by the naturality of $\mathrm{BC}$, which is ultimately based on the topological invariance of the étale topos together with the fact that universal homeomorphisms respect duality \cite[9.1.5 Proposition and 12.2]{LO2};
		\item the commutativity of the middle square comes from Lemma \ref{prop:invariance-2}, by applying $\mathfrak{p}_!$;
		\item the remaining pieces commute by the naturality of $\mathfrak{p}_! a_2^! \to \mathfrak{p}_!$ and of $\mathrm{Tr}$.
	\end{compactitem}

	The composite of the last row is \eqref{eqn:H-Verdier}, and that of the first row is its $\Frob_J^*$-image (now for the $\leq \mu$ part). The composite of the leftmost column yields $F_J \Lotimes F_J: \Frob_J^* \EuScript{H}^{\leq \mu, E}_{N, I, W^{\vee,\theta}} \Lotimes \Frob_J^* \EuScript{H}^{\leq \mu, E}_{N, I, W} \to \EuScript{H}^{\leq \mu', E}_{N, I, W^{\vee,\theta}} \Lotimes \EuScript{H}^{\leq \mu', E}_{N, I, W}$ by the very definition of $F_J$. This completes the proof.
\end{proof}

The case $k = 1$, i.e.\ when $F_J$ is the total Frobenius morphism, is relatively straightforward. Cf.\ the proof of Lemma \ref{prop:invariance-1}.

Recall from \S\ref{sec:partial-Frobenius} that $F_J$ furnishes an $E$-linear endomorphism of $\varinjlim_\mu \EuScript{H}^{0, \leq\mu, E}_{N,I,W} \bigg|_{\overline{\eta^I}}$, still denoted as $F_J$.

\begin{corollary}\label{prop:F_J-invariance}
	The pairing $\lrangle{\cdot , \cdot}_{\overline{\eta^I}}$ in \eqref{eqn:pairing-sp} is invariant under $F_J$ for all $J \subset I$.
\end{corollary}
\begin{proof}
	After taking $\mathrm{H}^0$ and $\varinjlim_\mu$, Proposition \ref{prop:Frob-invariance} implies that
	\[ \lrangle{h_1, h_2}_{\overline{\eta^I}} = \lrangle{F_J(h_1), F_J(h_2)}_{\overline{\eta^I}} \]
	for all $h_1, h_2$ in $\varinjlim_\mu \EuScript{H}^{0, \leq \mu, E}_{N, I, W^{\vee,\theta}} \big|_{\overline{\eta^I}}$ and $\varinjlim_\mu \EuScript{H}^{0, \leq \mu, E}_{N, I, W} \big|_{\overline{\eta^I}}$, respectively.
\end{proof}

The cautious reader might worry about a missing power of $p$ in Corollary \ref{prop:F_J-invariance} due to Tate twists. It does not occur here by our normalizations of $\EuScript{S}$, $\EuScript{F}$ and $\VerD$.

\subsection{Computation of the transpose}\label{sec:computation-transpose}
We adopt the notation of \S\ref{sec:excursion-operator}. The integration pairing $\lrangle{\cdot, \cdot}$ of Remark \ref{rem:pairing-level} is non-degenerate symmetric on the finite-dimensional $E$-vector space $H_{\{0\}, \mathbf{1}} \simeq H_{\emptyset, \mathbf{1}}$. The \emph{transpose} $S^*$ of any $S \in \End_E\left(H_{\{0\}, \mathbf{1}}\right)$, characterized by $\lrangle{h', Sh} = \lrangle{S^* h', h}$ for all $h,h' \in H_{\{0\}, \mathbf{1}}$. Identify $\lrangle{\cdot, \cdot}$ with the pairing $\lrangle{\cdot, \cdot}_{\Delta(\overline{\eta})}$ in \eqref{eqn:pairing-sp}.

\begin{lemma}\label{prop:transpose-0}
	For all data $I,W,\xi,x$ and $\vec{\gamma} = (\gamma_i)_i \in \pi_1(\eta, \overline{\eta})^I$ for excursion operators, we have
	\[ S_{I,W,\xi,x,\vec{\gamma}}^* = S_{I, W^{\vee,\theta}, x, \xi, \vec{\gamma}^{-1}}. \]
	In particular, the $E$-algebra $\mathcal{B}_E$ is closed under transpose $S \mapsto S^*$.
\end{lemma}
Note that the roles of $x,\xi$ are switched when one passes from $W$ to $W^{\vee,\theta}$. The transpose-invariance of $\mathcal{B}_E$ has already been sketched in \cite[Remarque 12.15]{Laf18}.
\begin{proof}
	Recall from \S\ref{sec:partial-Frobenius} that by choosing $\overline{\eta^I}$ and $\mathfrak{sp}$, there is a homomorphism $\mathrm{FWeil}(\eta^I, \overline{\eta^I}) \to \Weil{\mathring{F}}^I$ inducing an isomorphism between pro-finite completions. As $S_{I,W,\xi,x,\vec{\gamma}}$ and $S_{I, W^{\vee,\theta}, x, \xi, \vec{\gamma}^{-1}}$ are both continuous in $\vec{\gamma}$, it suffices to consider that case when $\vec{\gamma}$ comes from $\mathrm{FWeil}(\eta^I, \overline{\eta^I})$.

	By \cite[Remarque 5.4, (9.8)]{Laf18}, $\EuScript{C}^\flat_\xi$ and $\EuScript{C}^\sharp_\xi$ are already transposes of each other on the sheaf level with respect to $\mathfrak{B}^{\Xi,E}_{N,I,W}$; in particular they commute with $\mathfrak{sp}^*$. Ditto for $\EuScript{C}^\sharp_x$ and $\EuScript{C}^\flat_x$. Therefore, for all $h,h' \in H_{\{0\}, \mathbf{1}}$, we infer by using \eqref{eqn:pairing-sp} that
	\begin{align*}
		\lrangle{h', S_{I,W,x,\xi,\vec{\gamma}} (h) }_{\Delta(\overline{\eta})} & = \lrangle{h', \EuScript{C}^\flat_\xi (\mathfrak{sp}^*)^{-1} ( \vec{\gamma} \cdot \mathfrak{sp}^* \EuScript{C}^\sharp_x  h) }_{\Delta(\overline{\eta})} \\
		& = \lrangle{ \mathfrak{sp}^*(h'), \mathfrak{sp}^* \left( \EuScript{C}^\flat_\xi (\mathfrak{sp}^*)^{-1} ( \vec{\gamma} \cdot \mathfrak{sp}^* \EuScript{C}^\sharp_x h ) \right)  }_{\overline{\eta^I}} \\
		& = \lrangle{ \mathfrak{sp}^* \EuScript{C}^\sharp_\xi (h'), \vec{\gamma} \cdot \mathfrak{sp}^* \EuScript{C}^\sharp_x(h) }_{\overline{\eta}}, \\
		\lrangle{ S_{I, W^{\vee,\theta}, \xi, x, \vec{\gamma}^{-1}}(h'), h}_{\Delta(\overline{\eta})} & = \lrangle{ \EuScript{C}^\flat_x (\mathfrak{sp}^*)^{-1} (\vec{\gamma}^{-1} \cdot \mathfrak{sp}^* \EuScript{C}^\sharp_\xi h'), h }_{\Delta(\overline{\eta})} \\
		& = \lrangle{ \mathfrak{sp}^* \left( \EuScript{C}^\flat_x (\mathfrak{sp}^*)^{-1} (\vec{\gamma}^{-1} \cdot \mathfrak{sp}^* \EuScript{C}^\sharp_\xi h') \right), \mathfrak{sp}^*(h) }_{\overline{\eta^I}} \\
		& = \lrangle{ \vec{\gamma}^{-1} \cdot \mathfrak{sp}^* \EuScript{C}^\sharp_\xi (h'), \mathfrak{sp}^* \EuScript{C}^\sharp_x (h) }_{\overline{\eta^I}}.
	\end{align*}
	It remains to show that $\lrangle{\cdot, \cdot}_{\overline{\eta^I}}$ is $\mathrm{FWeil}(\eta^I, \overline{\eta^I})$-invariant. Recall that the $\mathrm{FWeil}(\eta^I, \overline{\eta^I})$-action unites those from $\pi_1(\eta^I, \overline{\eta^I})$ and partial Frobenius morphisms $F_J$. The $\pi_1(\eta^I, \overline{\eta^I})$-action leaves $\lrangle{\cdot, \cdot}_{\overline{\eta^I}}$ invariant since the latter comes from the sheaf-level pairing $\mathfrak{B}^{\Xi,E}_{N,I,W}$ over $(X \smallsetminus \hat{N})^I$. The $F_J$-invariance of $\lrangle{\cdot, \cdot}$ for all $J \subset I$ is assured by Corollary \ref{prop:F_J-invariance}.
\end{proof}

We are now able to describe the transpose of excursion operators.

\begin{definition}\label{def:dagger}
	For every $f \in \mathscr{O}(\hat{G} \bbslash (\Lgrp{G})^I \sslash \hat{G})$, set $f^\dagger(\vec{g}) := f(\Lgrp{\theta}(\vec{g}^{-1}))$ where $\vec{g} \in (\Lgrp{G})^I$ and $\Lgrp{\theta}$ stands for the Chevalley involution of $(\Lgrp{G})^I$. Then $f \mapsto f^\dagger$ defines an involution of $\mathscr{O}(\hat{G} \bbslash (\Lgrp{G})^I \sslash \hat{G})$.
% Below are now unnecessary:
%	Concurrently, $f \mapsto f^\dagger(\vec{g}) := f(\Lgrp{\theta}(\vec{g}^{-1}))$ induces an involution of $\mathscr{O}((\Lgrp{G})^n \sslash \hat{G})$, for every $n \in \Z_{\geq 1}$.
\end{definition}

\begin{lemma}\label{prop:transpose-1}
	For all $I$, $f \in \mathscr{O}(\hat{G} \bbslash (\Lgrp{G})^I \sslash \hat{G})$ and $\vec{\gamma} = (\gamma_i)_i \in \pi_1(\eta, \overline{\eta})^I$, we have $S_{I,f,\vec{\gamma}}^* = S_{I, f^\dagger, \vec{\gamma}^{-1}}$. 
\end{lemma}
\begin{proof}
	It suffices to consider the case $f(\vec{g}) = \lrangle{\xi, \vec{g} \cdot x}_{W^\vee \otimes W}$, where $\xi \in W^\vee$, $x \in W$ are as in Lemma \ref{prop:transpose-0}, and $\lrangle{\cdot, \cdot}_{W^\vee \otimes W}$ is the evident duality pairing. As before, denote by $W^\theta$, $W^{\vee,\theta}$ be the $\Lgrp{\theta}$-twists of the representations $W, W^\vee$ etc., and the preceding convention on pairing still applies. Note that $(W^{\vee,\theta})^\vee \simeq W^\theta$ canonically in $\cate{Rep}_E((\Lgrp{G})^I)$.

	For every $\vec{g} \in (\Lgrp{G})^I$, we have
	\begin{align*}
		f^\dagger(\vec{g}) & = \lrangle{\xi, \underbracket{\Lgrp{\theta}(g)^{-1} \cdot x}_{\text{original action}} }_{W^\vee \otimes W} \\
		& = \lrangle{\xi, \underbracket{\vec{g}^{-1} \cdot x}_{\Lgrp{\theta}\text{-twisted}} }_{W^{\theta, \vee} \otimes W^\theta} = \lrangle{\underbracket{\vec{g} \cdot \xi}_{\Lgrp{\theta}\text{-twisted}} , x}_{W^{\theta, \vee} \otimes W^\theta} \\
		& = \lrangle{ x, \vec{g} \cdot \xi }_{(W^{\vee,\theta})^\vee \otimes W^{\vee,\theta}}.
	\end{align*}

	In view of Lemma \ref{prop:transpose-0}, we deduce that $S_{I,f,\vec{\gamma}}^* = S_{I,W,\xi,x,\vec{\gamma}}^*$ equals $S_{I, W^{\vee,\theta}, x, \xi, \vec{\gamma}^{-1}} = S_{I, f^\dagger, \vec{\gamma}^{-1}}$, as asserted.
\end{proof}

Consider any homomorphism $\nu: \mathcal{B} := \mathcal{B}_E \otimes_E \overline{\Q_\ell} \to \overline{\Q_\ell}$ of $\overline{\Q_\ell}$-algebras. As $\mathcal{B}$ is commutative and closed under transpose, $\nu^*: S \mapsto \nu(S^*)$ is also a homomorphism of $\overline{\Q_\ell}$-algebras.

\begin{proposition}\label{prop:transpose-parameter}
	If $\sigma \in \Phi(G)$ is attached to $\nu: \mathcal{B} \to \overline{\Q_\ell}$, then $\Lgrp{\theta} \circ \sigma$ is attached to $\nu^*$.
\end{proposition}
\begin{proof}
%	Given the characterization \eqref{eqn:sigma-characterization} of the $L$-parameters attached to $\nu, \nu^*$, it boils down to the easy fact that
%	\[\begin{tikzcd}[row sep=small]
%		\nu^* \Theta_n(f)(\vec{\gamma}) \arrow[equal, d, "\because\;\text{Proposition \ref{prop:transpose}}" inner sep=1em] & & \\
%		\nu \Theta_n(f^\dagger)(\vec{\gamma}^{-1}) \arrow[equal, r] & f^\dagger(\sigma(\gamma_1)^{-1}, \ldots, \sigma(\gamma_n)^{-1}) \arrow[equal, r] & f\left( \Lgrp{\theta} \sigma(\gamma_1), \ldots, \Lgrp{\theta} \sigma(\gamma_n) \right)
%	\end{tikzcd}\]
%	for all $n \in \Z_{\geq 1}$, $\vec{\gamma} = (\gamma_1, \ldots, \gamma_n) \in \pi_1(\eta, \overline{\eta})^n$ and $f \in \mathscr{O}((\Lgrp{G})^n \sslash \hat{G})$.
	Fix $n \in \Z_{\geq 0}$ and let $I := \{0, \ldots, n\}$. Given the characterization \eqref{eqn:sigma-characterization} of the $L$-parameters attached to $\nu, \nu^*$, it boils down to the observation that for all $\vec{\gamma} = (\gamma_0, \ldots, \gamma_n) \in \pi_1(\eta, \overline{\eta})^I$ and $f \in \mathscr{O}(\hat{G} \bbslash \Lgrp{G}^I \sslash \hat{G})$,
	\begin{multline*}
			\nu^* \left( S_{I, f , \vec{\gamma}} \right) = \nu\left( S_{I, f , \vec{\gamma}}^* \right) = \nu \left( S_{I, f^\dagger, \vec{\gamma}^{-1}} \right) \\
			= f^\dagger\left(\sigma(\gamma_0)^{-1}, \ldots, \sigma(\gamma_n)^{-1}\right) = f\left(\Lgrp{\theta} \sigma (\gamma_0), \ldots, \Lgrp{\theta}\sigma(\gamma_n) \right),
	\end{multline*}
	in which the second equality stems from Proposition \ref{prop:transpose-1}.
\end{proof}

Write $\mathfrak{H}_\sigma := \mathfrak{H}_\nu$ if $\sigma \in \Phi(G)$ is attached to $\nu$, and write $\lrangle{\cdot, \cdot}_{\sigma, \sigma'} := \lrangle{\cdot, \cdot} \big|_{\mathfrak{H}_\sigma \otimes \mathfrak{H}_{\sigma'}}$, for all $\sigma, \sigma' \in \Phi(G)$.

\begin{proof}[Proof of Theorem \ref{prop:global-contragredient}]
	Enlarge $E$ so that all homomorphisms $\nu: \mathcal{B} \to \overline{\Q_\ell}$ are defined over $E$. Fix a $\nu$ such that $\mathfrak{H}_\nu \neq \{0\}$. Since $\mathcal{B}_E$ is closed under transpose, the subspace $\mathfrak{H}_\nu^\perp \subset H_{\{0\}, \mathbf{1}}$ defined relative to $\lrangle{\cdot, \cdot}$ is $\mathcal{B}_E$-stable as well. Since $\lrangle{\cdot, \cdot}$ is non-degenerate, $\mathfrak{H}_\nu^\perp \neq H_{\{0\}, \mathbf{1}}$. Use the $\mathcal{B}_E$-invariance to decompose $\mathcal{B}_E$-modules as follows
	\[ \mathfrak{H}_\nu^\perp = \bigoplus_\mu \mathfrak{H}_\nu^\perp \cap \mathfrak{H}_\mu , \quad \frac{H_{\{0\}, \mathbf{1}}}{\mathfrak{H}_\nu^\perp} = \bigoplus_\mu \frac{\mathfrak{H}_\mu}{\mathfrak{H}_\nu^\perp \cap \mathfrak{H}_\mu} \neq \{0\}. \]
	We contend that $\mathfrak{H}_\mu \not\subset \mathfrak{H}_\nu^\perp$ only if $\mu = \nu^*$, or equivalently $\mu^* = \nu$.
	
	Indeed, $\lrangle{\cdot, \cdot}$ induces a non-degenerate pairing
	\[ \lrangle{\cdot, \cdot}_\nu: \mathfrak{H}_\nu \dotimes{E} \bigoplus_\mu \frac{\mathfrak{H}_\mu}{\mathfrak{H}_\mu \cap \mathfrak{H}_\nu^\perp} \to E. \]
	Let $d := \dim H_{\{0\}, \mathbf{1}}$. For every $S \in \mathcal{B}_E$, write $S_\nu := S|_{\mathfrak{H}_\nu}$. Then $(S_\nu - \nu(S))^d = 0$. Taking transpose with respect to $\lrangle{\cdot, \cdot}_\nu$ yields $(S_\nu^* - \nu(S))^d = 0$ in $\End_E\left(\frac{\mathfrak{H}_\mu}{\mathfrak{H}_\mu \cap \mathfrak{H}_\nu^\perp}\right)$, for each $\mu$.
	
	On the other hand, the transpose $S^* \in \mathcal{B}_E$ with respect to $\lrangle{\cdot, \cdot}$ satisfies $(S^* - \mu(S^*))^d = 0$ on $\mathfrak{H}_\mu$, and $S^*|_{\mathfrak{H}_\mu}$ induces $S^{*,\mu} \in \End_E\left(\frac{\mathfrak{H}_\mu}{\mathfrak{H}_\mu \cap \mathfrak{H}_\nu^\perp}\right)$ satisfying $(S^{*,\mu} - \mu(S^*))^d = 0$. Clearly $S^{*,\mu} = S_\nu^*$. All in all, we deduce that $\mu^*(S) := \mu(S^*) = \nu(S)$ whenever $\mathfrak{H}_\mu \neq \mathfrak{H}_\nu^\perp \cap \mathfrak{H}_\mu$.

	It follows from the claim that if $\lrangle{\cdot, \cdot}_{\sigma, \sigma'}$ is not identically zero, then the corresponding $\nu, \nu': \mathcal{B} \to \overline{\Q_\ell}$ satisfy $\nu^* = \nu'$. Now Proposition \ref{prop:transpose-parameter} implies $\Lgrp{\theta} \circ \sigma = \sigma'$.
\end{proof}

\bibliographystyle{abbrv}	% For bibtex - preferred on arXiv
\bibliography{Contragredient}

% Below: for Biblatex...
%\printbibliography[heading=bibintoc]

% Below: for bibtex
% \bibliographystyle{abbrv}
% \bibliography{Contragredient}

\vspace{1em}
\begin{flushleft}
	Wen-Wei Li \\
	E-mail address: \href{mailto:wwli@bicmr.pku.edu.cn}{\texttt{wwli@bicmr.pku.edu.cn}} \\
	Beijing International Center for Mathematical Research, Peking University \\
	No.\ 5 Yiheyuan Road, Beijing 100871, People's Republic of China.
\end{flushleft}

\end{document}